\documentclass[a4paper,reqno]{amsart}
\usepackage{amsmath}
\usepackage{amssymb}
\usepackage{mathrsfs}
\usepackage{cite}
\usepackage{color}
\usepackage{hyperref}

\allowdisplaybreaks  

\theoremstyle{plain}
\newtheorem{theorem}{\bf Theorem}[section]
\newtheorem{proposition}[theorem]{\bf Proposition}
\newtheorem{lemma}[theorem]{\bf Lemma}

\newtheorem{conjecture}[theorem]{\bf Conjecture}

\theoremstyle{definition}

\newtheorem{remark}[theorem]{\bf Remark}

\newcommand{\R}{\mathbb R}
\newcommand{\Z}{\mathbb Z}

\DeclareMathOperator{\Div}{div}

\numberwithin{equation}{section}

\begin{document}
\title[]{Solutions with clustering concentration layers to the Ambrosetti-Prodi type problem}

\author[]{Qiang Ren$^{\dag}$}
\thanks{$^\dag$ Q. Ren is supported by Inner Mongolia Autonomous Region Natural Science Foundation(no. 2024LHMS01001) and NSF of China(no. 12001298)}

\address[Qiang Ren]{School of Mathematical Sciences and Inner Mongolia Key Laboratory of Mathematical Modeling and Scientific Computing, Yuquan West Campus, Inner Mongolia University, Hohhot 010030, P.R. China} \email{renq@imu.edu.cn}

\date{Submitted in \today.}



\begin{abstract}
We consider the following Ambrosetti-Prodi type problem
\begin{equation}\label{e50}
\left\{\begin{array}{ll}
-\Div (A(x)\nabla u)=|u|^p-\tilde{t}\mathbf{\Psi}(x), &\mbox{in $\Omega$,} \\
u=0, & \mbox{on $\partial \Omega$},
\end{array}
\right.
\end{equation}
where $\Omega \subset \R^2$, $\tilde{t}>0$, $p>3$ and  $\mathbf{\Psi}$ is an eigenfunction corresponding to the first eigenvalue of the following operator
\[\mathfrak{L}(u)=-\Div (A(x)\nabla u).\]
Moreover, $A(x)=\{A_{ij}(x)\}_{2\times 2}$ is a symmetric positive definite matrix-valued function.
Let $\Gamma \subset \Omega$ be a closed curve and also  a non-degenerate critical point of the functional
\[\mathcal{K}(\Gamma)=\int_\Gamma \mathbf{\Psi}^{\frac{p+3}{2p}}dvol_{\mathfrak{g}},\]
where $\mathfrak{g}(X,Y)=\langle A^*X,Y\rangle$ denotes a Riemannian metric on $\R^2$ and $A^*$ is the adjoint matrix for $A$.
We prove that there exists a sequence of $\tilde{t}=\tilde{t}_l\to +\infty$ such that problem \eqref{e50} admits solutions $u_{\tilde{t}_l}$ with clustering concentration layers directed along $\Gamma$.
\end{abstract}
\maketitle

\section{Introduction}
The Ambrosetti-Prodi problem is formulated as follows:
\begin{equation}\label{e51}
\left\{\begin{array}{ll}
-\Delta u=\zeta(u)-\tilde{t}\mathbf{\Psi}_0 (x), &\mbox{in $\Omega$,} \\
u=0,&\mbox{on $\partial \Omega$},
\end{array}\right.
\end{equation}
where $\tilde{t}>0$, $\Omega\subset \R^\mathcal{N}$ is a smooth bounded domain,  $\mathbf{\Psi}_0$ is  an eigenfunction of $-\Delta$ under the Dirichlet boundary condition corresponding to the first eigenvalue, the function $\zeta(t)$ satisfies
\[-\infty\le\mu=\lim_{t\to -\infty} \frac{\zeta(t)}{t}<\lim_{t\to +\infty} \frac{\zeta(t)}{t}=\nu\le +\infty,\]
and the interval $(\mu, \nu)$ contains some eigenvalues of $-\Delta$ with Dirichlet boundary condition.

Problem \eqref{e51} was first investigated by Ambrosetti and Prodi \cite{Ambrosetti_Prodi1972}, and it has been extensively studied since the 1980s( see, e.g., \cite{deFigueiredo1984,deFigueiredo_Solimini1984,Lazer_McKenna1981,Lazer_McKenna1984,Lazer_McKenna1988,Ruf_Srikanth986a,Ruf_Srikanth1986,Solimini1985,Ruf_Solimini1986}). The core results in these works show that if $\zeta(t)$ grows subcritical at infinity, problem \eqref{e51} has at least two solutions: one is local minimizer of the Euler-Lagrange functional, the other is the mountain-pass solution. Breuer, McKenna and Plum \cite{Breuer_McKenna_Plum2003} considered the case where $\zeta(t)=t^2$ and $\Omega$ is a unit square in $\R^2$. Through a computer assisted proof, they demonstrated  that \eqref{e51} has at least 4 solutions.  By comparing the Morse index of mountain pass solution in different spaces, de Figueiredo, Srikanth and Santra \cite{deFigueiredo_Srikanth_Santra2005}  obtained a non-radial solution of \eqref{e51}, where $\Omega$ is a unit ball and $\zeta(t)=t^2$.

When $\zeta(t)=|t|^p$ with $1<p<\frac{\mathcal{N}+2}{\mathcal{N}-2}$ for $\mathcal{N}\ge3$ and $p>1$ for $\mathcal{N}=2$, problem \eqref{e51} reduces to
\begin{equation}\label{e52}
\left\{\begin{array}{ll}
-\Delta u=|u|^p-\tilde{t}\mathbf{\Psi}_0 (x), &\mbox{in $\Omega$,} \\
u=0,&\mbox{on $\partial \Omega$}.
\end{array}\right.
\end{equation}
Dancer and Yan \cite{Dancer_Yan2005} constructed arbitrary many peak solutions of \eqref{e52} for $\tilde{t}>0$ large enough.  This result verifies that Lazer-McKenna conjecture \cite{Lazer_McKenna1981} is valid in this setting. Dancer and Yan \cite{Dancer_Yan2005} also proved that the mountain pass solution of \eqref{e51} has a sharp peak near the boundary as $\tilde{t}\to +\infty$. As a consequence,   \eqref{e51} admits solutions concentrating at some points in $\Omega$ or the  ones on $\partial \Omega$ as $\tilde{t}\to +\infty$. These results have been extended to various  nonlinearities( see \cite{Dancer_Yan2005a, Dancer_Santra2007, delPino_Munoz2006, Li_Yan_Yang2007, Li_Yan_Yang2006, Molle_Passaseo2010, Wei_Yan2007} for instance and references therein).

However, these studies above only focus on point-concentrating solutions. Based on numerical evidence, Hollman and McKenna \cite{Hollman_McKenna2011} posed the question: do there exist other types of concentration behaviors  for the solutions of \eqref{e52} as $\tilde{t}\to+\infty$. Addressing this, Manna and Santra \cite{Manna_Santra2016} considered \eqref{e52} with $\Omega\subset \R^2$ and $p>2$, proving that \eqref{e52} has a family of solutions clustering along a closed curve $\Gamma\subset \Omega$, where $\Gamma$ is a non-degenerate critical point of the functional
\begin{equation}\label{e54}
\mathcal{K}_0(\Gamma)=\int_{\Gamma} \mathbf{\Psi}_0^{\frac{p+3}{2p}}(x)dvol.
\end{equation}
Subsequently, Khemiri, Mahmoudi and Messaoudi  \cite{Khemiri_Mahmoudi_Messaoudi2017} extended this result to higher dimensions. For $\mathcal{N}\ge 3$ and $\Omega$ containing a $k$-dimensional compact submanifold $\Gamma$( a non-degenerate critical point of the functional
\[\mathcal{K}_1(\Gamma)=\int_{\Gamma} \mathbf{\Psi}_0^{\left(1-\frac1p\right)\left(\frac{p+1}{p-1}-\frac{n-k}2\right)}(x)dvol),\]
they showed that there exists a sequence $\tilde{t}=\tilde{t}_j\to +\infty$ such that \eqref{e52} has solutions  $u_{\tilde{t}_j}$ with  concentration layers  near $\Gamma$.

Baraket \textit{et. al.} \cite{Baraket_Khemiri_Mahmoudi_Messaoudi2018} considered the following Neumann problem:
\begin{equation}\label{e53}
\left\{\begin{array}{ll}
-\Delta u=|u|^p-\tilde{t}\psi (x), &\mbox{in $\Omega$,} \\
\frac{\partial u}{\partial \mathbf{n}}=0,&\mbox{on $\partial \Omega$},
\end{array}\right.
\end{equation}
where  $\Omega\subset\R^\mathcal{N}$ is a smooth bound domain and $\mathbf{n}$ is the unit outward normal vector of $\partial\Omega$.
For $\mathcal{N}=2$, they proved that \eqref{e53} has a solution $u_{\tilde{t}}$ concentrating along a curve $\Gamma\subset \bar{\Omega}$ when $\tilde{t}>0$ is sufficiently large, with  $\Gamma$ being a non-degenerate critical point of the functional $\mathcal{K}_2(\Gamma)=\int_{\Gamma} \psi^{\frac{p+3}{2p}}(x)dvol$. The curve $\Gamma$ intersects $\partial\Omega$ at a right angle and divides $\Omega$ into two parts. Additionally, for $\psi\equiv1$ and  $\mathcal{N}\ge 2$, Bendahou, Khemiri and Mahmoudi \cite{Bendahou_Khemiri_Mahmoudi2020} constructed a new family of solutions to \eqref{e53} with numerous spikes concentrating along an interior straight line in $\Omega$ as $t\to+\infty$. Using a similar approach, Ao, Fu and Liu \cite{Ao_Fu_Liu2022} constructed solutions of \eqref{e53} concentrating along a boundary segment on $\partial \Omega$ in the two dimensional case.

These results reveal that the high-dimensional concentration behavior of Ambrosetti-Prodi type problems resembles that of nonlinear Schr\"{o}dinger equation:
\begin{equation}\label{e55}
-\varepsilon^2 \Delta u+V(y)u=u^p, \qquad \mathrm{in} \qquad \R^\mathcal{N}.
\end{equation}
In 2003, Ambrosetti, Malchiodi and Ni \cite{Ambrosetti_Malchiodi_Ni2003} proposed the following conjecture:
\begin{conjecture}
Let $\Gamma$ be a $k$-dimensional submanifold in $\R^\mathcal{N}$ and a nondegenerate critical point of the following functional
\[\mathcal{K}(\Gamma)=\int_\Gamma V^{\frac{p+1}{p-1}-\frac12(\mathcal{N}-k)}dvol,\]
where $1<p<\frac{\mathcal{N}+2-k}{\mathcal{N}-2-k}$ for $\mathcal{N}-k\ge 3$ and $p>1$ for $\mathcal{N}-k=2$.
Then there exists a family of solutions to \eqref{e55} concentrating along $\Gamma$ at least for a subsequence $\varepsilon=\varepsilon_j\to 0$.
\end{conjecture}
del Pino, Kowalczyk and Wei \cite{delPino_Kowalczyk_Wei2007} first confirmed this conjecture for $\mathcal{N}=2$ and $k=1$.  Wang, Wei and Yang \cite{Wang_Wei_Yang2011} further established its validity for  $\mathcal{N}\ge3$ and $k=\mathcal{N}-1$. Mahmoudi, Sanchez and Yao \cite{Mahmoudi_Sanchez_Yao2015}  verified it for all cases.

Moreover, the following more general Neumann problem for the nonlinear Schr\"{o}dinger equation on smooth bounded domain also admits solutions  concentrating on high-dimensional sets:
\begin{equation}\label{e90}
\left\{\begin{array}{ll}
\varepsilon^2 \Div(\nabla_{\mathfrak{a}(y)}u)-V(y)u+u^p=0, \qquad u>0, &\mbox{in $\Omega$,}\\
\nabla_{\mathfrak{a}(y)} u\cdot \mathbf{n}=0,  &\mbox{on $\partial\Omega$,}
\end{array}\right.
\end{equation}
where $\Omega\subset \R^{\mathcal{N}}$ is a smooth bounded domain, $\mathbf{n}$ is the unit outward normal of $\partial \Omega$,
\[\mathfrak{a}(y)=(\mathfrak{a}_1(y),\mathfrak{a}_2(y), \cdots, \mathfrak{a}_N(y))\]
and
\[\nabla_{\mathfrak{a}(y)}u=(\mathfrak{a}_1(y)u_{y_1}, \mathfrak{a}_2(y)u_{y_2}, \cdots, \mathfrak{a}_N(y)u_{y_N}).\]

In the case of $\mathfrak{a}\equiv 1$ and $V(y)\equiv 1$, \eqref{e90} becomes the following problem
\begin{equation}\label{e113}
\left\{\begin{array}{ll}
\varepsilon^2 \Delta u-u+u^p=0, \qquad u>0, &\mbox{in $\Omega$,}\\
\frac{\partial u}{\partial \mathbf{n}}=0,  &\mbox{on $\partial\Omega$,}
\end{array}\right.
\end{equation}
It is known as the stationary equation of Keller-Segal system in chemotaxis \cite{Lin_Ni_Takagi1988}. It can be also considered as a limiting stationary equation of Gierer-Meinhardt system in biological pattern formation \cite{Gierer_Meinhardt1973}. In the pioneering papers \cite{Lin_Ni_Takagi1988,Ni_Takagi1991,Ni_Takagi1993}, Lin, Ni and Takagi proved that the least energy solution to \eqref{e113} concentrate near a maximum point of the mean curvature of $\partial \Omega$ as $\varepsilon\to 0$. Since then, many paper discussed the solutions to \eqref{e113} concentrating near one or multiple points in $\bar{\Omega}$. The existence of solutions concentrating at some points in $\Omega$ was studied in \cite{Bates_Fusco2000,Dancer_Yan1999,delPino_Felmer_Wei2000,Grossi_Pistoia2000,Grossi_Pistoia2000,Grossi_Pistoia2000,Grossi_Pistoia2000,Gui_Wei1999,Wei1998} and references therein. The solutions concentrating near some points on $\partial\Omega$ were discussed in \cite{Bates_Dancer_Shi1999,Dancer_Yan1999a,delPino_Felmer_Wei1999,Gui_Wei_Winter2000,Li1998,Wei1998,Wei_Winter1998} and references therein. We refer to the survey paper \cite{Ni2004} for this topic up to 2004.

Ni conjectured in \cite{Ni1998} that \eqref{e113} admits  solutions concentrating on a $k$-dimensional subset of $\bar{\Omega}$. For $\mathcal{N}\ge 2$, Malchiodi and Montenegro \cite{Malchiodi_Montenegro2002,Malchiodi_Montenegro2004} obtained a sequence of solutions to \eqref{e113} that  concentrate near $\partial\Omega$ or one of its connected components. In three dimensional setting, Malchiodi \cite{Malchiodi2005} constructed a sequence of solutions to \eqref{e113} concentrating near a nondegenerate closed geodesic on $\partial\Omega$. These results were subsequently generalized by Mahmoudi and Malchiodi. In \cite{Mahmoudi_Malchiodi2007}, they built solutions to \eqref{e113} concentrating along a $k$-dimensional non-degenerate minimal submanifold for $1\le k\le \mathcal{N}-1$. However, in the two dimensional case, Wei and Yang \cite{Wei_Yang2007} constructed a sequence of solutions concentrating near a segment $\Gamma_2$ in $\Omega$, where $\Gamma_2$ intersects $\partial \Omega$ at a right angle and divides $\Omega$ into two parts. Unlike the solutions in  \cite{delPino_Kowalczyk_Wei2007, Wang_Wei_Yang2011, Mahmoudi_Sanchez_Yao2015},  Wei and Yang \cite{Wei_Yang2008} further constructed solutions to \eqref{e113} with  multiple concentration layers near $\Gamma_2$.

In the case of $\mathfrak{a}(y)\equiv 1$ and $\mathcal{N}=2$, Wei, Xu and Yang \cite{Wei_Xu_Yang2018} investigated \eqref{e90}. Let $\Gamma\subset\bar{\Omega}$ be a curve that intersects $\partial \Omega$ perpendicularly and dividing $\Omega$ into two subdomains. Under the assumption that $\Gamma$ is a nondegenerate critical point of $\int_\Gamma V^{\frac{p+1}{p-1}-\frac12}dvol$. They constructed a  sequence of solutions to \eqref{e90} concentrating along $\Gamma$. Recently, Wei and Yang \cite{Wei_Yang2021} extended the result in \cite{Wei_Yang2008} to the general case. They  constructed a sequence of solutions to \eqref{e90} with multiple concentration layers near either a closed curve $\Gamma_0\subset\Omega$ or a curve $\Gamma_1$ that intersects $\partial \Omega$ at a right angle and divides $\Omega$ into two parts. Furthermore, both $\Gamma_0$ and $\Gamma_1$ are  nondegenerate geodesics embedded in the Riemannian manifold $\R^2$ equipped with the metric $V^{\frac{2(p+1)}{p-1}-1}[\mathfrak{a}_2(y)dy_1^2+\mathfrak{a}_1(y)dy_2^2]$.

Inspired by \cite{Wei_Yang2021,Wei_Yang2008}, we consider the following problem:
\begin{equation}\label{2}
\left\{\begin{array}{ll}
-\Div (A(x)\nabla u)=|u|^p-\tilde{t}\mathbf{\Psi}(x), &\mbox{in $\Omega$,} \\
u=0, & \mbox{on $\partial \Omega$},
\end{array}
\right.
\end{equation}
where $\Omega \subset \R^2$ is a smooth bounded domain, $\tilde{t}>0$ is a constant, and  $\mathbf{\Psi}(x)$ is an eigenfunction corresponding to the first Dirichlet eigenvalue of the operator $\mathfrak{L}(u)=-\Div (A(x)\nabla u)$ on $\Omega$.
Moreover, $A(x)=\{A_{ij}(x)\}_{2\times 2}$ is a symmetric positive definite matrix-valued function, and satisfies the condition that
\begin{equation}\label{e98}
\lambda |\alpha|^2 \leq \langle A(x)\alpha, \alpha \rangle \leq \Lambda |\alpha|^2,\qquad \mathrm{for} \qquad \forall x,\alpha \in \R^n,
\end{equation}
where $\lambda$ and $\Lambda$ are positive constants.
Note that $\Div(\nabla_{\mathfrak{a}(x)}u)$ is a special form of $\Div (A(x)\nabla u)$. Building on the aforementioned work, we investigate  whether problem \eqref{2} admits solutions similar to those constructed in \cite{Wei_Yang2021,Wei_Yang2008}.

Our motivation for writing this paper is twofold. First, we aim to construct solutions to \eqref{2} with multiple concentration layers. Second, we intend to explore the influence of the matrix $A(x)$ on the high-dimensional concentration behaviors of the solutions to \eqref{2}.

Let $\varepsilon^2=\tilde{t}^{-(p-1)/p}$. It is straightforward to see that $u$ is a solution of \eqref{2} if and only if $\tilde{t}^{-\frac1p}u$ is a solution of the following problem
\begin{equation}\label{1}
\left\{\begin{array}{ll}
-\varepsilon^2\Div (A(x)\nabla u)=|u|^p-\mathbf{\Psi}(x), &\mbox{in $\Omega$,} \\
u=0, & \mbox{on $\partial \Omega$}.
\end{array}
\right.
\end{equation}
\begin{theorem}\label{th1}
Let $p>3$.
Assume that $\Gamma$ is a simple closed smooth curve of unit length in $\Omega$, which is also a non-degenerate critical point of the functional
\begin{equation}\label{e91}
\mathcal{K}(\Gamma)=\int_\Gamma \mathbf{\Psi}^{\frac{p+3}{2p}}dvol_{\mathfrak{g}},
\end{equation}
where $\mathfrak{g}(X,Y)=\langle A^*X,Y\rangle$ denotes a Riemannian metric on $\R^2$ and $A^*$ is the adjoint matrix of $A$. We further assume that the following inequality holds on $\Gamma$:
\begin{eqnarray*}
\Upsilon_0&=&-\alpha^{1-p}\left[ 2\alpha^{-1}\alpha'a_{22}+\beta^{-1}\beta'b_{11}+b_{22}-a_{33}+\frac{p+3}2\alpha^{p-2}\beta^{-2}\mathbf{q}_{tt}+\frac{p+2}{2(p+3)}\alpha^{1-p} \beta^2 (a_{32})^2\right. \\
&&\left.-\frac{p+1}{p+3}\alpha^{1-p}\beta^2 a_{32}b_{21}-\frac{2}{p+3}\alpha^{1-p}\beta^2(b_{21})^2 +a_{11}\left(\beta^{-1}\beta''+2\alpha^{-1}\beta^{-1}\alpha'\beta'-\beta^{-2}(\beta')^2\right)\right]>0,
\end{eqnarray*}
where $a_{ij}$'s and $b_{ij}$'s are defined in Lemma \ref{lm8}, $\alpha$ and $\beta$ are defined in \eqref{80}, and $\mathbf{q}$ is defined in \eqref{e88}.

Then for each  integer $N>0$, there exists a sequence of $\varepsilon$, \textit{i.e.}, $\{\varepsilon_l\}$ converging to $0$ such that \eqref{1} admits a positive solution $u_{\varepsilon_l}$ with exactly $N$ concentration layers, where the mutual distance between these layers is $O(\varepsilon_l|\log\varepsilon_l|)$. The center of mass of the $N$ concentration layers collapses to $\Gamma$ at a rate of $O(\varepsilon_l^{1+\mu})$ for some small positive constant $\mu\in(0,1/2)$. More precisely $u_{\varepsilon_l}$ has the following form:
\begin{eqnarray*}
u_{\varepsilon_l}(y_1,y_2)+\mathbf{\Psi}^{\frac1p}(y_1,y_2)\approx \mathbf{\Psi}^{\frac1p}(\gamma(\theta))\sum_{k=1}^N w\left(\left[\frac{\mathbf{\Psi}^{\frac{p-1}p}(1-\langle\gamma',n\rangle)}{\langle A^*n,n\rangle}\right]^{\frac12}\frac{t-\varepsilon_l f_k}{\varepsilon_l}\right),
\end{eqnarray*}
where $\gamma$ is a natural parametrization of $\Gamma$, $n$ is the unit vector defined in \eqref{e102} and $w$ is  the unique solution of the following problem:
\begin{equation}\label{29}
-w''=|1-w|^p -1, \qquad w>0 \quad \mathrm{in} \quad \R, \qquad w'(0)=w(\pm \infty)=0.
\end{equation}
In the expression above, the functions $f_j$'s satisfy:
\[\|f_j\|_{L^\infty(0,1)}\le C|\log\varepsilon_l|^2,\qquad \sum_{j=1}^N f_j =O(\frac1{|\log\varepsilon_l|^{\frac32}}),\]
\[\min_{1\le j\le N}(f_{j+1}-f_j)\approx\frac2{\sqrt{p}}|\log\varepsilon_l|\left[\frac{\langle A^*n,n\rangle}{\Psi^{\frac{p-1}p}(1-\langle\gamma',n\rangle)}\right]^{\frac12}\]
and solves the Jacobi-Toda system for $j=1,2,\cdots,N$:
\begin{eqnarray*}
\nonumber&&\varepsilon^2\alpha^{1-p}\beta\left\{-a_{11}f_j''+\left[a_{22}-b_{11}-a_{11}\left(\beta^{-1}\beta'+2\alpha^{-1}\alpha'\right)\right]f_j'+\left[a_{22}(\beta^{-1}\beta'+2\alpha^{-1}\alpha') +b_{22}-a_{33}\right]f_j\right. \\
\nonumber&&+\left[\frac{p+3}2\alpha^{p-2}\beta^{-2} \mathbf{q}_{tt}+\frac{p+2}{2(p+3)}\alpha^{1-p}\beta^2(a_{32})^2-\frac{p+1}{p+3}\alpha^{1-p}\beta^2 a_{32}b_{21}-\frac2{p+3}\alpha^{1-p}\beta^2 (b_{21})^2\right]f_j \\
&&+C_0 p\alpha_p\left[e^{-\sqrt{p}\beta(f_j-f_{j-1})}-e^{-\sqrt{p}\beta (f_{j+1}-f_j)}\right]\approx0.
\end{eqnarray*}
\end{theorem}

\begin{remark}
The condition that  $\Gamma$ is a nondegenerate critical point of the functional $\mathcal{K}(\Gamma)$ is equivalent to  $\Gamma$ being a non-degenerate geodesic embedded into the Riemannian manifold $(\R^2,\tilde{\mathfrak{g}})$, where $\tilde{\mathfrak{g}}(y)=\mathbf{\Psi}^{\frac{p+3}p}[A_{22}dy_1^2-2A_{12}(y)dy_1dy_2+A_{11}dy_2^2]$.

In the Jacobi-Toda system, we does not impose any boundary condition, since the solutions we found are clustered near some closed curve $\Gamma \subset \Omega$. In this case $\mathrm{dist}(\Gamma, \partial \Omega)>0$.
\end{remark}
To prove Theorem \ref{th1}, we will employ the infinite dimensional reduction method developed in \cite{delPino_Kowalczyk_Wei2007,delPino_Kowalczyk_Wei2008,Wei_Yang2021}. Infinite dimensional reduction method is commonly used to find solutions to elliptic partial differential equations that concentrate on high dimensional sets. To construct a solution to \eqref{1}, we first analyze the negative solution of \eqref{1} and their associated properties in Section \ref{sectA}. However, we must overcome the difficulty that  $-\Div (A(x)\nabla u)$ is not a symmetric operator( see Lemma \ref{lm1} for details).

To derive a local approximate solution to \eqref{1}, we need to expand the operator $\Div(A(\varepsilon y)\nabla u)$ in an appropriate manner. Wei and Yang \cite{Wei_Yang2021} developed a method to expand $\Div(\nabla_{\mathfrak{a}(y)}u)$, while we propose a simpler approach for the more general operator $\Div(A(\varepsilon y)\nabla u)$ in Lemma \ref{lm8}. However, there is a slight difference. The solution we constructed concentrates near a closed curve $\Gamma\subset\Omega$ that does not intersect the boundary. Thus, we do not need to consider the boundary condition when expanding $\Div(A(\varepsilon y)\nabla u)$. If we use the method in Lemma \ref{lm8} to consider the following more general nonlinear Schr\"{o}dinger equation
\begin{equation*}
\left\{\begin{array}{ll}
\varepsilon^2 \Div(A(x)\nabla u)-V(y)u+u^p=0, \qquad u>0, &\mbox{in $\Omega$,}\\
A(x)\nabla u\cdot \mathbf{n}=0,  &\mbox{on $\partial\Omega$,}
\end{array}\right.
\end{equation*}
we believe it can refine the results obtained in \cite{Wei_Yang2021}.

Similar to \cite{Wei_Yang2021}, constructing a local approximate solution with $N$ concentration layers requires refined asymptotic estimate of the function $w$. Nevertheless, $w$ admits no explicit expression( c.f. \cite[(3.1)]{Wei_Yang2021}). To overcome this difficulty, we use the method in \cite{Gidas_Ni_Nirenberg1981} to derive the asymptotic estimates of $w$ in the Section \ref{sectB}. The identity \eqref{e115} is crucial in constructing the solutions with $N$ concentration layers(see \eqref{e82}, \eqref{e110}and \eqref{e114}).

This paper is organized as follows. In Section \ref{sect1}, we establish a new modified Fermi coordinate system in a small neighborhood of $\Gamma$ and derive the local form of the operator $\Div (A(\varepsilon y)\nabla u)$ under the stretched modified Fermi coordinate. In Section \ref{sect2}, we construct a local approximate solution to \eqref{1}. Subsequently, we perform the gluing procedure in Section \ref{sect5} and solve the outer problem. To tackle  the inner problem, we develop  the linear and nonlinear theories in Section \ref{sect6} and Section \ref{sect9}, respectively. Through intricate calculations in Section \ref{sect4}, we reduce the original problem to a system of partial differential equations in Section \ref{sect7}.  Finally, Theorem \ref{th1} is proved in Section \ref{sect8}.

\section{geometric description}\label{sect1}
Let $\gamma: [0,1]\to \Gamma \subset \Omega$ be the natural parametrization of $\Gamma$ and $\nu$ be the outward unit normal vector to $\Gamma$. The following Frenet formula holds:
\[\gamma''(\theta)=k(\theta)\nu(\theta), \qquad \nu'(\theta)=-k(\theta)\gamma'(\theta),\]
where $k(\theta)$ represents the curvature of $\Gamma$. For sufficient small $\tilde{\delta}\in \left(0, \mathrm{dist}(\Gamma, \partial \Omega)\right)$, the $\tilde{\delta}$-neighborhood of $\Gamma$ is parameterized by
\[y=\gamma(\theta)+t\nu(\theta), \quad \mathrm{where} \quad \theta\in[0,1], \quad t\in(-\tilde{\delta}, \tilde{\delta}).\]

To construct solutions to \eqref{2} near $\Gamma$, we modify the Fermi coordinate above. Define the following unit vector on $\Gamma$ by
\begin{equation}\label{e102}
n(\theta)=\frac{A(\gamma(\theta))\nu(\theta)}{|A(\gamma(\theta))\nu(\theta)|}.
\end{equation}
Then the following map is a local diffeomorphism for $\delta_0>0$ small enough:
\[\Phi^0(\theta, t)=\gamma(\theta)+tn(\theta), \quad \mathrm{where} \quad \theta\in[0,1], \quad t\in(-\delta_0, \delta_0).\]
Under this local coordinate, the components of the standard Riemanian metric of $\R^2$ are
\[\tilde{g}_{11}(\theta,t)=\left\langle \frac{\partial \Phi^0}{\partial \theta}, \frac{\partial \Phi^0}{\partial \theta}\right\rangle= 1+2t\langle \gamma'(\theta), n'(\theta)\rangle +t^2 \langle n'(\theta), n'(\theta)\rangle,\]
\[\tilde{g}_{12}(\theta,t)=\tilde{g}_{21}(\theta,t)=\left\langle \frac{\partial \Phi^0}{\partial \theta}, \frac{\partial \Phi^0}{\partial t}\right\rangle=\langle \gamma'(\theta), n(\theta)\rangle\]
and
\[\tilde{g}_{22}(\theta,t)=\left\langle \frac{\partial \Phi^0}{\partial t}, \frac{\partial \Phi^0}{\partial t}\right\rangle=1.\]
Thus,
\[\det \tilde{g}=1-\langle \gamma'(\theta), n(\theta)\rangle^2+2t\langle \gamma'(\theta), n'(\theta)\rangle +t^2 \langle n'(\theta), n'(\theta)\rangle.\]
Note that $\det \tilde{g}>0$ for $\delta_0>0$ small enough, since $1-\langle \gamma'(\theta), n(\theta)\rangle^2>0$. Otherwise, if $1-\langle \gamma'(\theta), n(\theta)\rangle^2=0$, then $n(\theta)=\pm\gamma'(\theta)$. From \eqref{e102}, $A\nu=C\gamma'$ for some constant $C$, leading to
\[\lambda |\nu|^2\le \langle  A\nu,\nu\rangle=C\langle \gamma', \nu\rangle=0.  \]
Then we get a contradiction.

Direct computation yields the following expansions:
\[\tilde{g}^{11}(\theta,t)=\frac1{1-\langle \gamma', n\rangle^2}-t\frac{2\langle\gamma',n'\rangle}{(1-\langle \gamma', n\rangle^2)^2}+t^2 \left[\frac{4\langle\gamma',n'\rangle^2}{({1-\langle \gamma', n\rangle}^2)^3}-\frac{\langle n',n'\rangle}{({1-\langle \gamma', n\rangle}^2)^2}\right]+O(t^3),\]
\[\tilde{g}^{12}(\theta,t)=\tilde{g}^{21}(\theta,t)=-\frac{\langle \gamma',n \rangle}{1-\langle \gamma', n\rangle^2}+ t\frac{2\langle \gamma',n\rangle\langle \gamma',n'\rangle}{(1-\langle \gamma', n\rangle^2)^2}+t^2\left[\frac{\langle \gamma',n\rangle\langle n',n'\rangle}{(1-\langle \gamma', n\rangle^2)^2}-\frac{4\langle \gamma',n\rangle\langle \gamma',n'\rangle^2}{(1-\langle \gamma', n\rangle^2)^3}\right]+O(t^3)\]
and
\[\tilde{g}^{22}(\theta,t)=\frac1{1-\langle \gamma', n\rangle^2}-t\frac{2\langle\gamma',n'\rangle\langle \gamma',n\rangle^2}{(1-\langle \gamma', n\rangle^2)^2}+t^2 \left[\frac{4\langle\gamma',n'\rangle^2\langle \gamma',n \rangle^2}{({1-\langle \gamma', n\rangle}^2)^3}-\frac{\langle n',n'\rangle\langle \gamma',n\rangle^2}{({1-\langle \gamma', n\rangle}^2)^2}\right]+O(t^3).\]

For notation simplicity, we denote
\[A(\theta, t)=A(\gamma(\theta)+tn(\theta)), \qquad {A^*}(\theta, t)={A^*}(\gamma(\theta)+tn(\theta))\]
and
\begin{equation}\label{e88}
\mathbf{q}(\theta, t)=\mathbf{\Psi}^{\frac1p}(\gamma(\theta)+tn(\theta)).
\end{equation}
Recall the functional $\mathcal{K}(\Gamma)$ is defined in \eqref{e91}. We have the following lemma.

\begin{lemma}\label{lm5}
If the simple closed curve $\Gamma\subset \Omega$ is a non-degenerate critical point of the functional $\mathcal{K}(\Gamma)$, then:
\begin{enumerate}
  \item It holds that
  \begin{equation}\label{33}
  \frac{p+3}2 \langle{A^*}\gamma',\gamma' \rangle \mathbf{q}_t=-\left(\langle {A^*}n',\gamma'\rangle +\frac12\langle{A^*_t}\gamma',\gamma'\rangle\right)\mathbf{q} \quad \mathrm{on} \quad \Gamma.
  \end{equation}
  \item The following equation has only trivial solution
  \begin{eqnarray}\label{42}
  \nonumber&&-\left(\frac{\langle {A^*}n,n\rangle}{\sqrt{\langle {A^*}\gamma',\gamma'\rangle}}\mathbf{q}^{\frac{p+3}2} h'\right)'+ \frac{\mathbf{q}^{\frac{p+3}2}}{\sqrt{\langle {A^*}\gamma',\gamma'\rangle}}\left(\langle{A^*}n',n'\rangle+2\langle{A^*_t} n',\gamma'\rangle+\frac12\langle{A^*_{tt}} \gamma',\gamma'\rangle \right) h \\
  \nonumber&&  -\left[\frac{\mathbf{q}^{\frac{p+3}2}}{\sqrt{\langle{A^*}\gamma',\gamma' \rangle}}\left(\langle{A^*}n,n'\rangle +\langle {A^*_t} n,\gamma'\rangle  \right)\right]_\theta h  -\frac{\mathbf{q}^{\frac{p+3}2}}{\langle{A^*}\gamma',\gamma' \rangle^{\frac32}}\left(\langle{A^*}n',\gamma'\rangle +\frac12\langle {A^*_t} \gamma',\gamma'\rangle\right)^2 h \\
  && +\left[\frac{p+3}2\mathbf{q}^{\frac{p+1}2}\mathbf{q}_{tt}-\frac{(p+3)(p+5)}4\mathbf{q}^{\frac{p-1}2}(\mathbf{q}_t)^2\right]\sqrt{\langle{A^*}\gamma',\gamma' \rangle} h=0.
\end{eqnarray}
\end{enumerate}
\end{lemma}
\begin{proof}
For any $h\in C^\infty(\R/ \Z)$, consider the closed curves $\Gamma_h:\gamma_h(\theta)=\gamma(\theta)+h(\theta)n(\theta)$.
It is apparent that $\Gamma_0=\Gamma$. The functional $J(h):=\mathcal{K}(\Gamma_h)$ is of the following form:
\begin{equation}\label{21}
J(h)=\int_0^1 \mathbf{\Psi}^{\frac{p+3}{2p}}(\gamma_h(\theta))\sqrt{\langle{A^*}(\gamma_h(\theta))\gamma_h'(\theta), \gamma_h'(\theta) \rangle} d\theta =\int_0^1\mathbf{q}^{\frac{p+3}2}(\theta, h(\theta))\sqrt{W(\theta, h(\theta))}d\theta,
\end{equation}
where  $W(\theta, h(\theta)):=\langle{A^*}(\gamma_h(\theta))\gamma_h'(\theta), \gamma_h'(\theta) \rangle$. Hence $0$ is a non-degenerate critical point of $J(h)$. It holds that \[J'(0)h=0, \qquad \forall h\in C^\infty(\R/ \Z).\]

A direct computation yields that
\begin{eqnarray}\label{41}
\nonumber W(\theta, h(\theta))&=& \langle {A^*}\gamma',\gamma'\rangle +h(\theta)\left[2\langle {A^*}n', \gamma'\rangle +\langle {A^*_t} \gamma',\gamma'\rangle\right] +2h(\theta)h'(\theta)\left[ \langle {A^*}n, n' \rangle +\langle{A^*_t} n, \gamma' \rangle \right]\\
\nonumber &&+(h'(\theta))^2\langle {A^*}n,n \rangle+(h(\theta))^2\left[\langle {A^*}n',n' \rangle +2\langle {A^*_t}n',\gamma' \rangle +\frac12\langle {A^*_{tt}}\gamma',\gamma'\rangle\right]  \\
&&+O(|h|^3)+O(|h'||h|^2)+O(|h||h'|^2),
\end{eqnarray}
where we use the fact that $\langle A^* \gamma',n\rangle =\frac{\det A}{|A\nu|}\langle \gamma', \nu \rangle=0$ on $\Gamma$. Then
\[J'(0)h=\int_0^1\frac{\mathbf{q}^{\frac{p+1}2}}{\sqrt{\langle {A^*}\gamma',\gamma'\rangle}}\left[\frac{p+3}2\mathbf{q}_t\langle {A^*}\gamma',\gamma' \rangle +\mathbf{q}\left(\langle {A^*}n,\gamma'\rangle +\frac12\langle{A^*_t}\gamma',\gamma'\rangle\right) \right]hd\theta.\]
Thus, \eqref{33} holds. From \eqref{33}, \eqref{41} and direct computation,
\begin{eqnarray*}
&&J''(0)[h,h] \\
&=&\int_0^1\left[\frac{(p+3)(p+1)}4\mathbf{q}^{\frac{p-1}2}(\mathbf{q}_t)^2+\frac{p+3}2\mathbf{q}^{\frac{p+1}2}\mathbf{q}_{tt}\right]\sqrt{\langle {A^*}\gamma',\gamma'\rangle} h^2 d\theta \\
&& +\frac{p+3}2\int_0^1\frac{\mathbf{q}^{\frac{p+1}2}\mathbf{q}_t}{\sqrt{\langle {A^*}\gamma',\gamma'\rangle}}\left[2\langle {A^*}n',\gamma'\rangle +\langle {A^*_t}\gamma',\gamma'\rangle\right]h^2d\theta  \\
&&+ \int_0^1 \frac{\mathbf{q}^{\frac{p+3}2}}{\sqrt{\langle {A^*}\gamma',\gamma'\rangle}}\left[\left(\langle {A^*}n',n'\rangle +2\langle {A^*_t}n',\gamma' \rangle +\frac12\langle {A^*_{tt}}\gamma',\gamma'\rangle \right)h^2+\langle{A^*}n,n\rangle (h')^2\right. \\
&&\left. +2\left(\langle {A^*}n,n'\rangle +\langle {A^*_t} n,\gamma'\rangle\right)hh'\right]d\theta -\int_0^1 \frac{\mathbf{q}^{\frac{p+3}2}}{\langle{A^*}\gamma',\gamma' \rangle^{\frac32}}\left[\langle {A^*}n',\gamma'\rangle+ \frac12\langle{A^*_t} \gamma',\gamma'\rangle\right]^2 h^2d\theta \\
&=&\int_0^1\frac{\langle {A^*}n,n\rangle}{\sqrt{\langle {A^*}\gamma',\gamma'\rangle}}\mathbf{q}^{\frac{p+3}2}(h')^2d\theta +\int_0^1 \frac{\mathbf{q}^{\frac{p+3}2}}{\sqrt{\langle {A^*}\gamma',\gamma'\rangle}}\left(\langle{A^*}n',n'\rangle+2\langle{A^*_t} n',\gamma'\rangle+\frac12\langle{A^*_{tt}} \gamma',\gamma'\rangle \right)h^2 d\theta \\
&& -\int_0^1\left[\frac{\mathbf{q}^{\frac{p+3}2}}{\sqrt{\langle{A^*}\gamma',\gamma' \rangle}}\left(\langle{A^*}n,n'\rangle +\langle {A^*_t} n,\gamma'\rangle  \right)\right]_\theta h^2d\theta -\int_0^1 \frac{\mathbf{q}^{\frac{p+3}2}}{\langle{A^*}\gamma',\gamma' \rangle^{\frac32}}\left(\langle{A^*}n',\gamma'\rangle +\frac12\langle {A^*_t} \gamma',\gamma'\rangle\right)^2 h^2 \\
&& +\int_0^1 \left[\frac{p+3}2\mathbf{q}^{\frac{p+3}2}\mathbf{q}_{tt}-\frac{(p+3)(p+5)}4\mathbf{q}^{\frac{p-1}2}(\mathbf{q}_t)^2\right]\sqrt{\langle{A^*}\gamma',\gamma' \rangle} h^2.
\end{eqnarray*}
Since $0$ is a nondegenerate critical point of $J(h)$, we get \eqref{42} has only trivial solutions.
\end{proof}

In some neighborhood of $\Gamma_\varepsilon=\Gamma/\varepsilon$, we define the stretched modified Fermi coordinate by
\[\Phi_\varepsilon(z,s)=\frac1\varepsilon \Phi^0(\varepsilon z,\varepsilon s)=\frac1\varepsilon \left(\gamma(\varepsilon z) +\varepsilon s n(\varepsilon z)\right), \quad \mathrm{where} \quad z\in[0,1/\varepsilon], \quad s\in(-\delta_0/\varepsilon, \delta_0/\varepsilon).\]
Under this coordinate, the components of the Riemanian metric are
\[g_{11}(z,s)=\left\langle \frac{\partial \Phi_\varepsilon}{\partial z}, \frac{\partial \Phi_\varepsilon}{\partial z}\right\rangle= 1+2\varepsilon s\langle \gamma'(\varepsilon z), n'(\varepsilon z)\rangle +\varepsilon^2s^2 \langle n'(\varepsilon z), n'(\varepsilon z)\rangle,\]
\[g_{12}(z,s)=g_{21}(z,s)=\left\langle \frac{\partial \Phi_\varepsilon}{\partial z}, \frac{\partial \Phi_\varepsilon}{\partial s}\right\rangle=\langle \gamma'(\varepsilon z), n(\varepsilon z)\rangle\]
and
\[g_{22}(z,s)=\left\langle \frac{\partial \Phi_\varepsilon}{\partial s}, \frac{\partial \Phi_\varepsilon}{\partial s}\right\rangle=1.\]
Thus,
\[\det g=1-\langle \gamma'(\varepsilon z), n(\varepsilon z)\rangle^2+2\varepsilon s\langle \gamma'(\varepsilon z), n'(\varepsilon z)\rangle +\varepsilon^2 s^2\langle n'(\varepsilon z), n'(\varepsilon z)\rangle.\]
We obtain that $g^{ij}$'s have the following expansion
\[g^{11}(z,s)=\frac1{1-\langle \gamma', n\rangle^2}-\varepsilon s\frac{2\langle\gamma',n'\rangle}{(1-\langle \gamma', n\rangle^2)^2}+\varepsilon^2 s^2 \left[\frac{4\langle\gamma',n'\rangle^2}{({1-\langle \gamma', n\rangle}^2)^3}-\frac{\langle n',n'\rangle}{({1-\langle \gamma', n\rangle}^2)^2}\right]+O(\varepsilon^3s^3),\]
\begin{eqnarray*}
g^{12}(z,s)=g^{21}(z,s)&=&-\frac{\langle \gamma',n \rangle}{1-\langle \gamma', n\rangle^2}+ \varepsilon s\frac{2\langle \gamma',n\rangle\langle \gamma',n'\rangle}{(1-\langle \gamma', n\rangle^2)^2} \\
&&+\varepsilon^2 s^2\left[\frac{\langle \gamma',n\rangle\langle n',n'\rangle}{(1-\langle \gamma', n\rangle^2)^2}-\frac{4\langle \gamma',n\rangle\langle \gamma',n'\rangle^2}{(1-\langle \gamma', n\rangle^2)^3}\right]+O(\varepsilon^3 s^3)
\end{eqnarray*}
and
\[g^{22}(z,s)=\frac1{1-\langle \gamma', n\rangle^2}-\varepsilon s\frac{2\langle\gamma',n'\rangle\langle \gamma',n\rangle^2}{(1-\langle \gamma', n\rangle^2)^2}+\varepsilon^2 s^2 \left[\frac{4\langle\gamma',n'\rangle^2\langle \gamma',n \rangle^2}{({1-\langle \gamma', n\rangle}^2)^3}-\frac{\langle n',n'\rangle\langle \gamma',n\rangle^2}{({1-\langle \gamma', n\rangle}^2)^2}\right]+O(\varepsilon^3 s^3).\]

Next, we expand the operator $\Div(A(\varepsilon y)\nabla v)$ under the local coordinate $(z,s)$.
\begin{lemma}\label{lm8}
Under the stretched modified Fermi coordinate defined by $\Phi_\varepsilon(z,s)$, we get
\begin{eqnarray*}
\nonumber\Div(A(\varepsilon y)\nabla v)&=& a_{11}(\varepsilon z)v_{zz}+2\varepsilon s a_{22}(\varepsilon z)v_{sz} +\left(a_{31}(\varepsilon z)+\varepsilon s a_{32}(\varepsilon z) +\varepsilon^2 s^2 a_{33}(\varepsilon z)\right)v_{ss} \\
&& +\varepsilon b_{11}(\varepsilon z)v_z +\left(\varepsilon b_{21}(\varepsilon z)+\varepsilon^2 s b_{22}(\varepsilon z)\right) v_s +B_0(v),
\end{eqnarray*}
where
\[a_{11}(\theta)=\frac{\langle{A^*}n,n\rangle}{1-\langle\gamma',n\rangle^2},\qquad a_{22}(\theta)=-\frac{\langle{A^*}n,n'\rangle}{1-\langle\gamma',n\rangle^2}-\frac{\langle{A^*_t}n,\gamma'\rangle}{1-\langle\gamma',n\rangle^2}, \]
\[a_{31}(\theta)=\frac{\langle{A^*}\gamma',\gamma'\rangle}{1-\langle\gamma',n\rangle^2}, \qquad a_{32}(\theta)=\frac{\langle{A^*_t}\gamma',\gamma'\rangle}{1-\langle\gamma',n\rangle^2}+ \frac{2\langle{A^*}\gamma',n'\rangle}{1-\langle\gamma',n\rangle^2}-\frac{2\langle\gamma',n'\rangle\langle{A^*}\gamma',\gamma'\rangle}{\left(1-\langle\gamma',n\rangle^2\right)^2}, \]
\begin{eqnarray*}
a_{33}(\theta)&=& \frac1{1-\langle\gamma',n\rangle^2}\left(\langle{A^*}n',n'\rangle +2\langle{A^*_t}\gamma',n'\rangle+\frac12\langle{A^*_{tt}}\gamma',\gamma'\rangle\right)  \\
&&-\frac{2\langle\gamma',n'\rangle}{\left(1-\langle\gamma',n\rangle^2\right)^2}\left(\langle{A^*_t} \gamma',\gamma'\rangle +2\langle{A^*}\gamma',n'\rangle\right)  \\
&&+\left(\frac{4\langle\gamma',n'\rangle^2}{\left(1-\langle\gamma',n\rangle^2\right)^3}- \frac{\langle n',n'\rangle}{\left(1-\langle\gamma',n\rangle^2\right)^2}\right)\langle{A^*}\gamma',\gamma'\rangle,
\end{eqnarray*}
\[b_{11}(\theta)=\left[\frac{\langle{A^*}n,n\rangle}{1-\langle\gamma',n\rangle^2}\right]_{\theta}-\frac12\left(\frac1{1-\langle\gamma',n\rangle^2}\right)_\theta \langle{A^*}n,n\rangle -\frac{\langle{A^*}n,n'\rangle}{1-\langle\gamma',n\rangle^2}-\frac{\langle{A^*_t}n,\gamma'\rangle}{1-\langle\gamma',n\rangle^2},\]
\[b_{21}(\theta)=\frac{\langle\gamma',n'\rangle}{1-\langle\gamma',n\rangle^2}a_{31}+a_{32},\]
\[b_{22}(\theta)=\frac12\frac{(1-\langle\gamma',n\rangle^2)_\theta}{1-\langle\gamma',n\rangle^2}a_{22}+\left[\frac{\langle n',n'\rangle}{1-\langle\gamma',n\rangle^2}-\frac{2\langle\gamma',n'\rangle^2}{(1-\langle\gamma',n\rangle^2)^2}\right]a_{31}+\frac{\langle\gamma',n'\rangle}{1-\langle\gamma', n\rangle^2}a_{32} +\partial_\theta a_{22}+2a_{33}\]
and
\[B_0(v)=\varepsilon a_{12}(\varepsilon z,s)v_{zz}+\varepsilon^2 a_{23}(\varepsilon z,s)v_{zs}+\varepsilon^3a_{33}(\varepsilon z,s)v_{ss} +\varepsilon^2 b_{12}(\varepsilon z,s) v_z +\varepsilon^3 b_{23}(\varepsilon z,s)v_s.\]
In the expression above, funtions $a_{12}$, $a_{23}$, $a_{33}$, $b_{12}$ and $b_{23}$ are smooth functions satisfying the following estimate
\[|a_{12}(\varepsilon z, s)|\le C(1+|s|), \quad |a_{23}(\varepsilon z, s)|\le C(1+|s|^2), \quad |a_{33}(\varepsilon z, s)|\le C(1+|s|^3),\]
and
\[|b_{12}(\varepsilon z, s)|\le C(1+|s|), \quad |b_{23}(\varepsilon z, s)|\le C(1+|s|^2).\]
\end{lemma}
\begin{proof}
Under the stretched modified Fermi coordinate, we get the following expression from the definition of gradient operator in Riemannian manifold(c.f.\cite{Chen_Li2002}):
\[\nabla v=\left(g^{11}v_z+g^{12}v_s\right)\frac{\partial \Phi_\varepsilon}{\partial z}+\left(g^{21}v_z+g^{22}v_s\right)\frac{\partial \Phi_\varepsilon}{\partial s}.\]
Then
\begin{equation}\label{e89}
A(\varepsilon z,\varepsilon s)\nabla v=\left(g^{11}v_z+g^{12}v_s\right)A(\varepsilon z,\varepsilon s)\frac{\partial \Phi_\varepsilon}{\partial z}+\left(g^{21}v_z+g^{22}v_s\right)A(\varepsilon z,\varepsilon s)\frac{\partial \Phi_\varepsilon}{\partial s}.
\end{equation}
According to the method in linear algebra, we have
\begin{eqnarray}\label{22}
A(\varepsilon z,\varepsilon s)\frac{\partial \Phi_\varepsilon}{\partial z}&=&\left[g^{11}\left\langle A(\varepsilon z,\varepsilon s)\frac{\partial \Phi_\varepsilon}{\partial z}, \frac{\partial \Phi_\varepsilon}{\partial z}\right\rangle+g^{12}\left\langle A(\varepsilon z,\varepsilon s)\frac{\partial \Phi_\varepsilon}{\partial z}, \frac{\partial \Phi_\varepsilon}{\partial s}\right\rangle\right]\frac{\partial \Phi_\varepsilon}{\partial z} \\
\nonumber&& + \left[g^{21}\left\langle A(\varepsilon z,\varepsilon s)\frac{\partial \Phi_\varepsilon}{\partial z}, \frac{\partial \Phi_\varepsilon}{\partial z}\right\rangle+g^{22}\left\langle A(\varepsilon z,\varepsilon s)\frac{\partial \Phi_\varepsilon}{\partial z}, \frac{\partial \Phi_\varepsilon}{\partial s}\right\rangle\right]\frac{\partial \Phi_\varepsilon}{\partial s}
\end{eqnarray}
and
\begin{eqnarray}\label{23}
A(\varepsilon z,\varepsilon s)\frac{\partial \Phi_\varepsilon}{\partial s}&=&\left[g^{11}\left\langle A(\varepsilon z,\varepsilon s)\frac{\partial \Phi_\varepsilon}{\partial s}, \frac{\partial \Phi_\varepsilon}{\partial z}\right\rangle+g^{12}\left\langle A(\varepsilon z,\varepsilon s)\frac{\partial \Phi_\varepsilon}{\partial s}, \frac{\partial \Phi_\varepsilon}{\partial s}\right\rangle\right]\frac{\partial \Phi_\varepsilon}{\partial z} \\
\nonumber&& + \left[g^{21}\left\langle A(\varepsilon z,\varepsilon s)\frac{\partial \Phi_\varepsilon}{\partial s}, \frac{\partial \Phi_\varepsilon}{\partial z}\right\rangle+g^{22}\left\langle A(\varepsilon z,\varepsilon s)\frac{\partial \Phi_\varepsilon}{\partial s}, \frac{\partial \Phi_\varepsilon}{\partial s}\right\rangle\right]\frac{\partial \Phi_\varepsilon}{\partial s}.
\end{eqnarray}
Hence \eqref{e89} can be written into the following form
\[A(\varepsilon z,\varepsilon s)\nabla v=\left(X_1v_z+X_2 v_s\right)\frac{\partial \Phi_\varepsilon}{\partial z}+\left(X_2 v_z+X_3v_s\right)\frac{\partial \Phi_\varepsilon}{\partial s},\]
where
\begin{eqnarray*}
X_1&=&\left(g^{11}\right)^2 \left\langle A(\varepsilon z,\varepsilon s)\frac{\partial \Phi_\varepsilon}{\partial z}, \frac{\partial \Phi_\varepsilon}{\partial z}\right\rangle +2g^{11}g^{12} \left\langle A(\varepsilon z,\varepsilon s)\frac{\partial \Phi_\varepsilon}{\partial s}, \frac{\partial \Phi_\varepsilon}{\partial z}\right\rangle \\
 &&+\left(g^{12}\right)^2 \left\langle A(\varepsilon z,\varepsilon s)\frac{\partial \Phi_\varepsilon}{\partial s}, \frac{\partial \Phi_\varepsilon}{\partial s}\right\rangle,
\end{eqnarray*}
\begin{eqnarray*}
X_2&=&g^{11}g^{12} \left\langle A(\varepsilon z,\varepsilon s)\frac{\partial \Phi_\varepsilon}{\partial z}, \frac{\partial \Phi_\varepsilon}{\partial z}\right\rangle +\left[\left(g^{12}\right)^2+g^{11}g^{22}\right] \left\langle A(\varepsilon z,\varepsilon s)\frac{\partial \Phi_\varepsilon}{\partial s}, \frac{\partial \Phi_\varepsilon}{\partial z}\right\rangle \\
 &&+g^{12}g^{22} \left\langle A(\varepsilon z,\varepsilon s)\frac{\partial \Phi_\varepsilon}{\partial s}, \frac{\partial \Phi_\varepsilon}{\partial s}\right\rangle
\end{eqnarray*}
and
\begin{eqnarray*}
X_3&=&\left(g^{12}\right)^2 \left\langle A(\varepsilon z,\varepsilon s)\frac{\partial \Phi_\varepsilon}{\partial z}, \frac{\partial \Phi_\varepsilon}{\partial z}\right\rangle +2g^{12}g^{22} \left\langle A(\varepsilon z,\varepsilon s)\frac{\partial \Phi_\varepsilon}{\partial s}, \frac{\partial \Phi_\varepsilon}{\partial z}\right\rangle \\
 &&+\left(g^{22}\right)^2 \left\langle A(\varepsilon z,\varepsilon s)\frac{\partial \Phi_\varepsilon}{\partial s}, \frac{\partial \Phi_\varepsilon}{\partial s}\right\rangle.
\end{eqnarray*}
Recall ${A^*}$ is the adjoint matrix of $A$. From \eqref{22}, \eqref{23} and the method in linear algebra, we get
\begin{eqnarray*}
{A^*}(\varepsilon z,\varepsilon s)\frac{\partial \Phi_\varepsilon}{\partial z}&=&\left[g^{12}\left\langle A(\varepsilon z,\varepsilon s)\frac{\partial \Phi_\varepsilon}{\partial s}, \frac{\partial \Phi_\varepsilon}{\partial z}\right\rangle+g^{22}\left\langle A(\varepsilon z,\varepsilon s)\frac{\partial \Phi_\varepsilon}{\partial s}, \frac{\partial \Phi_\varepsilon}{\partial s}\right\rangle\right]\frac{\partial \Phi_\varepsilon}{\partial z} \\
\nonumber&& - \left[g^{12}\left\langle A(\varepsilon z,\varepsilon s)\frac{\partial \Phi_\varepsilon}{\partial z}, \frac{\partial \Phi_\varepsilon}{\partial z}\right\rangle+g^{22}\left\langle A(\varepsilon z,\varepsilon s)\frac{\partial \Phi_\varepsilon}{\partial z}, \frac{\partial \Phi_\varepsilon}{\partial s}\right\rangle\right]\frac{\partial \Phi_\varepsilon}{\partial s}
\end{eqnarray*}
and
\begin{eqnarray*}
{A^*}(\varepsilon z,\varepsilon s)\frac{\partial \Phi_\varepsilon}{\partial s}\nonumber &=& -\left[g^{11}\left\langle A(\varepsilon z,\varepsilon s)\frac{\partial \Phi_\varepsilon}{\partial s}, \frac{\partial \Phi_\varepsilon}{\partial z}\right\rangle+g^{12}\left\langle A(\varepsilon z,\varepsilon s)\frac{\partial \Phi_\varepsilon}{\partial s}, \frac{\partial \Phi_\varepsilon}{\partial s}\right\rangle\right]\frac{\partial \Phi_\varepsilon}{\partial z} \\
&& + \left[g^{11}\left\langle A(\varepsilon z,\varepsilon s)\frac{\partial \Phi_\varepsilon}{\partial z}, \frac{\partial \Phi_\varepsilon}{\partial z}\right\rangle+g^{12}\left\langle A(\varepsilon z,\varepsilon s)\frac{\partial \Phi_\varepsilon}{\partial z}, \frac{\partial \Phi_\varepsilon}{\partial s}\right\rangle\right]\frac{\partial \Phi_\varepsilon}{\partial s}.
\end{eqnarray*}
Hence the following identities hold:
\[X_1=\frac1{\det g}\left\langle{A^*}(\varepsilon z,\varepsilon s)\frac{\partial \Phi_\varepsilon}{\partial s}, \frac{\partial \Phi_\varepsilon}{\partial s} \right\rangle,\]
\[X_2=-\frac1{\det g}\left\langle{A^*}(\varepsilon z,\varepsilon s)\frac{\partial \Phi_\varepsilon}{\partial s}, \frac{\partial \Phi_\varepsilon}{\partial z} \right\rangle\]
and
\[X_3=\frac1{\det g}\left\langle{A^*}(\varepsilon z,\varepsilon s)\frac{\partial \Phi_\varepsilon}{\partial z}, \frac{\partial \Phi_\varepsilon}{\partial z} \right\rangle.\]
From direct computation,
\begin{eqnarray*}
\left\langle{A^*}(\varepsilon z,\varepsilon s)\frac{\partial \Phi_\varepsilon}{\partial z}, \frac{\partial \Phi_\varepsilon}{\partial z}\right\rangle &=&\langle{A^*}\gamma',\gamma'\rangle +\varepsilon s\left[\langle{A^*_t}\gamma',\gamma'\rangle+2\langle{A^*}\gamma',n'\rangle\right]  \\
&&+\varepsilon^2s^2\left[\langle{A^*}n',n'\rangle+2\langle{A^*_t}\gamma',n' \rangle+\frac12\langle{A^*_{tt}}\gamma',\gamma'\rangle\right]+O\left(\varepsilon^3 s^3\right),
\end{eqnarray*}
\[\left\langle{A^*}(\varepsilon z,\varepsilon s)\frac{\partial \Phi_\varepsilon}{\partial s}, \frac{\partial \Phi_\varepsilon}{\partial z} \right\rangle= \varepsilon s\left[\langle{A^*}n,n'\rangle+\langle{A^*_t} n,\gamma'\rangle\right] +\varepsilon^2 s^2\left[\langle{A^*_t} n,n'\rangle +\frac12\langle{A^*_{tt}}n,\gamma'\rangle\right] +O\left(\varepsilon^3 s^3\right)\]
and
\[\left\langle{A^*}(\varepsilon z,\varepsilon s)\frac{\partial \Phi_\varepsilon}{\partial s}, \frac{\partial \Phi_\varepsilon}{\partial s} \right\rangle=\langle{A^*}n,n\rangle +\varepsilon s\langle {A^*_t} n,n\rangle+\frac12 \varepsilon^2 s^2\langle{A^*_{tt}}n,n\rangle+O\left(\varepsilon^3 s^3\right).\]
Thus,
\[X_1=\frac{\langle{A^*}n,n\rangle}{1-\langle\gamma',n\rangle^2}+O\left(\varepsilon s\right),\]
\[X_2=-\varepsilon s\left(\frac{\langle{A^*}n,n'\rangle}{1-\langle\gamma',n\rangle^2}+\frac{\langle{A^*_t}n,\gamma'\rangle}{1-\langle\gamma',n\rangle^2}\right)+ O\left(\varepsilon^2 s^2\right)\]
and
\begin{eqnarray*}
X_3&=& \frac{\langle{A^*}\gamma',\gamma'\rangle}{1-\langle\gamma',n\rangle^2} +\varepsilon s\left[\frac{\langle{A^*_t}\gamma',\gamma'\rangle}{1-\langle\gamma',n\rangle^2}+ \frac{2\langle{A^*}\gamma',n'\rangle}{1-\langle\gamma',n\rangle^2}-\frac{2\langle\gamma',n'\rangle\langle{A^*}\gamma',\gamma'\rangle}{\left(1-\langle\gamma',n\rangle^2\right)^2} \right] \\
&&+ \varepsilon^2 s^2 \left[\frac1{1-\langle\gamma',n\rangle^2}\left(\langle{A^*}n',n'\rangle +2\langle{A^*_t}\gamma',n'\rangle+\frac12\langle{A^*_{tt}}\gamma',\gamma'\rangle\right) \right. \\
&&\left.-\frac{2\langle\gamma',n'\rangle}{\left(1-\langle\gamma',n\rangle^2\right)^2}\left(\langle{A^*_t} \gamma',\gamma'\rangle +2\langle{A^*}\gamma',n'\rangle\right) \right. \\
&&\left.+\left(\frac{4\langle\gamma',n'\rangle^2}{\left(1-\langle\gamma',n\rangle^2\right)^3}- \frac{\langle n',n'\rangle}{\left(1-\langle\gamma',n\rangle^2\right)^2}\right)\langle{A^*}\gamma',\gamma'\rangle\right] +O\left(\varepsilon^3 s^3\right).
\end{eqnarray*}
According to the definition of divergence operator, we get
\begin{eqnarray*}
\Div(A(\varepsilon y)\nabla v)&=&X_1v_{zz}+2X_2v_{sz}+X_3 v_{ss} +\left(\frac{\partial_z \det g}{2\det g}X_1+\frac{\partial_s \det g}{2\det g}X_2+\partial_z X_1+\partial_s X_2\right) v_z \\
&& +\left(\frac{\partial_z \det g}{2\det g}X_2+\frac{\partial_s \det g}{2\det g}X_3+\partial_z X_2+\partial_s X_3\right) v_s.
\end{eqnarray*}
This lemma follows from direct computation.
\end{proof}

Using the notation in Lemma \ref{lm8} and \eqref{80}, we rewrite Lemma \ref{lm5} as follows:
\begin{remark}\label{rm1}
If $\Gamma\subset \Omega$ is a non-degenerate critical point of the functional $\mathcal{K}(\Gamma)$, then
\begin{equation}\label{e57}
\mathbf{q}_t =\frac1{p+3}a_{32}\alpha^{2-p}\beta^2- \frac2{p+3}b_{21}\alpha^{2-p}\beta^2,\quad \mathrm{on} \quad \Gamma
\end{equation}
and  the following problem has only trivial solution:
\begin{eqnarray}\label{e105}
\nonumber && -a_{11}h''+\left[a_{22}-b_{11}-a_{11}\left(\beta^{-1}\beta' +2\alpha^{-1}\alpha'\right)\right]h'  \\
\nonumber&&+\left[a_{22}\left(\beta^{-1}\beta' +2\alpha^{-1}\alpha'\right)-a_{33}+b_{22}+\frac{p+3}2\alpha^{p-2}\beta^{-2}\mathbf{q}_{tt}\right. \\
&&\left.+\frac{p+2}{2(p+3)}\alpha^{1-p}\beta^2 (a_{32})^2-\frac{p+1}{p+3}\alpha^{1-p}\beta^2 a_{32}b_{21}-\frac2{p+3} \alpha^{1-p}\beta^2 (b_{21})^2\right]h=0.
\end{eqnarray}
\end{remark}
Let $h(\theta)=\beta(\theta) u(\theta)$ in \eqref{e105}. we have:
\begin{remark}\label{rm3}
Under the condition that  $\Gamma\subset \Omega$ is a non-degenerate critical point of the functional $\mathcal{K}(\Gamma)$, we get  the following problem also has only trivial solution.
\[-\Upsilon_2 u''+\Upsilon_1 u'-\Upsilon_0 u=0, \qquad \mathrm{in} \qquad (0,1),\]
where $\Upsilon_2$, $\Upsilon_1$ and $\Upsilon_0$ are defined in \eqref{e103} and \eqref{e104}.
\end{remark}

\section{approximate solutions}\label{sect2}

To prove Theorem \ref{th1},  We only need to find a solution to \eqref{1} of the form
\[u(y)=\bar{u}_\varepsilon(y)+v(y/\varepsilon),\]
where $\bar{u}_\varepsilon(y)$ is the unique negative solution of \eqref{1}, whose properties are studied in Proposition \ref{pro1}. Then $v$ solves the following problem
\begin{equation}\label{25}
\left\{\begin{array}{ll}
-\Div(A(\varepsilon y)\nabla v)=|\bar{u}_\varepsilon(\varepsilon y)+v|^p-|\bar{u}_\varepsilon(\varepsilon y)|^p, &\mbox{in $\Omega_\varepsilon$}, \\
v=0, &\mbox{on $\partial \Omega_\varepsilon$},
\end{array}\right.
\end{equation}
where $\Omega_\varepsilon=\Omega/\varepsilon$. In this section, we construct a local approximate solution to \eqref{25}.

Let
\[\bar{\mathbf{q}}(\theta, t)=-\bar{u}_\varepsilon (\gamma(\theta)+tn(\theta))\qquad \mathrm{where} \qquad |t|<\delta_0, \quad \theta \in[0,1].\]
With the help of Lemma \ref{lm8}, we write the first equation in \eqref{25} into the following form in some neighborhood of the curve $\Gamma_\varepsilon$
\begin{equation}\label{26}
a_{11}(\varepsilon z)v_{zz}+a_{31}(\varepsilon z)v_{ss}+|v-\bar{\mathbf{q}}(\varepsilon z, \varepsilon s)|^p-|\bar{\mathbf{q}}(\varepsilon z, \varepsilon s)|^p+B_1(v)=0,
\end{equation}
where $z\in[0,1/\varepsilon]$, $s\in (-\delta_0/\varepsilon, \delta_0/\varepsilon)$ and
\[B_1(v)=2\varepsilon s a_{22}(\varepsilon z)v_{sz}+\left(\varepsilon s a_{32}(\varepsilon z)+\varepsilon^2 s^2a_{33}(\varepsilon z)\right)v_{ss} +\varepsilon b_{11}(\varepsilon z)v_z +\left(\varepsilon b_{21}(\varepsilon z)+\varepsilon^2 s b_{22}(\varepsilon z)\right)v_s +B_0(v).\]
Since we need to construct an approximate solution near $\Gamma_\varepsilon\subset \Omega_\varepsilon$, and $\mathrm{dist}(\Gamma, \Omega)>0$, we ignore the boundary condition for the moment. We will deal with the boundary condition in outer problem in Section \ref{sect5}.

Let
\begin{equation*}
\tilde{\alpha}(\theta)=\bar{\mathbf{q}}(\theta,0), \qquad \tilde{\beta}(\theta)=\frac{\left[\bar{\mathbf{q}}(\theta,0)\right]^{\frac{p-1}2}}{\left[a_{31}(\theta)\right]^{\frac12}}
\end{equation*}
and
\begin{equation}\label{80}
\alpha(\theta)=\mathbf{q}(\theta,0), \qquad \beta(\theta)=\frac{\left[\mathbf{q}(\theta,0)\right]^{\frac{p-1}2}}{\left[a_{31}(\theta)\right]^{\frac12}}.
\end{equation}
From the argument in \cite{Manna_Santra2016}, we get
\begin{equation}\label{65}
\tilde{\alpha}(\theta)=\alpha(\theta)+O(\varepsilon^2), \qquad \tilde{\alpha}'(\theta)=\alpha'(\theta)+O(\varepsilon^2), \qquad \tilde{\alpha}''(\theta)=\alpha''(\theta)+O(\varepsilon^2)
\end{equation}
and
\begin{equation}\label{e56}
\tilde\beta(\theta)=\beta(\theta)+O(\varepsilon^2), \qquad \tilde\beta'(\theta)=\beta'(\theta)+O(\varepsilon^2), \qquad \tilde\beta''(\theta)=\beta''(\theta)+O(\varepsilon^2).
\end{equation}
Let
\[v(z,s)=\tilde{\alpha}(\varepsilon z)u(z,x), \qquad \mathrm{where} \qquad x=\tilde{\beta}(\varepsilon z)s.\]
Then we have
\[v_s=\tilde{\alpha}\tilde{\beta}u_x, \qquad v_{ss}=\tilde{\alpha}\tilde{\beta}^2u_{xx},\]
\[v_z=\varepsilon \tilde{\alpha}'u+\tilde{\alpha}\left(u_z+\varepsilon \tilde\beta's u_x\right),\]
\[v_{sz}=\varepsilon \tilde{\alpha}'\tilde{\beta}u_x+\varepsilon \tilde{\alpha}\tilde{\beta}'u_x+\tilde{\alpha}\tilde{\beta}\left(u_{xz}+\varepsilon \tilde\beta's u_{xx}\right)\]
and
\begin{eqnarray*}
v_{zz}&=& \varepsilon^2 \tilde{\alpha}'' u +2\varepsilon \tilde{\alpha}'\left(u_z+\varepsilon \tilde{\beta}' s u_x\right) +\tilde{\alpha}\left[u_{zz}+2\varepsilon \tilde{\beta}'s u_{xz}+ \varepsilon^2(\tilde{\beta}')^2s^2 u_{xx} +\varepsilon^2\tilde{\beta}'' su_x  \right].
\end{eqnarray*}

It is easy to get
\[\bar{\mathbf{q}}(\varepsilon z, \varepsilon s) =\tilde{\alpha}(\varepsilon z) +\varepsilon s \bar{\mathbf{q}}_t(\varepsilon z, 0) +\frac12 \varepsilon^2 s^2 \bar{\mathbf{q}}_{tt}(\varepsilon z,0) +O(\varepsilon^3 |s|^3).\]
Thus,
\begin{eqnarray*}
&&|v-\bar{\mathbf{q}}(\varepsilon z, \varepsilon s)|^p- \left|\bar{\mathbf{q}}(\varepsilon z, \varepsilon s)\right|^p \\
&=& \tilde{\alpha}^p\left\{|u-1|^p-1 -p\varepsilon s\tilde{\alpha}^{-1}\bar{\mathbf{q}}_t\left(|u-1|^{p-2}(u-1)+1\right) \right. \\
&&\left.+\frac12\varepsilon^2 s^2\left[p(p-1)\tilde{\alpha}^{-2}\bar{\mathbf{q}}_t^2\left( |u-1|^{p-2}-1\right) -p\tilde\alpha^{-1}\bar{\mathbf{q}}_{tt}\left(|u-1|^{p-2}(u-1)+1\right)\right] \right\} \\
&& +O(\varepsilon^3|s|^3|u|^{\min\{p-3,1\}}).
\end{eqnarray*}
The equation \eqref{26} is transformed into the following one:
\begin{equation}\label{27}
S(u):=u_{xx}+a_{11}\tilde{\alpha}^{1-p}u_{zz}+|u-1|^p-1 +B_4(u)=0,
\end{equation}
where $B_4(u)=B_4^1(u)+B_4^2(u)+B_4^3(u)$, $B_4^1(u)$ and $B_4^3(u)$ are linear functions of $u$. However, $B_4^2(u)$ is a nonlinear function of $u$. More precisely,
\begin{eqnarray*}
B_4^1(u)&=& \varepsilon a_{32} \tilde{\alpha}^{1-p}\tilde{\beta} xu_{xx} +\varepsilon b_{21}\tilde{\alpha}^{1-p}\tilde\beta u_x  +2\varepsilon a_{11}\tilde{\alpha}^{-p}\tilde{\alpha}' u_z +2\varepsilon a_{11}\tilde{\alpha}^{1-p}\tilde{\beta}^{-1}\tilde{\beta}'xu_{xz}+2\varepsilon a_{22}\tilde{\alpha}^{1-p}xu_{xz}  \\
&& +\varepsilon b_{11}\tilde{\alpha}^{1-p}u_z+\varepsilon^2 a_{11}\tilde{\alpha}^{-p}\tilde\alpha'' u +2\varepsilon^2 a_{11}\tilde{\alpha}^{-p}\tilde{\beta}^{-1}\tilde{\alpha}'\tilde{\beta}' xu_x +\varepsilon^2 a_{11}\tilde{\alpha}^{1-p}\tilde{\beta}^{-2}(\tilde{\beta}')^2 x^2 u_{xx} \\
&&+\varepsilon^2 a_{11}\tilde{\alpha}^{1-p}\tilde\beta^{-1}\tilde{\beta}'' xu_x +2\varepsilon^2 a_{22}\tilde{\alpha}^{-p}\tilde{\alpha}'xu_x +2\varepsilon^2 a_{22}\tilde{\alpha}^{1-p}\tilde{\beta}^{-1}\tilde{\beta}' xu_x +2\varepsilon^2 a_{22} \tilde{\alpha}^{1-p} \tilde{\beta}^{-1}\tilde{\beta}'x^2 u_{xx} \\
&&+\varepsilon^2 a_{33}\tilde{\alpha}^{1-p}x^2 u_{xx} +\varepsilon^2 b_{11}\tilde{\alpha}^{-p}\tilde{\alpha}' u +\varepsilon^2 b_{11}\tilde{\alpha}^{1-p}\tilde{\beta}^{-1}\tilde{\beta}'xu_x +\varepsilon^2 b_{22}\tilde{\alpha}^{1-p}xu_x,
\end{eqnarray*}
\begin{eqnarray*}
B_4^2(u)&=& -p\varepsilon\tilde{\alpha}^{-1}\tilde{\beta}^{-1}\bar{\mathbf{q}}_t x\left[|u-1|^{p-2}(u-1)+1\right] +\frac{p(p-1)}2\varepsilon^2 \tilde{\alpha}^{-2}\tilde{\beta}^{-2}( \bar{\mathbf{q}}_t)^2x^2\left[|u-1|^{p-2}-1\right] \\
&& -\frac{p}2\varepsilon^2 \tilde{\alpha}^{-1}\tilde{\beta}^{-2}\bar{\mathbf{q}}_{tt}x^2\left[|u-1|^{p-2}(u-1)+1\right]+O(\varepsilon^3s^3|u|^{\min\{p-3,1\}})
\end{eqnarray*}
and
\begin{eqnarray*}
B_4^3(u)&=& \varepsilon a_{12}\tilde{\alpha}^{1-p}u_{zz} +2\varepsilon^2 a_{12}\tilde{\alpha}^{-p}\tilde{\alpha}' u_z +2\varepsilon^2 a_{12}\tilde{\alpha}^{1-p}\tilde{\beta}^{-1}\tilde{\beta}' xu_{zx} +\varepsilon^2 a_{23}\tilde{\alpha}^{1-p}\tilde{\beta} u_{zx} \\
&&+\varepsilon^2 b_{12}\tilde{\alpha}^{1-p}u_z +\varepsilon^3 a_{12}\tilde{\alpha}^{-p}\tilde{\alpha}'' u +2\varepsilon^3a_{12}\tilde{\alpha}^{-p}\tilde{\beta}^{-1}\tilde{\alpha}'\tilde{\beta}'xu_x +\varepsilon^3 a_{12}\tilde{\alpha}^{1-p}\tilde{\beta}^{-2}(\tilde{\beta}')^2x^2 u_{xx} \\
&&+\varepsilon^3 a_{12}\tilde{\alpha}^{1-p}\tilde{\beta}^{-1}\tilde{\beta}'' xu_x +\varepsilon^3 a_{23}\tilde{\alpha}^{-p}\tilde{\beta}\tilde{\alpha}'u_x +\varepsilon^3a_{23}\tilde{\alpha}^{1-p}\tilde{\beta}'u_x +\varepsilon^3 a_{23}\tilde{\alpha}^{1-p}\tilde{\beta}'xu_{xx} \\
&&+\varepsilon^3 a_{33}\tilde{\alpha}^{1-p}\tilde{\beta}^2 u_{xx} +\varepsilon^3 b_{12}\tilde{\alpha}^{-p}\tilde{\alpha}' u +\varepsilon^3 b_{12}\tilde{\alpha}^{1-p}\tilde{\beta}^{-1}\tilde{\beta}' xu_x +\varepsilon^3 b_{23}\tilde{\alpha}^{1-p}\tilde{\beta} u_x.
\end{eqnarray*}
In order to construct a local approximate solution to \eqref{27}, we introduce the parameters $\{f_j\}_{j=1}^N$ and $\{e_j\}_{j=1}^N$ satisfying:
\begin{equation}\label{66}
\|f_j\|_{H^2(0,1)}\leq C|\log\varepsilon|^2, \qquad \beta(f_{j+1}-f_j)>\frac2{\sqrt{p}}|\ln \varepsilon|-\frac4{\sqrt{p}}\ln|\ln\varepsilon|
\end{equation}
and
\begin{equation}\label{11}
\|e_j\|_*:= \|e_j\|_{L^\infty(0,1)}+\varepsilon\|e_j'\|_{L^2(0,1)}+\varepsilon^2\|e_j''\|_{L^2(0,1)}<\varepsilon^{\frac12}.
\end{equation}
Let $f_0=-\infty$ and $f_{N+1}=+\infty$, and denote
\[\mathbf{f}=(f_1, \cdots, f_N), \qquad \mathrm{and} \qquad \mathbf{e}=(e_1, \cdots, e_N).\]

Recall $w$ is the unique solution of \eqref{29}. It is well known the associated linearized eigenvalue problem of \eqref{29}
\[h''+p|w-1|^{p-2}(w-1)h=\lambda h, \quad \mathrm{in} \quad \R, \qquad h\in H^1(\R^2)\]
has a unique positive eigenvalue $\lambda_0$, with a unique even and positive eigenfunction $Z$ which we normalized so that $\int_\R Z^2 dx=1$.

We define an approximate solution of \eqref{27} by
\[\mathcal{V}(x,z)=\sum_{k=1}^N \bar{\mathcal{V}}_k(z,x),\]
where
\[\bar{\mathcal{V}}_k(z,x):=\mathcal{V}_k(z, x-\tilde{\beta}(\varepsilon z)f_k(\varepsilon z))\]
and
\begin{equation}\label{69}
\mathcal{V}_k(z,x)=w(x)+\varepsilon \varphi_k^{(1)}(\varepsilon z, x)+\varepsilon e_k(\varepsilon z) Z(x)+\varepsilon^2\varphi_k^{(2)}(\varepsilon z,x).
\end{equation}
In \eqref{69}, the function $\varphi_k^{(1)}$ and $\varphi^{(2)}_k$ will be determined in \eqref{39} and \eqref{78}, respectively.
We assume that $\mathcal{V}_k$ decays at infinity as $e^{-\sigma_1|x|}$ for any constant $\sigma_1\in (0, \sqrt{p})$,.

From direct computation,
\[\bar{\mathcal{V}}_{k,x}(z, x)=\mathcal{V}_{k,x}(z, x-\tilde{\beta} f_k), \qquad \bar{\mathcal{V}}_{k,xx}(z, x)=\mathcal{V}_{k,xx}(z, x-\tilde{\beta} f_k),\]
\[\bar{\mathcal{V}}_{k,z}(z, x)=\mathcal{V}_{k,z}(z, x-\tilde{\beta} f_k)-\varepsilon (\tilde{\beta} f_k)' \mathcal{V}_{k,x}(z, x-\tilde{\beta} f_k),\]
\begin{eqnarray*}
\bar{\mathcal{V}}_{k,zz}(z, x)&=&\mathcal{V}_{k,zz}(z, x-\tilde{\beta} f_k)-2\varepsilon(\tilde{\beta} f_k)' \mathcal{V}_{k,zx}(z, x-\tilde{\beta} f_k) \\
&&-\varepsilon^2 (\tilde{\beta} f_k)''\mathcal{V}_{k,x}(z, x-\tilde{\beta} f_k) +\varepsilon^2 |(\tilde{\beta} f_k)'|^2 \mathcal{V}_{k,xx}(z, x-\tilde{\beta} f_k)
\end{eqnarray*}
and
\[\bar{\mathcal{V}}_{k,zx}(z, x)=\mathcal{V}_{k,zx}(z, x-\tilde{\beta} f_k)-\varepsilon (\tilde{\beta} f_k)' \mathcal{V}_{k,xx}(z, x-\tilde{\beta} f_k).\]

Denote
\begin{equation}\label{28}
S(\bar{\mathcal{V}}_k)=\tilde{S}(\mathcal{V}_k)(z, x-\tilde{\beta} f_k),
\end{equation}
where
\[\tilde{S}(u) = u_{xx}+a_{11}\tilde{\alpha}^{1-p}u_{zz}+|u-1|^p-1+B_3(u), \]
and $B_3(u)=B_3^1(u)+B_3^2(u)+B_3^3(u)$. More precisely
\begin{eqnarray*}
B_3^1(u)&=& -2\varepsilon a_{11}\tilde{\alpha}^{1-p}(\tilde{\beta} f_k)'u_{zx}-\varepsilon^2 a_{11}\tilde{\alpha}^{1-p}(\tilde{\beta} f_k)'' u_x +\varepsilon^2 a_{11}\tilde{\alpha}^{1-p}|(\tilde{\beta} f_k)'|^2 u_{xx} +\varepsilon b_{21}\tilde{\alpha}^{1-p}\tilde{\beta} u_x\\
&& +\varepsilon a_{32} \tilde{\alpha}^{1-p}\tilde{\beta}(x+\tilde{\beta} f_k)u_{xx} +2\varepsilon a_{11}\tilde{\alpha}^{-p}\tilde{\alpha}' u_z -2\varepsilon^2 a_{11}\tilde{\alpha}^{-p}\tilde{\alpha}'(\tilde{\beta} f_k)'u_x\\
&& +2\varepsilon a_{11}\tilde{\alpha}^{1-p}\tilde{\beta}^{-1}\tilde{\beta}'(x+\tilde{\beta} f_k)u_{zx} -2\varepsilon^2 a_{11}\tilde{\alpha}^{1-p}\tilde{\beta}^{-1}\tilde{\beta}'(\tilde{\beta} f_k)'(x+\tilde{\beta} f_k)u_{xx} \\
&& +2\varepsilon a_{22}\tilde{\alpha}^{1-p}(x+\tilde{\beta} f_k)u_{zx} -2\varepsilon^2 a_{22}\tilde{\alpha}^{1-p}(\tilde{\beta} f_k)'(x+\tilde{\beta} f_k)u_{xx}+ \varepsilon b_{11}\tilde{\alpha}^{1-p}u_z \\
&& -\varepsilon^2 b_{11}\tilde{\alpha}^{1-p}(\tilde{\beta} f_k)' u_x +\varepsilon^2 a_{11}\tilde{\alpha}^{-p}\tilde{\alpha}'' u +2\varepsilon^2 a_{11}\tilde{\alpha}^{-p}\tilde{\beta}^{-1}\tilde{\alpha}'\tilde{\beta}'(x+\tilde{\beta} f_k)u_x \\
&& +\varepsilon^2 a_{11}\tilde{\alpha}^{1-p}\tilde{\beta}^{-2}(\tilde{\beta}')^2 (x+\tilde{\beta} f_k)^2u_{xx} +\varepsilon^2 a_{11}\tilde{\alpha}^{1-p}\tilde{\beta}^{-1}\tilde{\beta}''(x+\tilde{\beta} f_k)u_x  \\
&& +2\varepsilon^2 a_{22} \tilde{\alpha}^{-p} \tilde{\alpha}'(x+\tilde{\beta} f_k)u_x  +2\varepsilon^2 a_{22}\tilde{\alpha}^{1-p}\tilde{\beta}^{-1}\tilde{\beta}'(x+\tilde{\beta} f_k) u_x \\
&& +\varepsilon^2 a_{33}\tilde{\alpha}^{1-p}(x+\tilde{\beta} f_k)^2 u_{xx}+2\varepsilon^2 a_{22}\tilde{\alpha}^{1-p}\tilde{\beta}^{-1}\tilde{\beta}'(x+\tilde{\beta} f_k)^2 u_{xx}+\varepsilon^2 b_{11}\tilde{\alpha}^{-p}\tilde{\alpha}' u \\
&& +\varepsilon^2 b_{11}\tilde{\alpha}^{1-p}\tilde{\beta}^{-1}\tilde{\beta}'(x+\tilde{\beta} f_k)u_x  +\varepsilon^2 b_{22}\tilde{\alpha}^{1-p}(x+\tilde{\beta} f_k)u_x,
\end{eqnarray*}
\begin{eqnarray*}
B_3^2(u)&=& -p\varepsilon \tilde{\alpha}^{-1}\tilde{\beta}^{-1}\bar{\mathbf{q}}_t (x+\tilde{\beta} f_k)\left[|u-1|^{p-2}(u-1)+1\right] \\
&&+\frac{p(p-1)}2\varepsilon^2 \tilde{\alpha}^{-2}\tilde{\beta}^{-2}\bar{\mathbf{q}}_t^2(x+\tilde{\beta} f_k)^2\left[|u-1|^{p-2}-1\right] \\
&&-\frac{p}2\varepsilon^2 \tilde{\alpha}^{-1}\tilde{\beta}^{-2}\bar{\mathbf{q}}_{tt}(x+\tilde{\beta} f_k)^2 \left[|u-1|^{p-2}(u-1)+1\right] +O(\varepsilon^3|x+\tilde{\beta} f_k|^3 |u|^{\min\{p-3,1\}})
\end{eqnarray*}
and
\begin{eqnarray*}
B_3^3(u)&=&\varepsilon a_{12} \tilde{\alpha}^{1-p} u_{zz} -2\varepsilon^2 a_{12}\tilde{\alpha}^{1-p}(\tilde{\beta} f_k)' u_{zx} +2\varepsilon^2a_{12}\tilde{\alpha}^{-p}\tilde{\alpha}' u_z+\varepsilon^2a_{23}\tilde{\alpha}^{1-p}\tilde{\beta} u_{zx} \\
&& +\varepsilon^2 b_{12}\tilde{\alpha}^{1-p}u_z +2\varepsilon^2 a_{12}\tilde{\alpha}^{1-p}\tilde{\beta}^{-1}\tilde{\beta}'(x+\tilde{\beta} f_k)u_{zx} +\varepsilon^3a_{12}\tilde{\alpha}^{1-p}|(\tilde{\beta} f_k)'|^2u_{xx} \\
&& -\varepsilon^3 a_{12}\tilde{\alpha}^{1-p}(\tilde{\beta} f_k)''u_x-2\varepsilon^3 a_{12}\tilde{\alpha}^{-p}\tilde{\alpha}'(\tilde{\beta} f_k)' u_x -2\varepsilon^3 a_{12}\tilde{\alpha}^{1-p}\tilde{\beta}^{-1}\tilde{\beta}'(\tilde{\beta} f_k)'(x+\tilde{\beta} f_k)u_{xx} \\
&& -\varepsilon^3 a_{23}\tilde{\alpha}^{1-p}\tilde{\beta} (\tilde{\beta} f_k)'u_{xx}-\varepsilon^3 b_{12}\tilde{\alpha}^{1-p}(\tilde{\beta} f_k)' u_x +\varepsilon^3 a_{12}\tilde{\alpha}^{-p}\tilde{\alpha}'' u \\
&& +2\varepsilon^3 a_{12}\tilde{\alpha}^{-p}\tilde{\beta}^{-1}\tilde{\alpha}'\tilde{\beta}'(x+\tilde{\beta} f_k)u_x +\varepsilon^3 a_{12}\tilde{\alpha}^{1-p}\tilde{\beta}^{-2}(\tilde{\beta}')^2(x+\tilde{\beta} f_k)^2u_{xx} \\
&& +\varepsilon^3 a_{12}\tilde{\alpha}^{1-p}\tilde{\beta}^{-1}\tilde{\beta}''(x+\tilde{\beta} f_k)u_x +\varepsilon^3 a_{23}\tilde{\alpha}^{-p}\tilde{\beta} \tilde{\alpha}' u_x +\varepsilon^3 a_{23}\tilde{\alpha}^{1-p}\tilde{\beta}' u_x \\
&&+\varepsilon^3 a_{23}\tilde{\alpha}^{1-p}\tilde{\beta}'(x+\tilde{\beta} f_k)u_{xx} +\varepsilon^3 a_{33}\tilde{\alpha}^{1-p}\tilde{\beta}^2u_{xx} +\varepsilon^3 b_{12}\tilde{\alpha}^{-p}\tilde{\alpha}' u \\
&& +\varepsilon^3 b_{12}\tilde{\alpha}^{1-p}\tilde{\beta}^{-1}\tilde{\beta}'(x+\tilde{\beta} f_k) u_x+\varepsilon^3 b_{23}\tilde{\alpha}^{1-p}\tilde{\beta} u_x.
\end{eqnarray*}
In the expression of $B_3^3(u)$, the functions $a_{ij}$'s and $b_{ij}$'s take values at $(\varepsilon z,  \tilde{\beta}^{-1} x+ f_k)$.

Define
\[\mathcal{U}_k=\left\{(z,x)\in \mathfrak{S} \;:\; \frac12\tilde{\beta}(\varepsilon z)\left[f_{k-1}(\varepsilon z)+f_k(\varepsilon z)\right]\leq x\leq \frac12\tilde{\beta}(\varepsilon z)\left[f_k(\varepsilon z)+f_{k+1}(\varepsilon z)\right]\right\}.\]
We use the method in \cite{Yang_Yang2013, delPino_Kowalczyk_Wei2008} to estimate the nonlinear terms in $S(\mathcal{V})$. From \eqref{e56} and \eqref{66}, we derive that for $(z,x)\in \mathcal{U}_j$,
\begin{eqnarray*}
|\mathcal{V}-1|^p-1&=&|\bar{\mathcal{V}}_j-1|^p-1+p|\bar{\mathcal{V}}_j-1|^{p-2}(\bar{\mathcal{V}}_j-1)\sum_{k\not=j}\bar{\mathcal{V}}_k+O\left(\left(\sum_{k\not=j} \bar{\mathcal{V}}_k\right)^2\right) \\
&=& \sum_{k=1}^N \left(|\bar{\mathcal{V}}_k-1|^p-1\right) +\left[p|\bar{\mathcal{V}}_j-1|^{p-2}(\bar{\mathcal{V}}_j-1)\sum_{k\not=j}\bar{\mathcal{V}}_k -\sum_{k\not=j}(|\bar{\mathcal{V}}_k-1|^p-1)\right]     +O\left(\sum_{k\not=j} \bar{\mathcal{V}}_k^2\right) \\
&=& \sum_{k=1}^N \left(|\bar{\mathcal{V}}_k-1|^p-1\right) +p\left[|\bar{\mathcal{V}}_j-1|^{p-2}(\bar{\mathcal{V}}_j-1)+1\right] \sum_{k\not=j}\bar{\mathcal{V}}_k +\max_{k\not=j}O\left(e^{-(2\sqrt{p}-\tilde{\sigma})|x-\beta f_k|}\right),
\end{eqnarray*}
where $\tilde{\sigma}>0$ is a  constant small enough.
From the same method, we also obtain the following estimates for $(z,x)\in \mathcal{U}_j$:
\begin{eqnarray*}
&&|\mathcal{V}-1|^{p-2}(\mathcal{V}-1)+1 \\
&=&\sum_{k=1}^N\left[|\bar{\mathcal{V}}_k-1|^{p-2}(\bar{\mathcal{V}}_k-1)+1\right] + (p-1)\left(|\bar{\mathcal{V}}_j-1|^{p-2}-1\right)\sum_{k\not=j}\bar{\mathcal{V}}_k +\max_{k\not=j}O\left(e^{-(2\sqrt{p}-\tilde{\sigma})|x-\beta f_k|}\right) \\
&=& \sum_{k=1}^N\left[|\bar{\mathcal{V}}_k-1|^{p-2}(\bar{\mathcal{V}}_k-1)+1\right]+ \max_{k\not=j}O\left(e^{-(2\sqrt{p}-\tilde{\sigma})|x-\beta f_k|}+e^{-(\sqrt{p}-\tilde{\sigma})(|x-\beta f_k|+|x-\beta f_j|)}\right)
\end{eqnarray*}
and
\[|\mathcal{V}-1|^{p-2}-1=\sum_{k=1}^N \left[|\bar{\mathcal{V}}_k-1|^{p-2}-1\right]+\max_{k\not=j}O\left(e^{-(\sqrt{p}-\tilde{\sigma})|x-\beta f_k|}\right).\]

Hence we arrive at
\begin{equation}\label{67}
S(\mathcal{V}) = \sum_{j=1}^N\left[S(\bar{\mathcal{V}}_j) +p\chi_{\mathcal{U}_j}\left(|\bar{\mathcal{V}}_j-1|^{p-2}( \bar{\mathcal{V}}_j-1)+1\right)\sum_{k\not=j}\bar{\mathcal{V}}_k\right]  +\sum_{j=1}^N \tilde{\theta}_j,
\end{equation}
where $\chi_{\mathcal{U}_j}$ is a characteristic function which equals to $1$ if $x\in \mathcal{U}_j$ and equals to $0$ if $x\not\in \mathcal{U}_j$, and
\begin{eqnarray}\label{e62}
\nonumber \tilde{\theta}_j&=&\chi_{\mathcal{U}_j}\left[\max_{k\not=j}O\left(e^{-(2\sqrt{p}-\tilde{\sigma})|x-\beta f_k|}\right)+\max_{k\not=j}O\left(\varepsilon |x| e^{-(\sqrt{p}-\tilde{\sigma})(|x-\beta f_k|+|x-\beta f_j|)}\right)\right. \\
&& \left. + \max_{k\not=j}O\left(\varepsilon^2 |x|^2 e^{-(\sqrt{p}-\tilde{\sigma})|x-\beta f_k|}\right)\right].
\end{eqnarray}
Next, we expand  the term $\tilde{S}(\mathcal{V}_k)$. It is obvious that
\[|\mathcal{V}_k-1|^{p-2}(\mathcal{V}_k-1)+1=|w-1|^{p-2}(w-1)+1+(p-1)\varepsilon |w-1|^{p-2}(\varphi_k^{(1)}+e_k Z)+O(\varepsilon^2e^{-(\sqrt{p}-\tilde{\sigma})|x|}),\]
and
\begin{align}\label{68}
 \nonumber |\mathcal{V}_k-1|^p-1 =&|w-1|^p-1+p\varepsilon |w-1|^{p-2}(w-1)(\varphi_k^{(1)}+e_k Z) +p\varepsilon^2 |w-1|^{p-2}(w-1)\varphi^{(2)}_k \\
   & +\frac{p(p-1)}2\varepsilon^2|w-1|^{p-2}(\varphi_k^{(1)}+e_k Z)^2 +O(\varepsilon^3e^{-(\sqrt{p}-\tilde{\sigma})|x|}).
\end{align}
From \eqref{65}, \eqref{e56}, \eqref{69} and  direct computation, we get
\begin{eqnarray*}
B_3(\mathcal{V}_k)&=& \varepsilon\left\{a_{32}\alpha^{1-p}\beta xw_{xx}+b_{21}\alpha^{1-p}\beta w_x -p\alpha^{-1}\beta^{-1}\mathbf{q}_t x\left[|w-1|^{p-2}(w-1)+1\right]\right.\\
&& \left.+a_{32}\alpha^{1-p}\beta^2 f_k w_{xx}-p\alpha^{-1}\mathbf{q}_t f_k\left[|w-1|^{p-2}(w-1)+1\right]\right\} + \varepsilon^2 \left\{ -a_{11}\alpha^{1-p}(\beta f_k)''w_x \right.\\
&& -2a_{11}\alpha^{-p}\alpha'(\beta f_k)' w_x -2a_{11}\alpha^{1-p}\beta^{-1}\beta'(\beta f_k)'xw_{xx}   +b_{11}\alpha^{1-p}\beta' f_k w_x+b_{22}\alpha^{1-p}\beta f_k w_x\\
&&  -2a_{22}\alpha^{1-p}(\beta f_k)'xw_{xx} -b_{11}\alpha^{1-p}(\beta f_k)'w_x+2a_{11}\alpha^{-p}\alpha'\beta' f_k w_x +2a_{11}\alpha^{1-p}\beta^{-1}(\beta')^2 f_k xw_{xx}\\
&& +a_{11}\alpha^{1-p}\beta'' f_k w_x+ 2a_{22}\alpha^{-p}\alpha'\beta f_k w_x +p(p-1)\alpha^{-2}\beta^{-1}\mathbf{q}_t^2 f_k x[|w-1|^{p-2}-1]\\
&& -p\alpha^{-1}\beta^{-1}\mathbf{q}_{tt}f_k x\left[|w-1|^{p-2}(w-1)+1\right]+2a_{22}\alpha^{1-p}\beta' f_k w_x +2a_{33}\alpha^{1-p}\beta f_k xw_{xx} \\
&& +4a_{22}\alpha^{1-p}\beta' f_k xw_{xx} -p(p-1)\alpha^{-1}\beta^{-1}\mathbf{q}_t (x+\beta f_k)|w-1|^{p-2}\varphi_k^{(1)} +a_{32}\alpha^{1-p}\beta(x+\beta f_k)\varphi^{(1)}_{k,xx}\\
&& +b_{21}\alpha^{1-p}\beta \varphi^{(1)}_{k,x}-2\varepsilon a_{11}\alpha^{1-p}(\beta f_k)'e_k' Z'+2\varepsilon a_{11}\alpha^{1-p}\beta'e_k'f_kZ'+2\varepsilon a_{22}\alpha^{1-p}\beta e_k' f_k Z'\\
&& \left. + \hat{b}_2(e_k) + \check{b}_2(e_k, f_k)\right\} +O(\varepsilon^3|\log\varepsilon|^8e^{-\gamma_1|x|}).
\end{eqnarray*}
In this expression, $\gamma_1>0$ is a constant and  $\hat{b}_2$ is the  combinations of $e_k$ and some known functions, which is odd on $x$. However $\check{b}_2(e_k, f_k)$ is even on $x$ and it is a combinations of $f_k$, $e_k$ and some known functions. In the argument below, we consider $e_k$, $\varepsilon e_k'$ and $\varepsilon^2 e_k''$ as the same order( see \eqref{11}).

From  \eqref{67} and \eqref{68}, we get
\begin{eqnarray*}
\nonumber\tilde{S}(\mathcal{V}_k)&=&\varepsilon\left(\lambda_0 e_k +\varepsilon^2a_{11} \alpha^{1-p}e_k''\right) Z+\varepsilon\left[\varphi^{(1)}_{k,xx}+p|w-1|^{p-2}(w-1)\varphi^{(1)}_k\right] \\
\nonumber&& +\frac{p(p-1)}2\varepsilon^2 |w-1|^{p-2}(\varphi^{(1)}_k +e_k Z)^2 +\varepsilon^2\left[\varphi^{(2)}_{k,xx}+p|w-1|^{p-2}(w-1)\varphi^{(2)}_k\right] \\
\nonumber&&+B_3(\mathcal{V}_k) +O(\varepsilon^3|\log\varepsilon|^8 e^{-\gamma_1|x|}) \\
\nonumber&=& \varepsilon\left(\varepsilon^2 a_{11}\alpha^{1-p}e_k''+\lambda_0e_k\right)Z(x) +\varepsilon \left[\varphi^{(1)}_{k,xx}+p|w-1|^{p-2}(w-1)\varphi^{(1)}_k\right] +\varepsilon\left\{a_{32}\alpha^{1-p}\beta xw_{xx}\right.\\
\nonumber&& +b_{21}\alpha^{1-p}\beta w_x -p\alpha^{-1}\beta^{-1}\mathbf{q}_t x\left[|w-1|^{p-2}(w-1)+1\right] +a_{32}\alpha^{1-p}\beta^2 f_k w_{xx}\\
\nonumber&& \left.-p\alpha^{-1}\mathbf{q}_t f_k\left[|w-1|^{p-2}(w-1)+1\right]\right\} +\varepsilon^2 \left[\varphi^{(2)}_{k,xx} +p|w-1|^{p-2}(w-1)\varphi^{(2)}_k\right]  \\
&&+ \varepsilon^2 \mathbf{A}_k +O(\varepsilon^3|\log\varepsilon|^8e^{-\gamma_1|x|}),
\end{eqnarray*}
where
\begin{eqnarray*}
\mathbf{A}_k &= &  -a_{11}\alpha^{1-p}(\beta f_k)''w_x -2a_{11}\alpha^{-p}\alpha'(\beta f_k)' w_x -2a_{11}\alpha^{1-p}\beta^{-1}\beta'(\beta f_k)'xw_{xx}   +b_{11}\alpha^{1-p}\beta' f_k w_x\\
&&+b_{22}\alpha^{1-p}\beta f_k w_x -2a_{22}\alpha^{1-p}(\beta f_k)'xw_{xx} -b_{11}\alpha^{1-p}(\beta f_k)'w_x+2a_{11}\alpha^{-p}\alpha'\beta' f_k w_x\\
&&  +2a_{11}\alpha^{1-p}\beta^{-1}(\beta')^2 f_k xw_{xx} +a_{11}\alpha^{1-p}\beta'' f_k w_x+ 2a_{22}\alpha^{-p}\alpha'\beta f_k w_x\\
&& +p(p-1)\alpha^{-2}\beta^{-1}\mathbf{q}_t^2 f_k x[|w-1|^{p-2}-1]-p\alpha^{-1}\beta^{-1}\mathbf{q}_{tt}f_k x\left[|w-1|^{p-2}(w-1)+1\right]\\
&&+2a_{22}\alpha^{1-p}\beta' f_k w_x +2a_{33}\alpha^{1-p}\beta f_k xw_{xx} +4a_{22}\alpha^{1-p}\beta' f_k xw_{xx}+a_{32}\alpha^{1-p}\beta(x+\beta f_k)\varphi^{(1)}_{k,xx}\\
&&-p(p-1)\alpha^{-1}\beta^{-1}\mathbf{q}_t (x+\beta f_k)|w-1|^{p-2}\varphi^{(1)}_k +b_{21}\alpha^{1-p}\beta \varphi^{(1)}_{k,x}-2\varepsilon a_{11}\alpha^{1-p}(\beta f_k)'e_k' Z'\\
&&+2\varepsilon a_{11}\alpha^{1-p}\beta'e_k'f_kZ'+2\varepsilon a_{22}\alpha^{1-p}\beta e_k' f_k Z'+\frac{p(p-1)}2 |w-1|^{p-2}(\varphi^{(1)}_k +e_k Z)^2  \\
&&+ \hat{b}_2(e_k) + \check{b}_2(e_k, f_k).
\end{eqnarray*}
In order to get a better approximate solution, we need to cancel the terms of the order $\varepsilon$ in $\tilde{S}(\mathcal{V}_k)$. We solve the following equation
\begin{eqnarray}\label{30}
\nonumber&& -\varphi^{(1)}_{k,xx}-p|w-1|^{p-2}(w-1)\varphi^{(1)}_k  \\
\nonumber &=& a_{32} \alpha^{1-p}\beta xw_{xx} +b_{21}\alpha^{1-p}\beta w_x -p\alpha^{-1}\beta^{-1}\mathbf{q}_t x\left[|w-1|^{p-2}(w-1)+1\right] \\
&&+a_{32}\alpha^{1-p}\beta^2 f_k w_{xx}-p\alpha^{-1}\mathbf{q}_t f_k \left[|w-1|^{p-2}(w-1)+1\right]=: \tilde{R}.
\end{eqnarray}

From \eqref{29} and direct computation, we get
\begin{equation}\label{31}
\int_\R w= \frac{p+3}{2p} \int_\R w_x^2 dx, \qquad  \int_\R xw_{xx}w_xdx =-\frac12 \int_\R w_x^2,
\end{equation}
and
\begin{equation}\label{32}
\int_\R \left[|w-1|^{p-2}(w-1)+1\right]xw_x dx=-\int_\R w=-\frac{p+3}{2p}\int_\R w_x^2 dx.
\end{equation}
It is well known that \eqref{30} is uniquely solvable provided
\begin{equation}\label{e92}
\int_\R \tilde{R}w_xdx=0.
\end{equation}
With the help of  \eqref{31} and \eqref{32}, we see \eqref{e92} is equivalent to
\[\mathbf{q}_t =\frac1{p+3}a_{32}\alpha^{2-p}\beta^2- \frac2{p+3}b_{21}\alpha^{2-p}\beta^2.\]
From Remark \ref{rm1}, we get \eqref{30} is uniquely solvable.

Let $w_0$ be the unique solution of the following problem
\begin{equation}\label{34}
-w_{0,xx}-p|w-1|^{p-2}(w-1)w_0 =w_x+\frac{2p}{p+3}x\left[|w-1|^{p-2}(w-1)+1\right], \quad \int_\R w_0w_x dx=0
\end{equation}
and $w_1$ be the unique solution of the following problem
\begin{equation}\label{37}
-w_{1,xx}-p|w-1|^{p-2}(w-1)w_1=xw_{xx}-\frac{p}{p+3}x\left[|w-1|^{p-2}(w-1)+1\right], \quad \int_\R w_1w_x dx=0.
\end{equation}
Denote
\[w_2=-\frac12 xw_x\qquad \mathrm{and} \qquad w_3=\frac{1-p}{2p}xw_x -\frac1p w.\]
They solve the following problems, respectively:
\begin{equation}\label{35}
-w_{2,xx}-p|w-1|^{p-2}(w-1)w_2=w_{xx}, \quad \int_\R w_2w_x dx=0
\end{equation}
and
\begin{equation}\label{36}
-w_{3,xx}-p|w-1|^{p-2}(w-1)w_3=|w-1|^{p-2}(w-1)+1, \quad \int_\R w_3w_x dx=0.
\end{equation}
Hence the solution of \eqref{30} is represented by
\begin{equation}\label{39}
\varphi^{(1)}_k(\varepsilon z,x)=b_{21}\alpha^{1-p}\beta w_0+a_{32}\alpha^{1-p}\beta w_1+a_{32}\alpha^{1-p}\beta^2 f_k w_2 -p\alpha^{-1}\mathbf{q}_t f_k w_3.
\end{equation}
With \eqref{39}, we succeed to cancel the terms of the order $\varepsilon$ in $\tilde{S}(\mathcal{V}_k)$. There holds
\begin{equation}\label{40}
\tilde{S}(\mathcal{V}_k)=\varepsilon\left(\varepsilon^2 a_{11}\alpha^{1-p}e_k''+\lambda_0 e_k\right)Z+\varepsilon^2\left[\varphi^{(2)}_{k,xx} +p|w-1|^{p-2}(w-1)\varphi^{(2)}_k\right] +\varepsilon^2 \mathbf{A}_k +O(\varepsilon^3|\log\varepsilon|^8e^{-\gamma_1|x|}).
\end{equation}
The sums of odd terms and that of even part terms in $\mathbf{A}_k$ are denoted by $\hat{\mathbf{A}}_k$ and $\check{\mathbf{A}}_k$, respectively. Moreover, we have
\begin{eqnarray}\label{75}
\nonumber\hat{\mathbf{A}}_k&=& -a_{11}\alpha^{1-p}(\beta f_k)'' w_x -2a_{11}\alpha^{-p}\alpha'(\beta f_k)' w_x -2a_{11}\alpha^{1-p}\beta^{-1}\beta'(\beta f_k)'xw_{xx} -2a_{22}\alpha^{1-p}(\beta f_k)'xw_{xx} \\
\nonumber&& -b_{11}\alpha^{1-p}(\beta f_k)' w_x +2a_{11}\alpha^{-p}\alpha' \beta' f_k w_x +2 a_{11}\alpha^{1-p}\beta^{-1}(\beta')^2 f_k xw_{xx} +a_{11}\alpha^{1-p}\beta'' f_k w_x \\
\nonumber&& +2a_{22}\alpha^{-p}\alpha' \beta f_k w_x +p(p-1)\beta^{-1}\alpha^{-2}(\mathbf{q}_t)^2 f_k x\left[|w-1|^{p-2}-1\right] +2 a_{22}\alpha^{1-p}\beta' f_k w_x \\
\nonumber&& -p\beta^{-1}\alpha^{-1}\mathbf{q}_{tt}f_kx\left[|w-1|^{p-2}(w-1)+1\right] +2a_{33} \alpha^{1-p}\beta f_k xw_{xx} +4a_{22}\alpha^{1-p}\beta' f_k xw_{xx}  \\
\nonumber&& +b_{11}\alpha^{1-p}\beta' f_k w_x +b_{22}\alpha^{1-p}\beta f_kw_x + b_{21}a_{32}\alpha^{2-2p}\beta^3 f_k \left[w_{0,xx} +p(p-1)|w-1|^{p-2}w_0w_2\right] \\
\nonumber&& +(a_{32})^2 \alpha^{2-2p}\beta^3 f_k \left[w_{1,xx} +p(p-1)|w-1|^{p-2}w_1w_2\right] +(a_{32})^2 \alpha^{2-2p}\beta^3 f_k xw_{2,xx} \\
\nonumber&& -pa_{32}\alpha^{-p}\beta \mathbf{q}_t f_k \left[xw_{3,xx}+p(p-1)|w-1|^{p-2}w_1w_3\right] +b_{21}a_{32}\alpha^{2-2p}\beta^3 f_k w_{2,x}  \\
\nonumber&& -pb_{21}\alpha^{-p}\beta \mathbf{q}_t f_k w_{3,x} -p^2(p-1)b_{21}\alpha^{-p}\beta \mathbf{q}_t f_k |w-1|^{p-2}w_0w_3 \\
\nonumber&& -p(p-1)a_{32}\alpha^{-p}\beta \mathbf{q}_t f_k |w-1|^{p-2}xw_2 +p^2(p-1)\alpha^{-2}\beta^{-1}(\mathbf{q}_t)^2 f_k |w-1|^{p-2}xw_3 \\
\nonumber&& -p(p-1)b_{21}\alpha^{-p}\beta \mathbf{q}_t f_k|w-1|^{p-2}w_0 -p(p-1)a_{32}\alpha^{-p}\beta \mathbf{q}_t f_k|w-1|^{p-2}w_1  \\
&& -2\varepsilon a_{11}\alpha^{1-p}(\beta f_k)' e_k' Z_x +2\varepsilon a_{11}\alpha^{1-p}\beta' f_k e_k' Z_x +2\varepsilon a_{22}\alpha^{1-p}\beta f_k e_k' Z_x +\hat{b}_2(e_k).
\end{eqnarray}
To cancel the even terms of the order $\varepsilon^2$, we consider the following problem
\begin{equation}\label{78}
\varphi^{(2)}_{k,xx}+p|w-1|^{p-2}(w-1)\varphi^{(2)}_k=-\check{\mathbf{A}}_k.
\end{equation}
It is uniquely solvable, since $\int_\R\check{\mathbf{A}}_kw_xdx=0$.

Hence we get
\begin{equation}\label{49}
\tilde{S}(\mathcal{V}_k)=\varepsilon \left[\lambda_0 e_k +\varepsilon^2 a_{11}\alpha^{1-p}e_k''\right]Z+\varepsilon^2 \hat{\mathbf{A}}_k+\varepsilon^3\mathbf{B}_k,
\end{equation}
where $|\mathbf{B}_k|\leq C|\log\varepsilon|^8 e^{-\gamma_1|x|}$.

Let
\begin{equation}\label{e109}
w_j(z,x)=w(x-\beta f_j),\qquad \mathrm{and}  \quad Z_j(z,x)=Z(x-\beta f_j), \quad \mathrm{for} \quad j=1,2,\cdots, N.
\end{equation}
In the case of $(z,x)\in \mathcal{U}_j$, we get the following estimate from Lemma \ref{lm9}
\begin{eqnarray*}
&&\left[|\bar{\mathcal{V}}_j-1|^{p-2}(\bar{\mathcal{V}}_j-1)+1\right]\sum_{k\not=j}\bar{\mathcal{V}}_k \\
&=&\left[|\bar{\mathcal{V}}_j-1|^{p-2}(\bar{\mathcal{V}}_j-1)+1\right]\left(\bar{\mathcal{V}}_{j-1}+\bar{\mathcal{V}}_{j+1}\right) +O(\varepsilon^{3-\mu}e^{-(\sqrt{p}-\tilde{\sigma})|x-\beta f_j|}) \\
&=&\left[|w_j-1|^{p-2}(w_j-1)+1\right](w_{j-1}+w_{j+1})+\max_{k\not=j}O(\varepsilon e^{-(\sqrt{p}-\tilde{\sigma})(|x-\beta f_j|+|x-\beta f_k|)}) +O(\varepsilon^{3-\mu}e^{-\tilde{\sigma}|x-\beta f_j|}) \\
&=& \alpha_p\left[|w_j-1|^{p-2}(w_j-1)+1\right]\left(e^{-\sqrt{p}(x-\beta f_{j-1})} + e^{\sqrt{p}(x-\beta f_{j+1})}\right)+O(\varepsilon^{3-\mu}e^{-(\sqrt{p}-\tilde{\sigma})|x-\beta f_j|}) \\
&& +\max_{k\not=j}O(\varepsilon e^{-(\sqrt{p}-\tilde{\sigma})(|x-\beta f_j|+|x-\beta f_k|)}).
\end{eqnarray*}
Therefore, we obtain
\begin{eqnarray}\label{74}
\nonumber S(\mathcal{V})&=& \sum_{j=1}^N \varepsilon \left(\lambda_0 e_j +\varepsilon^2 a_{11}\alpha^{1-p}e_j''\right)Z_j +\sum_{j=1}^N\varepsilon^2 \hat{\mathbf{A}}_j(z,x-\beta f_j) +\varepsilon^3 \mathbf{B}_j(z,x-\beta f_j)+ \sum_{j=1}^N \tilde{\theta}_j \\
&& +\sum_{j=1}^N \chi_{\mathcal{U}_j}\left[p\alpha_p\left[|w_j-1|^{p-2}(w_j-1)+1\right]\left(e^{-\sqrt{p}(x-\beta f_{j-1})}+e^{\sqrt{p}(x-\beta f_{j+1})}\right)+\tilde{\theta}_{j2}\right],
\end{eqnarray}
where
\begin{equation}\label{e63}
\tilde{\theta}_{j2}=\chi_{\mathcal{U}_j}\left[O(\varepsilon^{3-\mu}e^{-(\sqrt{p}-\tilde{\sigma})|x-\beta f_j|})   +\max_{k\not=j}O(\varepsilon e^{-(\sqrt{p}-\tilde{\sigma})(|x-\beta f_j|+|x-\beta f_k|)})\right].
\end{equation}

Denote
\[\mathcal{E}_1= \sum_{j=1}^N \varepsilon \left(\lambda_0 e_j +\varepsilon^2 a_{11}\alpha^{1-p}e_j''\right)Z_j \qquad \mathrm{and} \qquad \mathcal{E}_2=S(\mathcal{V})-\mathcal{E}_1.\]
Let $\mathfrak{S}=\left\{(z,x): z\in \Gamma_\varepsilon , x\in \R\right\}$. From direct computation(see \cite{Wang_Wei_Yang2011} for the method), we get
\begin{equation}\label{e33}
\|\mathcal{E}_2\|_{L^2(\mathfrak{S})}\leq C\varepsilon^{\frac32-\mu},
\end{equation}
where $\mu>0$ is constant small enough.
 For example, a typical term is
\[K_1=p\alpha_p\chi_{\mathcal{U}_j}\left[|w_j-1|^{p-2}(w_j-1)+1\right]\left(e^{-\sqrt{p}(x-\beta f_{j-1})}+e^{\sqrt{p}(x-\beta f_{j+1})}\right).\]
It is easy to get
\begin{eqnarray*}
&&\|K_1\|_{L^2(\mathfrak{S})}^2 \\
&\leq& C\int_0^{1/\varepsilon} dz \int_{\frac{\beta}2(f_j+f_{j-1})}^{\frac{\beta}2(f_j+f_{j+1})}\left|[|w_j-1|^{p-2}(w_j-1)+1]\left(e^{-\sqrt{p}(x-\beta f_{j-1})}+e^{\sqrt{p}(x-\beta f_{j+1})}\right)\right|^2 dx \\
&=& C \int_0^{1/\varepsilon} dz \int_{\frac{\beta}2(f_{j-1}-f_j)}^{\frac{\beta}2(f_{j+1}-f_j)}\left|[|w-1|^{p-2}(w-1)+1]\left(e^{-\sqrt{p}\beta(f_j- f_{j-1})}e^{-\sqrt{p}t}+e^{\sqrt{p}\beta(f_{j+1}- f_j)}e^{\sqrt{p}t}\right)\right|^2 dx \\
&\leq& C \int_0^{1/\varepsilon} dz \int_{\frac{\beta}2(f_{j-1}-f_j)}^{\frac{\beta}2(f_{j+1}-f_j)}\left|[|w-1|^{p-2}(w-1)+1]e^{-\sqrt{p}\beta(f_j- f_{j-1})}e^{-\sqrt{p}t}\right|^2 dt \\
&&+ \int_0^{1/\varepsilon} dz \int_{\frac{\beta}2(f_{j-1}-f_j)}^{\frac{\beta}2(f_{j+1}-f_j)}\left|[|w-1|^{p-2}(w-1)+1]e^{-\sqrt{p}\beta(f_{j+1}- f_j)}e^{-\sqrt{p}t}\right|^2 dt \\
&\leq& C\varepsilon^3|\log\varepsilon|^{2q}.
\end{eqnarray*}
Hence $\|K_1\|_{L^2(\mathfrak{S})}\leq C\varepsilon^{3/2}|\log\varepsilon|^q$.

In fact, $\mathcal{E}_2$ can be written in the following form:
\begin{eqnarray*}
\mathcal{E}_2&=&\sum_{k=1}^N \left[S(\bar{\mathcal{V}}_k)-\varepsilon \left(\varepsilon^2 a_{11}\alpha^{1-p}e_k''+\lambda_0 e_k\right)Z_k\right] +\left[|\mathcal{V}-1|^p-1-\sum_{k=1}^N \left(|\bar{\mathcal{V}}_k-1|^p-1\right)\right] \\
&& +B_4^2(\mathcal{V})-\sum_{k=1}^N B_4^2(\bar{\mathcal{V}}_k).
\end{eqnarray*}
By estimating the derivatives of $\mathcal{E}_2$ with respect to $f_k$ and $e_k$, we get
\begin{equation}\label{e32}
\|\mathcal{E}_2(\mathbf{f}_1,\mathbf{e}_1)-\mathcal{E}_2(\mathbf{f}_2,\mathbf{e}_2)\|_{L^2(\mathfrak{S})}\leq C\varepsilon^{\frac32-\mu}\left[\|\mathbf{f}_1-\mathbf{f}_2\|_{H^2(0,1)} +\|\mathbf{e}_1-\mathbf{e}_2\|_*\right].
\end{equation}

\section{The gluing procedure}\label{sect5}
Fix a positive constant $\delta<\delta_0/100$ and let $\eta_\delta$ be a smooth cutoff function with $\eta_\delta(t)=1$ for $t<\delta$ and $\eta_\delta(t)=0$ for $t>2\delta$. Denote $\eta_\delta^\varepsilon(x)=\eta_\delta(\varepsilon |x|)$. We define a global approximate solution of \eqref{25} by
\begin{equation}\label{92}
\mathtt{w}(y)=\eta_{3\delta}^\varepsilon(x) \tilde{\alpha}(\varepsilon z)\mathcal{V}(z,x).
\end{equation}
Then $v=\mathtt{w}+\tilde{\phi}$ solves \eqref{25} if and only if $\tilde{\phi}$ solves the following problem
\begin{equation}\label{81}
\left\{\begin{array}{ll}
-\tilde{L}(\tilde{\phi})=\tilde{E}+\tilde{N}(\tilde{\phi}), &\mbox{in $\Omega_\varepsilon$,} \\
\tilde{\phi}=0, &\mbox{on $\partial\Omega_\varepsilon$,}
\end{array}
\right.
\end{equation}
where
\[\tilde{L}(\tilde{\phi})=\Div(A(\varepsilon y)\nabla\tilde{\phi}) +p|\mathtt{w}+\bar{u}_\varepsilon|^{p-2}(\mathtt{w}+\bar{u}_\varepsilon)\tilde{\phi},\]
\[\tilde{E}=\Div(A(\varepsilon y)\nabla\mathtt{w})+|\mathtt{w}+\bar{u}_\varepsilon|^p-|\bar{u}_\varepsilon|^p\]
and
\[\tilde{N}(\tilde{\phi})=|\mathtt{w}+\bar{u}_\varepsilon+\tilde{\phi}|^p-|\mathtt{w}+\bar{u}_\varepsilon|^p-p|\mathtt{w}+\bar{u}_\varepsilon|^{p-2}(\mathtt{w}+\bar{u}_\varepsilon) \tilde{\phi}.\]

In this section, we aim to find a solution to \eqref{81} of the following form
\[\tilde{\phi}=\eta_{3\delta}^\varepsilon \varphi+\psi,\]
where we assume that $\varphi$ is defined in the whole strip $\mathfrak{S}$.

From direct computation, we recognize that $\tilde{\phi}$ is a solution of \eqref{81} if the pair $(\varphi, \psi)$ solves the following problems:
\begin{equation}\label{8}
-\eta_{3\delta}^\varepsilon \tilde{L}(\varphi)=\eta_\delta^\varepsilon \tilde{E}+ \eta_\delta^\varepsilon\tilde{N}(\varphi+\psi) +p\eta_\delta^\varepsilon|\mathtt{w}+\bar{u}_\varepsilon|^{p-2}(\mathtt{w}+\bar{u}_\varepsilon)\psi-p\eta_\delta^\varepsilon |\bar{u}_\varepsilon|^{p-2}|\bar{u}_\varepsilon|\psi
\end{equation}
and
\begin{equation}\label{4}
\left\{\begin{array}{ll}
-\Div(A(\varepsilon y)\nabla \psi)-p(1-\eta_\delta^\varepsilon)|\mathtt{w}+\bar{u}_\varepsilon|^{p-2}(\mathtt{w}+\bar{u}_\varepsilon)\psi-p\eta_\delta^\varepsilon |\bar{u}_\varepsilon|^{p-2}\bar{u}_\varepsilon \psi \\
\qquad\qquad=(1-\eta_\delta^\varepsilon)\tilde{E}+(1-\eta_\delta^\varepsilon)\tilde{N}(\eta_{3\delta}^\varepsilon\varphi+\psi) \\
\qquad\qquad\qquad+\Div(A(\varepsilon y)\nabla \eta_{3\delta}^\varepsilon)\varphi+2\langle A(\varepsilon y)\nabla \eta_{3\delta}^\varepsilon, \nabla \varphi\rangle, &\mbox{in $\Omega_\varepsilon$} \\
\psi=0, & \mbox{on $\partial \Omega_\varepsilon$}.
\end{array}\right.
\end{equation}
We call \eqref{4} the outer problem and \eqref{8} the inner problem.


\subsection{Outer problem}
In order to solve the outer problem \eqref{4}, we consider the following problem first
\begin{equation}\label{e1}
\left\{\begin{array}{ll}
-\Div(A(\varepsilon x)\nabla \psi)-(1-\eta_{\delta}^\varepsilon)p|\mathtt{w}+\bar{u}_\varepsilon|^{p-2}(\mathtt{w}+\bar{u}_\varepsilon)\psi-p\eta_\delta^\varepsilon|\bar{u}_\varepsilon|^{p-2}\bar{u}_\varepsilon \psi=h, & \mbox{in $\Omega_\varepsilon$}, \\
\psi=0, &\mbox{on $\partial \Omega_\varepsilon$.}
\end{array}\right.
\end{equation}
\begin{lemma}\label{lme1}
Provided $h\in L^2(\Omega_\varepsilon)$ and $\varepsilon>0$ small enough, \eqref{e1} admits a unique solution $\psi\in H_0^1(\Omega_\varepsilon)$ satisfying the following estimate
\begin{equation}\label{e65}
\|\psi\|_{L^\infty(\Omega_\varepsilon)}\le C \sup_{y\in \Omega_\varepsilon}\|h\|_{L^2(B_1(y)\cap\Omega_\varepsilon)},
\end{equation}
where the constant $C>0$ is independent of $\varepsilon$.
\end{lemma}
\begin{proof}
According to the standard theory of elliptic partial differential equations( c.f. \cite{Gilbarg_Trudinger2001}), we get \eqref{e1} admits a unique solution. We only need to prove the priori estimate  \eqref{e65}.

Suppose this estimate does not hold. There exist $\varepsilon_n\to 0$, and the function $h_n$ with the property $\sup_{y\in \Omega_\varepsilon}\|h_n\|_{L^2(B_1(y)\cap\Omega_\varepsilon)}\to 0$, such that the solution $\psi_n$ of \eqref{e1} corresponding to $h=h_n$ satisfies $\|\psi_n\|_{L^\infty(\Omega_{\varepsilon_n})}=1$. Thus, there exist $x_n\in \Omega$, so that $\psi_n(x_n/\varepsilon_n)>1/2$. Since $\psi_n$ satisfies Dirichlet condition, we get $x_n\to x_0\in \Omega$.

Let
\[\bar{\psi}_n(y)=\psi_n(y+x_n/\varepsilon_n).\]
Then it is the solution of the following problem
\begin{eqnarray}\label{e66}
\nonumber &&-\Div(A(x_n+\varepsilon_n y)\nabla \bar{\psi}_n) -p\eta_\delta(x_n+\varepsilon_ny) |\bar{u}_{\varepsilon_n}(x_n+\varepsilon_ny)|^{p-2}(\bar{u}_{\varepsilon_n}(x_n+\varepsilon_ny))\bar{\psi}_n  \\
\nonumber &&\qquad -p(1-\eta_\delta(x_n+\varepsilon_ny))|\mathtt{w}(y+x_n/\varepsilon_n)+\bar{u}_{\varepsilon_n}(x_n+\varepsilon_ny)|^{p-2} \\
&&\qquad \qquad\times(\mathtt{w}(y+x_n/\varepsilon_n)+\bar{u}_{\varepsilon_n}(x_n+\varepsilon_ny))\bar{\psi}_n   =\bar{h}_n, \quad \mathrm{in} \quad \tilde{\Omega}_{\varepsilon_n},
\end{eqnarray}
\[\bar{\psi}_n=0 \qquad \mathrm{on} \qquad \partial\tilde{\Omega}_{\varepsilon_n},\]
where
\[\tilde{\Omega}_{\varepsilon_n}=(\Omega-x_n)/\varepsilon_n, \qquad \mathrm{and}\qquad \bar{h}_n(y)=h_n(y+x_n/\varepsilon_n).\]

For any smooth bounded domain $D$ in $\R^2$, we get $\|\bar{h}_n\|_{L^2(D)}$ is uniformly bounded for $n$ large enough. From elliptic estimate, $\|\bar{\phi}_n\|_{H^2(D')}$ is uniformly bounded for any compact set $D'\subset\subset D$. Hence $\|\bar{\phi}_n\|_{C^{0, \gamma}(D')}$ is uniformly bounded. With the help of Proposition \ref{pro1}, we derive that $\bar{\psi}_n$ converges to the solution of the following problem on compact sets:
\begin{equation}\label{e67}
-\Div(A(x_0)\nabla \psi_0)+p\mathbf{\Psi}^{\frac{p-1}p}(x_0)\psi_0=0, \qquad \mathrm{in} \qquad \R^2.
\end{equation}
In fact, by multiplying the both sides of \eqref{e66} by $\phi\in C_0^\infty(\R^2)$, we have
\begin{eqnarray*}
&&-\int_\R \Div(A(x_n+\varepsilon_n y)\nabla\phi)\bar{\psi}_ndy -p\int_\R \eta_\delta(x_n+\varepsilon_ny) |\bar{u}_{\varepsilon_n}(x_n+\varepsilon_ny)|^{p-2}(\bar{u}_{\varepsilon_n}(x_n+\varepsilon_ny))\bar{\psi}_n\phi dy \\
&& \qquad -p\int_\R(1-\eta_\delta(x_n+\varepsilon_ny))|\mathtt{w}(y+x_n/\varepsilon_n)+\bar{u}_{\varepsilon_n}(x_n+\varepsilon_ny)|^{p-2} \times \\
&&\qquad \qquad(\mathtt{w}(y+x_n/\varepsilon_n)+\bar{u}_{\varepsilon_n}(x_n+\varepsilon_ny))\bar{\psi}_n\phi dy =\int_\R \bar{h}_n \phi dy.
\end{eqnarray*}
Letting $n\to +\infty$, we get
\[-\int_\R \Div(A(x_0)\nabla \phi)\psi_0-\int_\R \Psi^{\frac{p-1}p}(x_0)\psi_0\phi=0.\]
Hence $\psi_0$ is the solution to \eqref{e67}. Thus, $\psi_0=0$, which is a contradiction.
\end{proof}

Assume $\varphi$ satisfies the following decay condition
\begin{equation}\label{e69}
\sup_{|x|\ge \frac{3\delta}\varepsilon-1}\|\nabla \varphi\|_{L^2(B_1(z,x))}+\|\varphi\|_{L^\infty(|x|\ge 3\delta/\varepsilon)}\leq e^{-\sqrt{p}\delta/(4\varepsilon)}.
\end{equation}
From Lemma \ref{lme1} and fixed point theorem, we get \eqref{4} has a unique solution $\psi=\psi(\varphi)$, which satisfies the following estimates
\begin{equation}\label{e29}
\|\psi\|_{L^\infty(\Omega_\varepsilon)}\leq C\left[e^{-\frac{\sqrt{p}\delta}{2\varepsilon}}+\|\varphi\|_{L^\infty}^{\min\{p,2\}} +\varepsilon\| \varphi\|_{L^\infty(|s|>3\delta/\varepsilon)}+\varepsilon\sup_{|x|\ge \frac{3\delta}{\varepsilon} -1}\|\nabla\varphi\|_{L^2(B_1(z,x))}\right],
\end{equation}
and
\begin{eqnarray*}
\|\psi(\varphi_1)-\psi(\varphi_2)\|_{L^\infty(\Omega_\varepsilon)}&\leq& C\varepsilon\left[\|\varphi_1- \varphi_2\|_{L^\infty(|s|>3\delta/\varepsilon)}+\sup_{|x|\ge \frac{3\delta}{\varepsilon}-1}\|\nabla \varphi_1-\nabla \varphi_2\|_{L^2(B_1(z,x))}\right] \\
&&\quad+C\left[\|\varphi_1\|_{L^\infty}^{\min\{p-1,1\}}+\|\varphi_2\|_{L^\infty}^{\min\{p-1,1\}}\right]\|\varphi_1-\varphi_2\|_{L^\infty}.
\end{eqnarray*}

\subsection{Inner problem}

After solving the outer problem \eqref{4} for given $\varphi$ satisfying \eqref{e69}, we only need to solve the following problem
\begin{equation}\label{12}
-\eta_{3\delta}^\varepsilon \tilde{L}(\varphi)=\eta_\delta^\varepsilon \tilde{E}+ \eta_\delta^\varepsilon\tilde{N}(\varphi+\psi(\varphi)) +p\eta_\delta^\varepsilon|\mathtt{w}+\bar{u}_\varepsilon|^{p-2}(\mathtt{w}+\bar{u}_\varepsilon)\psi(\varphi)-p\eta_\delta^\varepsilon |\bar{u}_\varepsilon|^{p-2}|\bar{u}_\varepsilon|\psi(\varphi).
\end{equation}
Let
\[\varphi(z,s)=\tilde{\alpha}(\varepsilon z)\hat{\varphi}(z,x), \quad \mathrm{where} \quad x=\tilde{\beta}(\varepsilon z)s.\]
In $(z,x)$ coordinate, we have
\begin{equation*}
\tilde{\alpha}^{-p}\tilde{L}(\varphi)=\hat{\varphi}_{xx}+a_{11}\tilde{\alpha}^{1-p}\hat{\varphi}_{zz}+B_4^1(\hat{\varphi})+B_4^3(\hat{\varphi})+p\left|\eta_{3\delta}^\varepsilon \mathcal{V}-\tilde{\alpha}^{-1}\bar{\mathbf{q}}\right|^{p-2}\left(\eta_{3\delta}^\varepsilon \mathcal{V}-\tilde{\alpha}^{-1}\bar{\mathbf{q}}\right)\hat{\varphi}.
\end{equation*}
Now we extend the operator on the right hand side of the equality above to the whole $\mathfrak{S}$. Let
\begin{equation}\label{62}
\mathcal{L}(\phi)=\phi_{xx}+a_{11}\tilde{\alpha}^{1-p}\phi_{zz}+p|\mathcal{V}-1|^{p-2}(\mathcal{V}-1)\phi +B_5(\phi),
\end{equation}
where
\begin{equation*}
B_5(\phi)= \eta_{6\delta}^\varepsilon \left(B_4^1(\phi)+B_4^3(\phi)\right) +p\eta_{6\delta}^\varepsilon\left[\left|\eta_{3\delta}^\varepsilon \mathcal{V}-\tilde{\alpha}^{-1}\bar{\mathbf{q}}\right|^{p-2}\left(\eta_{3\delta}^\varepsilon \mathcal{V}-\tilde{\alpha}^{-1}\bar{\mathbf{q}}\right)-|\mathcal{V}-1|^{p-2}(\mathcal{V}-1)\right]\phi.
\end{equation*}

Rather than solving problem \eqref{12} directly, we consider the following problem
\begin{equation*}
-\mathcal{L}(\hat{\varphi})=\tilde{\alpha}^{-p}\left[\eta_\delta^\varepsilon \tilde{E}+ \eta_\delta^\varepsilon\tilde{N}(\varphi+\psi(\varphi)) +p\eta_\delta^\varepsilon|\mathtt{w}+\bar{u}_\varepsilon|^{p-2}(\mathtt{w}+\bar{u}_\varepsilon)\psi(\varphi)-p\eta_\delta^\varepsilon |\bar{u}_\varepsilon|^{p-2}|\bar{u}_\varepsilon|\psi(\varphi)\right],
\end{equation*}
which is equivalent to the following one:
\begin{equation}\label{e93}
\mathcal{L}(\hat{\varphi})=-\eta_\delta^\varepsilon S(\mathcal{V})-\eta_\delta^\varepsilon \hat{N}(\hat{\varphi}+\hat{\psi}(\hat{\varphi}))-p\eta_\delta^\varepsilon\left[\left| \mathcal{V}-\tilde{\alpha}^{-1}\bar{\mathbf{q}}\right|^{p-2}\left( \mathcal{V}-\tilde{\alpha}^{-1}\bar{\mathbf{q}}\right)+\left|\tilde{\alpha}^{-1}\bar{\mathbf{q}}\right|^{p-2}\left(\tilde{\alpha}^{-1}\bar{\mathbf{q}}\right)\right]\hat{\psi}(\hat{\varphi}),
\end{equation}
where $\psi(z,s)=\tilde{\alpha}(\varepsilon z)\hat{\psi}(z,x)$ and
\[\hat{N}(\phi)=|\mathcal{V}-\tilde{\alpha}^{-1}\bar{\mathbf{q}}+\phi|^p-|\mathcal{V}-\tilde{\alpha}^{-1}\bar{\mathbf{q}}|^p-p|\mathcal{V}-\tilde{\alpha}^{-1}\bar{\mathbf{q}}|^{p-2}(\mathcal{V}- \tilde{\alpha}^{-1}\bar{\mathbf{q}})\phi.\]

We first solve the following projective version of \eqref{e93}:
\begin{eqnarray}\label{24}
\nonumber\mathcal{L}(\hat{\varphi}) &=&-p\eta_\delta^\varepsilon\left[\left| \mathcal{V}-\tilde{\alpha}^{-1}\bar{\mathbf{q}}\right|^{p-2}\left( \mathcal{V}- \tilde{\alpha}^{-1}\bar{\mathbf{q}}\right)+\left|\tilde{\alpha}^{-1}\bar{\mathbf{q}}\right|^{p-2}\left(\tilde{\alpha}^{-1}\bar{\mathbf{q}}\right)\right]\hat{\psi}(\hat{\varphi})\\
&&-\eta_\delta^\varepsilon \mathcal{E}_2-\eta_\delta^\varepsilon \hat{N}(\hat{\varphi}+\hat{\psi}(\hat{\varphi})) +\sum_{j=1}^N \eta_\delta^\varepsilon c_j(\varepsilon z)w_{j,x}+\sum_{j=1}^N \eta_\delta^\varepsilon d_j(\varepsilon z)Z_j,
\end{eqnarray}
\begin{equation}\label{43}
\int_\R \hat{\varphi}(z,x)w_{j,x}(x)dx= \int_\R \hat{\varphi}(z,x)Z_j(x)dx=0,
\end{equation}
\begin{equation}\label{e2}
\hat{\varphi}(0,x)=\hat{\varphi}(1/\varepsilon,x), \qquad  \hat{\varphi}_z(0,x)=\hat{\varphi}_z(1/\varepsilon,x),
\end{equation}
where the term $\mathcal{E}_1$ is absorbed in $\sum_{j=1}^N \eta_\delta^\varepsilon d_j(\varepsilon z)Z_j$. Here we recall $w_j$ and $Z_j$ are defined in \eqref{e109}.

\section{Linear theory}\label{sect6}
In order to solve problem \eqref{24}-\eqref{e2}, we consider the following problem
\begin{equation}\label{48}
\left\{\begin{array}{ll}
L(\phi)=h+\sum_{j=1}^N c_j(\varepsilon z)w_{j,x} +\sum_{j=1}^N d_j(\varepsilon z)Z_j, &\mbox{in $\mathfrak{S}$,} \\
\int_\R \phi(z,x)w_{j,x}(x)dx= \int_\R \phi(z,x)Z_j(x)dx=0, &z\in (0,1/\varepsilon), \\
\phi(0,x)=\phi(1/\varepsilon,x), \quad \phi_z(0,x)=\phi_z(1/\varepsilon,x), &x\in \R,
\end{array}\right.
\end{equation}
where
\[L(\phi)=\phi_{xx}+a_{11}\tilde{\alpha}^{1-p}\phi_{zz}+p|\mathcal{V}-1|^{p-2}(\mathcal{V}-1)\phi.\]

Let
\[\phi(z,x)=\sum_{j=1}^N \eta_j\tilde{\phi}_j+\tilde{\psi}, \qquad \eta_j(z,x)=\eta\left(\frac{x-\beta(\varepsilon z) f_j(\varepsilon z)}R\right)\]
where $R=\frac1{\sqrt{p}}|\log\varepsilon|$,  $\eta(r)$ is a smooth cutoff function such that $\eta(r)=1$ for $|r|<1/2$ and  $\eta(r)=0$ for $|r|>5/6$. Then $\phi$ solves \eqref{48} if and only if the functions $\tilde{\phi}_j$, $\tilde{\psi}$ solve the following equations:
\begin{eqnarray}\label{50}
\nonumber&&\tilde{\phi}_{j,xx}+a_{11}\tilde{\alpha}^{1-p}\tilde{\phi}_{j,zz}+p|w_j-1|^{p-2}(w_j-1)\tilde{\phi}_j \\
&=& \tilde{\chi}_j h+p\left[|w_j-1|^{p-2}(w_j-1)-|\mathcal{V}-1|^{p-2}(\mathcal{V}-1)\right]\tilde{\chi}_j \tilde{\phi}_j  \\
\nonumber&&-p\left[|\mathcal{V}-1|^{p-2}(\mathcal{V}-1)+1\right]\tilde{\chi}_j\tilde{\psi} +c_j(\varepsilon z)w_{j,x}+d_j(\varepsilon z)Z_j,
\end{eqnarray}
\begin{equation}\label{51}
\int_\R \tilde{\phi}_j(z,x)w_{j,x}(x)dx=\Lambda_{j1}, \qquad \int_\R \tilde{\phi}_j(z,x)Z_j(x)dx=\Lambda_{j2}, \qquad z\in \left(0,\frac1\varepsilon\right),
\end{equation}
\begin{equation}\label{52}
\tilde{\phi}_j(0,x)=\tilde{\phi}_j(1/\varepsilon,x), \qquad \tilde{\phi}_{j,z}(0,x)=\tilde{\phi}_{j,z}(1/\varepsilon,x), \qquad x\in \R
\end{equation}
and
\begin{eqnarray}\label{54}
\nonumber&&\tilde{\psi}_{xx}+a_{11}\tilde{\alpha}^{1-p}\tilde{\psi}_{zz}+p\left(1-\sum_{j=1}^N\eta_j\right)|\mathcal{V}-1|^{p-2}(\mathcal{V}-1)\tilde{\psi}-p\sum_{j=1}^N \eta_j \tilde{\psi} \\
&=& \left(1-\sum_{j=1}^N\eta_j\right)h-\sum_{j=1}^N\left(\eta_{j,xx}\tilde{\phi}_j+2\eta_{j,x}\tilde{\phi}_{j,x}\right) -a_{11}\tilde{\alpha}^{1-p}\sum_{j=1}^N\left(\eta_{j,zz}\tilde{\phi}_j+2\eta_{j,z}\tilde{\phi}_{j,z}\right) \\
\nonumber&& +\sum_{j=1}^N (1-\eta_j)c_j(\varepsilon z)w_{j,x} +\sum_{j=1}^N (1-\eta_j)d_j(\varepsilon z) Z_j,
\end{eqnarray}
where
\[\Lambda_{j1}=\int_\R(1-\eta_j)\tilde{\phi}_j w_{j,x}dx-\sum_{k\not=j}\int_\R\eta_k \tilde{\phi}_k w_{j,x}dx-\int_\R \tilde{\psi}w_{j,x}dx,\]
\[\Lambda_{j2}=\int_\R(1-\eta_j)\tilde{\phi}_j Z_jdx-\sum_{k\not=j}\int_\R\eta_k \tilde{\phi}_k Z_jdx-\int_\R \tilde{\psi}Z_jdx,\]
\[\tilde{\chi}_j(z,x)=\chi\left(\frac{x-\beta(\varepsilon z) f_j(\varepsilon z)}R\right)\]
and $\chi(r)$ is a smooth cutoff function such that $\chi(r)=1$ for $|r|<5/6$ and  $\chi(r)=0$ for $|r|>7/8$.

In order to solve these problems, we consider the following problem first
\begin{equation}\label{e3}
\phi_{xx}+a_{11}\alpha^{1-p}\phi_{zz}+p|w-1|^{p-2}(w-1)\phi=h, \qquad \mathrm{in} \qquad \mathfrak{S},
\end{equation}
\begin{equation}\label{e4}
\int_\R \phi(z,x) w_x(x)dx=\Lambda_1(z); \qquad \int_\R \phi(z,x) Z(x)dx=\Lambda_3(z), \qquad z\in(0,1/\varepsilon),
\end{equation}
\begin{equation}\label{e5}
\phi(0,x)=\phi(1/\varepsilon,x), \qquad \phi_z(0,x)=\phi_z(1/\varepsilon,x), \qquad x\in \R.
\end{equation}
\begin{lemma}\label{lme2}
There exists a constant $C>0$ independent of $\varepsilon>0$ such that the solution $\phi$ to  \eqref{e3}-\eqref{e5} satisfies the following priori estimate
\begin{equation}\label{e10}
\|\phi\|_{H^2(\mathfrak{S})}\leq C\left[\|h\|_{L^2(\mathfrak{S})}+\|\Lambda_1\|_{H^2(0, 1/\varepsilon)}+\|\Lambda_2\|_{H^2(0, 1/\varepsilon)}\right].
\end{equation}
\end{lemma}
\begin{proof}
We first consider the case of $\Lambda_1=\Lambda_2=0$. According to \cite[Theorem 1.5.1]{Levitan_Sargsjan1991}, we expand $\phi$ and $h$ into the following form
\begin{equation}\label{e8}
\phi(z,x)=\sum_{k=0}^{+\infty}\phi_k(x) \upsilon_k(z), \qquad \mathrm{and} \qquad h(z,x)=\sum_{k=0}^{+\infty}h_k(x) \upsilon_k(z),
\end{equation}
where $\upsilon_k$ is the unit $L^2$ eigenfunction of the following problem
\begin{equation*}
\left\{\begin{array}{ll}
-a_{11}\alpha^{1-p}\upsilon_k''=\tilde{\mu}_k \upsilon_k, &\mbox{in $(0,1/\varepsilon)$}, \\
\upsilon_k(0)=\upsilon_k(1/\varepsilon),
\end{array}\right.
\end{equation*}
and $\tilde{\mu}_k\ge 0$.
Then $\phi_k$ and $h_k$ satisfy the following equations
\begin{equation}\label{e6}
\phi_{k,xx}-\tilde{\mu}_k \phi_k +p|w-1|^{p-2}(w-1)\phi_k=h_k \quad \mathrm{in} \quad \R,
\end{equation}
\begin{equation}\label{e7}
\int_\R \phi_k(x) w_x(x)dx=\int_\R \phi_k(x) Z(x)dx=0.
\end{equation}
Multiplying the both sides of \eqref{e6} by $\phi_k$ and integrating, we have
\begin{equation*}
\int_{-\infty}^{+\infty}\left[ |\phi_{k,x}|^2-p|w-1|^{p-2}(w-1)\phi_k^2 \right]dx +\tilde{\mu}_k\int_\R \phi_k^2 dx=-\int_\R h_k \phi_kdx.
\end{equation*}
Since \eqref{e7} holds, we get
\[\int_{-\infty}^{+\infty}\left[ |\phi_{k,x}|^2-p|w-1|^{p-2}(w-1)\phi_k^2 \right]dx \geq C\left[\|\phi_k\|_{L^2(\R)}^2+\|\phi_{k,x}\|_{L^2(\R)}^2\right].\]
Then we arrive at
\begin{equation}\label{e9}
\int_\R|\phi_{k,x}|^2 dx+(1+\tilde{\mu}_k)\int_\R\phi_k^2dx \leq \int_\R h_k^2dx.
\end{equation}
From the expression \eqref{e8}, we get
\[\|\phi\|_{L^2(\mathfrak{S})}^2=\sum_{k=0}^\infty \int_\R |\phi_k|^2dx, \qquad \|\phi_x\|_{L^2(\mathfrak{S})}^2=\sum_{k=0}^{+\infty}\int_\R|\phi_{k,x}|^2dx\]
and
\[\int_{\mathfrak{S}}a_{11}\alpha^{1-p}\phi_{zz}\phi dxdz=-\sum_{k=0}^\infty \tilde{\mu}_k\int_\R|\phi_k|^2dx, \qquad  \|h\|_{L^2(\mathfrak{S})}^2=\sum_{k=0}^\infty \int_\R |h_k|^2dx.\]
With the help of \eqref{e9} and these identities above, we get
\[\|\phi\|_{H^1(\mathfrak{S})}\leq C\|h\|_{L^2(\mathfrak{S})}.\]
From elliptic estimate,
\[\|\phi\|_{H^2(\mathfrak{S})}\leq C\|h\|_{L^2(\mathfrak{S})}.\]
Hence in the case of  $\Lambda_1=\Lambda_2=0$, \eqref{e10} holds.

In the general case, we define
\[\bar{\phi}(z,x)=\phi(z,x)-\frac{\Lambda_1(z)}{\int_\R w_x^2dx}w_x(x)-\frac{\Lambda_2(z)}{\int_\R Z^2dx}Z(x).\]
It is the solution of the following problem
\begin{eqnarray*}
\bar{\phi}_{xx}+a_{11}\alpha^{1-p}\bar{\phi}_{zz}+p|w-1|^{p-2}(w-1)\bar{\phi}&=&h-\frac{a_{11}\alpha^{1-p}\Lambda_1''}{\int_\R w_x^2dx}w_x-\frac{a_{11}\alpha^{1-p}\Lambda_2''}{\int_\R Z^2dx}Z(x) \\
&&-\frac{\lambda_0\Lambda_2}{\int_\R Z^2dx}Z(x),\qquad\qquad \mbox{in $\mathfrak{S}$},
\end{eqnarray*}
\begin{equation*}
\int_\R \bar{\phi}(x) w_x(x)dx=\int_\R \bar{\phi}(x) Z(x)dx=0, \qquad z\in(0, 1/\varepsilon).
\end{equation*}
Using the conclusion in the case $\Lambda_1=\Lambda_2=0$, we get
\[\|\bar{\phi}\|_{H^2(\mathfrak{S})}\leq C\left[\|h\|_{L^2(\mathfrak{S})}+\|\Lambda_1\|_{H^2(0, 1/\varepsilon)}+\|\Lambda_2\|_{H^2(0, 1/\varepsilon)}\right].\]
Then \eqref{e10} holds.
\end{proof}

Now we consider the following projective version
\begin{equation}\label{e11}
\phi_{xx}+a_{11}\alpha^{1-p}\phi_{zz}+p|w-1|^{p-2}(w-1)\phi=h+c(\varepsilon z)w_x(x) +d(\varepsilon z)Z(x), \qquad \mathrm{in} \qquad \mathfrak{S},
\end{equation}
\begin{equation}\label{e12}
\int_\R \phi(z,x) w_x(x)dx=\Lambda_1(z), \qquad \int_\R \phi(z,x) Z(x)dx=\Lambda_2(z), \qquad z\in(0,1/\varepsilon),
\end{equation}
\begin{equation}\label{e13}
\phi(0,x)=\phi(1/\varepsilon,x), \qquad \phi_z(0,x)=\phi_z(1/\varepsilon,x), \qquad x\in \R.
\end{equation}
\begin{lemma}\label{lme3}
Provided $h\in L^2(\mathfrak{S})$ and $\Lambda_1, \Lambda_2\in H^2(0,1/\varepsilon)$,  \eqref{e11}-\eqref{e13} has a unique solution $\phi=T_0(h, \Lambda_1, \Lambda_2)$. Moreover, it satisfies the following estimate
\begin{equation}\label{e16}
\|\phi\|_{H^2(\mathfrak{S})}\leq C\left[\|h\|_{L^2(\mathfrak{S})}+\|\Lambda_1\|_{H^2(0, 1/\varepsilon)}+\|\Lambda_2\|_{H^2(0, 1/\varepsilon)}\right].
\end{equation}
\end{lemma}
\begin{proof}
We first consider the case of $\Lambda_1=\Lambda_2=0$. Write $h$ into the following form
\[
 h(z,x)=\sum_{k=0}^{+\infty}h_k(x) \upsilon_k(z).
\]
In order to solve \eqref{e11}-\eqref{e13}, we consider the following problem
\begin{equation}\label{e14}
\phi_{k,xx}-\tilde{\mu}_k \phi_k +p|w-1|^{p-2}(w-1)\phi_k=h_k(x) +c_k w_x(x)+d_k Z(x), \qquad \mathrm{in} \qquad \R,
\end{equation}
\begin{equation}\label{e15}
\int_\R \phi_k w_xdx=\int_\R \phi_k Zdx=0,
\end{equation}
where $c_k$ and $d_k$ are constants. From Fredholm alternative, we get \eqref{e14}-\eqref{e15} has a unique solution $\phi_k$, with
\[c_k=-\frac{\int_\R h_k w_xdx}{\int_\R w_x^2dx}, \qquad  d_k=-\frac{\int_\R h_k Zdx}{\int_\R Z^2dx}.\]
It is easy to get
\[\sum_{k=0}^\infty |c_k|^2 \leq C \|h\|_{L^2(\mathfrak{S})}^2, \qquad \sum_{k=0}^\infty |d_k|^2 \leq C \|h\|_{L^2(\mathfrak{S})}^2.\]
Let
\[\phi(z,x)=\sum_{k=0}^\infty \phi_k(x) \upsilon_k(z),\]
\[c(\varepsilon z)=\sum_{k=0}^\infty c_k \upsilon_k(z), \qquad \mathrm{and} \qquad  d(\varepsilon z)=\sum_{k=0}^\infty d_k \upsilon_k(z).\]
Then $\phi$ is the unique solution of  problem \eqref{e11}-\eqref{e13}.

However, in the general case, we define
\[\bar{\phi}(z,x)=\phi(z,x)-\frac{\Lambda_1(z)}{\int_\R w_x^2dx}w_x(x)-\frac{\Lambda_2(z)}{\int_\R Z^2dx}Z(x).\]
Then \eqref{e11}-\eqref{e13} is transformed into the following problem
\begin{eqnarray*}
\bar{\phi}_{xx}+a_{11}\alpha^{1-p}\bar{\phi}_{zz}+p|w-1|^{p-2}(w-1)\bar{\phi}&=&h-\frac{a_{11}\alpha^{1-p}\Lambda_1''}{\int_\R w_x^2dx}w_x-\frac{a_{11}\alpha^{1-p}\Lambda_2''}{\int_\R Z^2dx}Z(x) \\
&&-\frac{\lambda_0\Lambda_2}{\int_\R Z^2dx}Z(x) +c(\varepsilon z)w_x +d(\varepsilon z)Z(x), \quad \mathrm{in} \quad \mathfrak{S},
\end{eqnarray*}
\begin{equation*}
\int_\R \bar{\phi}(x) w_x(x)dx=\int_\R \bar{\phi}(x) Z(x)dx=0, \qquad z\in(0,1/\varepsilon).
\end{equation*}
It is uniquely solvable according to the argument above. The priori estimate \eqref{e16} follows from Lemma \ref{lme2}.
\end{proof}

Next,  we consider the following problem
\begin{equation}\label{e17}
\phi_{xx}+a_{11}\tilde{\alpha}^{1-p}\phi_{zz}+p|w_j-1|^{p-2}(w_j-1)\phi=h+c_j(\varepsilon z)w_{j,x}(x) +d_j(\varepsilon z)Z_j(x), \qquad \mathrm{in} \qquad \mathfrak{S},
\end{equation}
\begin{equation}\label{e18}
\int_\R \phi(z,x) w_{j,x}(x)dx=\Lambda_1(z), \qquad \int_\R \phi(z,x) Z_j(x)dx=\Lambda_2(z), \qquad z\in(0,1/\varepsilon),
\end{equation}
\begin{equation}\label{e19}
\phi(0,x)=\phi(1/\varepsilon,x), \qquad \phi_z(0,x)=\phi_z(1/\varepsilon,x), \qquad x\in \R.
\end{equation}
\begin{lemma}\label{lme4}
Given $h\in L^2(\mathfrak{S})$, \eqref{e17}-\eqref{e19} has a unique solution $\phi=T_j(h, \Lambda_1, \Lambda_2)$, which satisfies the following estimate
\begin{equation}\label{e21}
\|\phi\|_{H^2(\mathfrak{S})}\leq C\left[\|h\|_{L^2(\mathfrak{S})}+\|\Lambda_1\|_{H^2(0, 1/\varepsilon)}+\|\Lambda_2\|_{H^2(0, 1/\varepsilon)}\right].
\end{equation}
Furthermore, the operator $T_j$ is Lipschitz continuous with respect to $\mathbf{f}$, \text{i.e.}
\begin{equation}\label{e72}
\|T_{j,\mathbf{f}_1}-T_{j,\mathbf{f}_2}\|\le C\|\mathbf{f}_1-\mathbf{f}_2\|_{H^2(0,1)}.
\end{equation}
\end{lemma}
\begin{proof}
Let
\[\tilde{\phi}(z,x)=\phi(z, x+\beta(\varepsilon z)f_j(\varepsilon z)), \qquad \mathrm{and} \qquad \tilde{h}(z,x)=h(z, x+\beta(\varepsilon z)f_j(\varepsilon z)).\]
Then \eqref{e17}-\eqref{e19} is transformed into the following problem
\[\tilde{\phi}_{xx}+a_{11}\alpha^{1-p}\tilde{\phi}_{zz}+B_6(\tilde{\phi})+p|w-1|^{p-2}(w-1)\tilde{\phi}=\tilde{h}+c_j w_x+d_j Z, \qquad \mathrm{in} \qquad \mathfrak{S},\]
\[\int_\R \tilde{\phi}(z,x) w_x(x)dx=\Lambda_1(z); \qquad \int_\R \tilde{\phi}(z,x) Z(x)dx=\Lambda_2(z),\]
\[\tilde{\phi}(0,x)=\tilde{\phi}(1/\varepsilon,x), \qquad \tilde{\phi}_z(0,x)=\tilde{\phi}_z(1/\varepsilon,x), \qquad x\in \R,\]
where
\[B_6(\tilde{\phi})=a_{11}\tilde{\alpha}^{1-p}\left[\varepsilon^2|(\beta f_j)'|^2\tilde{\phi}_{xx}-\varepsilon^2(\beta f_j)''\tilde{\phi}_x -2\varepsilon(\beta f_j)' \tilde{\phi}_{zx}\right] +a_{11}(\tilde{\alpha}^{1-p}-\alpha^{1-p})\tilde{\phi}_{zz}.\]
We reformulate this problem as the following fixed point problem:
\begin{equation}\label{e20}
\tilde{\phi}=T_0(\tilde{h}-B_6(\tilde{\phi}), \Lambda_1, \Lambda_2).
\end{equation}
We get $\|B_6(\tilde{\phi})\|_{L^2(\mathfrak{S})}\leq C\varepsilon^{1/2}\|\tilde{\phi}\|_{H^2(\mathfrak{S})}$ via direct computation along with \cite[Theorem 8.8]{Brezis2011}. By virtue of the fixed point theorem, \eqref{e20} has a unique solution. The unique solution further satisfies \eqref{e21} in light of Lemma \ref{lme3}.

Now we estimate the Lipschitz property of $T_j$. Let $\phi_i=T_{j,\mathbf{f}_i}(h, \Lambda_1, \Lambda_2)$ with $i=1,2$. Then $\phi_i$ is the solution of the following problem
\begin{equation}\label{e71}
\phi_{i,xx}+a_{11}\tilde{\alpha}^{1-p}\phi_{i,zz} +p|w_{ji}-1|^{p-2}(w_{ji}-1)\phi_i=h+c_{ji}w_{ji,x}+d_{ji}Z_{ji}, \qquad \mathrm{in} \qquad \mathfrak{S},
\end{equation}
\[\int_\R \phi_i(z,x)w_{ji}dx=\Lambda_1, \qquad \int_\R \phi_i(z,x)Z_{ji}dx=\Lambda_2,\]
\[\phi_i(0,x)=\phi_i(1/\varepsilon,x), \qquad \phi_{i,z}(0,x)=\phi_{i,z}(1/\varepsilon,x),\]
where
\[w_{ji}(z,x)=w(x-\beta(\varepsilon z) f_{ji}(\varepsilon z)), \qquad Z_{ji}(z,x)=Z(x-\beta(\varepsilon z) f_{ji}(\varepsilon z))\]

Set $\phi^*=\phi_1-\phi_2$. Then $\phi^*$ satisfies the following equations:
\begin{eqnarray*}
   &&\phi^*_{xx}+a_{11}\tilde{\alpha}^{1-p}\phi^*_{zz}+p|w_{j1}-1|^{p-2}(w_{j1}-1)\phi^* \\
   &=&  c_{j2}(w_{j1,x}-w_{j2,x})+d_{j2}(Z_{j1}-Z_{j2})  +(c_{j1}-c_{j2})w_{j1,x}+(d_{j1}-d_{j2})Z_{j1} \\
   && \quad -p\left[|w_{j1}-1|^{p-2}(w_{j1}-1)-|w_{j2}-1|^{p-2}(w_{j2}-1)\right]\phi_2, \qquad \mathrm{in} \qquad \mathfrak{S},
\end{eqnarray*}
\[\int_\R \phi^* w_{j1,x}dx=-\int_\R \phi_2(w_{j1,x}-w_{j2,x})dx, \qquad \int_\R \phi^* Z_{j1}dx=-\int_\R \phi_2(Z_{j1}-Z_{j2})dx,\]
\[\phi^*(0,x)=\phi^*(1/\varepsilon,x), \qquad \phi^*_z(0,x)=\phi^*_z(1/\varepsilon,x).\]

Multiplying the both sides of \eqref{e71} by $w_{ji,x}$ and $Z_{ji}$, respectively and integrating, we get
\[a_{11}\tilde{\alpha}^{1-p}\Lambda_1''=\int_\R hw_{ji,x}dx+c_{ji}\int_\R w_x^2dx\]
and
\[a_{11}\tilde{\alpha}^{1-p}\Lambda_2''+\lambda_0 \Lambda_2=\int_\R hZ_{ji}dx+d_{ji}.\]
Thus, we obtain
\[\|c_{ji}\|_{L^2(0,1/\varepsilon)}\le C\|\Lambda_1\|_{H^2(0,1/\varepsilon)}+\|h\|_{L^2(\mathfrak{S})}\]
and
\[\|d_{ji}\|_{L^2(0,1/\varepsilon)}\le C\|\Lambda_2\|_{H^2(0,1/\varepsilon)}+\|h\|_{L^2(\mathfrak{S})}.\]
From this fact, we have
\begin{eqnarray*}
&&\|c_{j2}(w_{j1,x}-w_{j2,x})+d_{j2}(Z_{j1}-Z_{j2})\|_{L^2(\mathfrak{S})} \\
&\le& \left[\|c_{ji}\|_{L^2(0,1/\varepsilon)}+\|d_{ji}\|_{L^2(0,1/\varepsilon)}\right] \|\mathbf{f}_1-\mathbf{f}_2\|_{H^2(0,1)} \\
&\le& C\|\mathbf{f}_1-\mathbf{f}_2\|_{H^2(0,1)}\left[\|h\|_{L^2(\mathfrak{S})}+\|\Lambda_1\|_{H^2(0, 1/\varepsilon)}+\|\Lambda_2\|_{H^2(0, 1/\varepsilon)}\right].
\end{eqnarray*}
Denote
\[\tilde{\Lambda}_1(z)=-\int_\R \phi_2(z,x)(w_{j1,x}-w_{j2,x})dx, \quad \mathrm{and} \quad \tilde{\Lambda}_2(z)=-\int_\R \phi_2(z,x)(Z_{j1}-Z_{j2})dx.\]
Through direct computation, we get
\[\|\tilde{\Lambda}_1\|_{H^2(0,1/\varepsilon)}+\|\tilde{\Lambda}_2\|_{H^2(0,1/\varepsilon)}\le C\|\phi_2\|_{H^2(\mathfrak{S})}\|\mathbf{f}_1-\mathbf{f}_2\|_{H^2(0,1)}\]
and
\[\|p[|w_{j1}-1|^{p-2}(w_{j1}-1)-|w_{j2}-1|^{p-2}(w_{j2}-1)]\phi_2\|_{L^2(\mathfrak{S})}\le C\|f_1-f_2\|_{H^2(0,1)}\|\phi_2\|_{H^2(\mathfrak{S})}.\]
Using the priori estimate in \eqref{e21}, we arrive at
\[\|\phi^*\|_{H^2(S)}\le  C\|\mathbf{f}_1-\mathbf{f}_2\|_{H^2(0,1)}\left[\|h\|_{L^2(\mathfrak{S})}+\|\Lambda_1\|_{H^2(0, 1/\varepsilon)}+\|\Lambda_2\|_{H^2(0, 1/\varepsilon)}\right].\]
Hence \eqref{e72} holds.
\end{proof}
\begin{proposition}\label{pre1}
Problem \eqref{48} has a unique solution $\phi=\tilde{T}(h)$, where $\tilde{T}$ is a linear operator and  the following estimate holds:
\begin{equation}\label{e74}
\|\phi\|_{H^2(\mathfrak{S})}\leq C\|h\|_{L^2(\mathfrak{S})}.
\end{equation}
Moreover, $\tilde{T}$ is Lipschitz continuous on $\mathbf{f}$ and $\mathbf{e}$:
\begin{equation}\label{e73}
\|\tilde{T}_{\mathbf{f}_1,\mathbf{e}_1}-\tilde{T}_{\mathbf{f}_2,\mathbf{e}_2}\|\le C\left[ \|\mathbf{f}_1-\mathbf{f}_2\|_{H^2(\mathfrak{S})}+\|\mathbf{e}_1-\mathbf{e}_2\|_*\right].
\end{equation}
\end{proposition}
\begin{proof}
According to the argument above, we only need to solve the system \eqref{50}-\eqref{52} and the problem \eqref{54}.
With Lemma \ref{lme4}, we first consider problem \eqref{50}-\eqref{52} for fixed function $\tilde{\psi}$, which we reformulate as the following problem:
\begin{eqnarray}\label{e22}
\nonumber \tilde{\phi}_j&=&T_j\left[\tilde{\chi}_j h+p\left[|w_j-1|^{p-2}(w_j-1)-|\mathcal{V}-1|^{p-2}(\mathcal{V}-1)\right]\tilde{\chi}_j \tilde{\phi}_j \right. \\ &&\left.-p\left[|\mathcal{V}-1|^{p-2}(\mathcal{V}-1)+1\right]\tilde{\chi}_j\tilde{\psi}, \Lambda_{j1}, \Lambda_{j2}\right], \qquad j=1,2,\cdots, N.
\end{eqnarray}
From tedious computation and Sobolev imbbeding theorem in \cite[Theorem 8.8]{Brezis2011}, we get
\[\|\Lambda_{j1}\|_{H^2(0,1/\varepsilon)}+\|\Lambda_{j2}\|_{H^2(0,1/\varepsilon)}\leq C\varepsilon^{1/2}\sum_{k=1}^N\|\tilde{\phi}_k\|_{H^2(\mathfrak{S})}+C\|\tilde{\psi}\|_{H^2(\mathfrak{S})},\]
\[\|p[|\mathcal{V}-1|^{p-2}(\mathcal{V}-1)+1]\tilde{\chi}_j\tilde{\psi}\|_{L^2(\mathfrak{S})}\le C\|\tilde{\psi}\|_{H^2(\mathfrak{S})}\]
and
\[\|p\left[|w_j-1|^{p-2}(w_j-1)-|\mathcal{V}-1|^{p-2}(\mathcal{V}-1)\right]\tilde{\chi}_j \tilde{\phi}_j\|_{L^2(\mathfrak{S})}\leq C\varepsilon^{1/2}\|\tilde{\phi}_j\|_{H^2(\mathfrak{S})}.\]
Using the fixed point theorem, we deduce that \eqref{e22} admits a unique solution $\tilde{\phi}_j=\tilde{\phi}_j(\tilde{\psi})$, $j=1,2,\cdots,N$. Furthermore, $\tilde{\phi}_j$ satisfies the following estimate
\begin{equation}\label{e24}
\|\tilde{\phi}_j\|_{H^2(\mathfrak{S})}\leq C\left[\|h\|_{L^2(\mathfrak{S})}+\|\tilde{\psi}\|_{H^2(\mathfrak{S})}\right].
\end{equation}
It is easy to get $\phi_j$ is Lipschitz dependent on $\psi$:
\begin{equation}\label{e70}
\|\tilde{\phi}_j(\tilde{\psi}_1)-\tilde{\phi}_j(\tilde{\psi}_2)\|_{H^2(\mathfrak{S})}\le C\|\tilde{\psi}_1-\tilde{\psi}_2\|_{H^2(\mathfrak{S})}.
\end{equation}

Substituting $\tilde{\phi}_j$ with $\tilde{\phi}_j(\tilde{\psi})$ in \eqref{54}, we consider the following problem
\begin{eqnarray}\label{e23}
\nonumber&&\tilde{\psi}_{xx}+a_{11}\tilde{\alpha}^{1-p}\tilde{\psi}_{zz}+p\left(1-\sum_{j=1}^N\eta_j\right)|\mathcal{V}-1|^{p-2}(\mathcal{V}-1)\tilde{\psi}-p\sum_{j=1}^N \eta_j \tilde{\psi} \\
\nonumber&=& \left(1-\sum_{j=1}^N\eta_j\right)h-\sum_{j=1}^N\left(\eta_{j,xx}\tilde{\phi}_j(\tilde{\psi})+2\eta_{j,x}\tilde{\phi}_{j,x}(\tilde{\psi})\right) -a_{11}\tilde{\alpha}^{1-p}\sum_{j=1}^N\left(\eta_{j,zz}\tilde{\phi}_j(\tilde{\psi})+2\eta_{j,z}\tilde{\phi}_{j,z}(\tilde{\psi})\right) \\
&& +\sum_{j=1}^N (1-\eta_j)c_j(\varepsilon z)w_{j,x} +\sum_{j=1}^N (1-\eta_j)d_j(\varepsilon z) Z_j.
\end{eqnarray}
It is easy to see that for $\varepsilon>0$ small enough,
\[p\left(1-\sum_{j=1}^N\eta_j\right)|\mathcal{V}-1|^{p-2}(\mathcal{V}-1)\tilde{\psi}-p\sum_{j=1}^N \eta_j \tilde{\psi}\le -\frac p2\]
Now we consider the solutions to \eqref{e23} via fixed point theorem.

Integrating the both sides of \eqref{50} by $w_{j,x}$, we get
\begin{eqnarray*}
c_j(\varepsilon z)\int_\R w_x^2dx&=& -\int_\R \tilde{\chi}_jh w_{j,x}dx +p\int_\R \left[|\mathcal{V}-1|^{p-2}(\mathcal{V}-1)-|w_j-1|^{p-2}(w_j-1)\right]\tilde{\chi}_j\tilde{\phi}_j w_{j,x}dx \\
&&+p\int_\R \left[|\mathcal{V}-1|^{p-2}(\mathcal{V}-1)+1\right]\tilde{\chi}_j \tilde{\psi}w_{j,x}dx.
\end{eqnarray*}
Then
\[\|(1-\eta_j)c_j(\varepsilon z)w_{j,x}\|_{L^2(\mathfrak{S})}\leq C\varepsilon^{1/2}\left[\|h\|_{L^2(\mathfrak{S})}+\sum_{j=1}^N \|\tilde{\phi}_j\|_{H^2(\mathfrak{S})}+\|\tilde{\psi}\|_{H^2(\mathfrak{S})}\right].\]
From the same method, we also get
\[\|(1-\eta_j)d_j(\varepsilon z)w_{j,x}\|_{L^2(\mathfrak{S})}\leq C\varepsilon^{1/2}\left[\|h\|_{L^2(\mathfrak{S})}+\sum_{j=1}^N \|\tilde{\phi}_j\|_{H^2(\mathfrak{S})}+\|\tilde{\psi}\|_{H^2(\mathfrak{S})}\right].\]
With the help of the Sobolev imbedding theorem, we have
\[\|\eta_{j,xx}\tilde{\phi}_j(\tilde{\psi})+2\eta_{j,x}\tilde{\phi}_j(\tilde{\psi})\|_{L^2(\mathfrak{S})}\le \frac{C}{|\log\varepsilon|^{1/2}}\|\tilde{\phi}_j\|_{H^2(\mathfrak{S})}\]
and
\[\left\|a_{11}\tilde{\alpha}^{1-p}\left(\eta_{j,zz}\tilde{\phi}_j(\tilde{\psi})+2\eta_{j,x}\tilde{\phi}_{j,x}(\tilde{\psi})\right)\right\|_{L^2(\mathfrak{S})}\leq C\varepsilon^{\frac12}\|\tilde{\phi}_j\|_{H^2(\mathfrak{S})}.\]
Using the fixed point theorem together with \eqref{e24} and \eqref{e70}, we find that \eqref{54} has a unique solution, which further satisfies the estimate
\begin{equation}\label{e75}
\|\tilde{\psi}\|_{H^2(\mathfrak{S})}\leq C\|h\|_{L^2(\mathfrak{S})}.
\end{equation}

As a conclusion,  \eqref{48} has a unique solution, which satisfies the estimate \eqref{e74}.

Now we prove the estimate \eqref{e73}, which we prove by estimating the Lipschitz dependence of $\tilde{\phi}_j$ and $\tilde{\psi}$ on $\mathbf{f}$ and $\mathbf{e}$.

\noindent\textbf{Claim:} $\tilde{\phi}_j(\tilde{\psi})$ is Lipschitz continuous on  $\mathbf{f}$, $\mathbf{e}$ and $\tilde{\psi}$:
\begin{eqnarray}\label{e76}
\nonumber&&\|\tilde{\phi}_j(\mathbf{f}_1,\mathbf{e}_1,\tilde{\psi}_1)-\tilde{\phi}_j(\mathbf{f}_2,\mathbf{e}_2,\tilde{\psi}_2)\|_{H^2(\mathfrak{S})} \\
&\le& C\|h\|_{L^2(\mathfrak{S})}\left[\|\mathbf{f}_1-\mathbf{f}_2\|_{H^2(0,1)}+\|\mathbf{e}_1-\mathbf{e}_2\|_*\right]+C\|\tilde{\psi}_1-\tilde{\psi}_2\|_{H^2(\mathfrak{S})}.
\end{eqnarray}

Denote $\tilde{\phi}_{ji}=\tilde{\phi}_j(\mathbf{f}_i,\mathbf{e}_i,\tilde{\psi}_i)$, $i=1,2$. Then it satisfies the following equation:
\begin{eqnarray*}
\nonumber \tilde{\phi}_{ji}&=&T_{j,\mathbf{f}_i}\left\{\tilde{\chi}_{ji} h+p\left[|w_{ji}-1|^{p-2}(w_{ji}-1)-|\mathcal{V}^{(i)}-1|^{p-2}(\mathcal{V}^{(i)}-1)\right]\tilde{\chi}_{ji} \tilde{\phi}_{ji} \right. \\ &&\left.-p\left[|\mathcal{V}^{(i)}-1|^{p-2}(\mathcal{V}^{(i)}-1)+1\right]\tilde{\chi}_{ji}\tilde{\psi}_i, \Lambda_{j1}^{(i)}, \Lambda_{j2}^{(i)}\right\},
\end{eqnarray*}
where
\[\tilde{\chi}_{ji}=\chi\left(\frac{x-\beta(\varepsilon z)f_{ji}(\varepsilon z)}R\right), \qquad \mathcal{V}^{(i)}=\mathcal{V}(\mathbf{f}_i,\mathbf{e}_i),\]
\[\Lambda_{j1}^{(i)}=\int_\R(1-\eta_{ji})\tilde{\phi}_j w_{ji,x}dx-\sum_{k\not=j}\int_\R\eta_{ki} \tilde{\phi}_{ki} w_{ji,x}dx-\int_\R \tilde{\psi}_iw_{ji,x}dx,\]
\[\Lambda_{j2}^{(i)}=\int_\R(1-\eta_{ji})\tilde{\phi}_j Z_{ji}dx-\sum_{k\not=j}\int_\R\eta_{ki} \tilde{\phi}_{ki} Z_{ji}dx-\int_\R \tilde{\psi}_iZ_{ji}dx\]
and
\[\eta_{ji}=\eta\left(\frac{x-\beta(\varepsilon z)f_{ji}(\varepsilon z)}R\right).\]
Let $\bar{\phi}_j=\tilde{\phi}_{j1}-\tilde{\phi}_{j2}$. Then it satisfies the following equation
\begin{eqnarray*}
\bar{\phi}_j&=&\left(T_{j,\mathbf{f}_1}-T_{j, \mathbf{f}_2}\right)\left\{\tilde{\chi}_{j1} h+p\left[|w_{j1}-1|^{p-2}(w_{j1}-1)-|\mathcal{V}^{(1)}-1|^{p-2}(\mathcal{V}^{(1)}-1)\right]\tilde{\chi}_{j1} \tilde{\phi}_{j1} \right. \\ &&\left.-p\left[|\mathcal{V}^{(1)}-1|^{p-2}(\mathcal{V}^{(1)}-1)+1\right]\tilde{\chi}_{j1}\tilde{\psi}_1, \Lambda_{j1}^{(1)}, \Lambda_{j2}^{(1)}\right\} \\
&&+ T_{j, \mathbf{f}_2}\left\{(\tilde{\chi}_{j1}-\tilde{\chi}_{j2})h -p[|\mathcal{V}^{(1)}-1|^{p-2}(\mathcal{V}^{(1)}-1)-|\mathcal{V}^{(2)}-1|^{p-2}(\mathcal{V}^{(2)}-1)]\tilde{\chi}_{j2}\tilde{\psi}_2\right. \\
&&-p(\tilde{\chi}_{j1}-\tilde{\chi}_{j2})\left[|\mathcal{V}^{(1)}-1|^{p-2}(\mathcal{V}^{(1)}-1)+1\right]\tilde{\psi}_1 -p\left[|\mathcal{V}^{(1)}-1|^{p-2}(\mathcal{V}^{(1)}-1)+1\right]\tilde{\chi}_{j2}(\tilde{\psi}_1-\tilde{\psi_2}) \\
&&+p\left[|w_{j1}-1|^{p-2}(w_{j1}-1)-|\mathcal{V}^{(1)}-1|^{p-2}(\mathcal{V}^{(1)}-1)\right]\tilde{\chi}_{j1}\tilde{\phi}_1  \\
&&\left.- p\left[|w_{j2}-1|^{p-2}(w_{j2}-1)-|\mathcal{V}^{(2)}-1|^{p-2}(\mathcal{V}^{(2)}-1)\right]\tilde{\chi}_{j2}\tilde{\phi}_2, \Lambda_{j1}^{(1)}-\Lambda_{j1}^{(2)},  \Lambda_{j2}^{(1)}-\Lambda_{j2}^{(2)}\right\}.
\end{eqnarray*}
From \eqref{e24}, \eqref{e75}, Lemma \ref{lme4} and direct computation, we get \eqref{e76}.

From the same method, we get
\begin{equation}\label{e77}
\|\tilde{\psi}(f_1,e_1)-\tilde{\psi}(f_2,e_2)\|_{H^2(\mathfrak{S})}\le \|h\|_{L^2(\mathfrak{S})}\left[\|\mathbf{f}_1-\mathbf{f}_2\|_{H^2(0,1)}+\|\mathbf{e}_1-\mathbf{e}_2\|_*\right].
\end{equation}
From \eqref{e76}, \eqref{e77} and direct computation, we get \eqref{e78}.
\end{proof}

Now we consider the following problem
\begin{equation}\label{e25}
\left\{\begin{array}{ll}
\mathcal{L}(\phi)=h+\sum_{j=1}^N c_j(\varepsilon z)w_{j,x} +\sum_{j=1}^N d_j(\varepsilon z)Z_j, &\mbox{in $\mathfrak{S}$,} \\
\int_\R \phi(z,x)w_{j,x}(x)dx= \int_\R \phi(z,x)Z_j(x)dx=0, &z\in (0,1/\varepsilon), \\
\phi(0,x)=\phi(1/\varepsilon,x), \quad \phi_z(0,x)=\phi_z(1/\varepsilon,x), &x\in \R.
\end{array}\right.
\end{equation}
Recall the operator $\mathcal{L}$ is defined in \eqref{62}. In the view of Proposition \ref{pre1}, \eqref{e25} is written into the following fixed point problem
\[\phi=\tilde{T}[-B_5(\phi)+h].\]
From the definition of $B_5(\phi)$, we get
\[\|B_5(\phi)\|_{L^2(\mathfrak{S})}\leq C\delta\|\phi\|_{H^2(\mathfrak{S})}.\]
Moreover, $B_5(\cdot)$ is Lipschitz dependent on $\mathbf{f}$ and $\mathbf{e}$:
\[\|B_{5,\mathbf{f}_1,\mathbf{e}_1}(\phi)-B_{5,\mathbf{f}_2,\mathbf{e}_2}(\phi)\|_{L^2(\mathfrak(S))}\le C\varepsilon^2\left[ \|\mathbf{f}_1-\mathbf{f}_2\|_{H^2(\mathfrak{S})}+\|\mathbf{e}_1-\mathbf{e}_2\|_*\right]\|\phi\|_{H^2(\mathfrak{S})}.\]
With the help of Proposition \ref{pre1} and fixed point theorem, we get the  following proposition.
\begin{proposition}\label{pre2}
Fix $\delta>0$ small enough, \eqref{e25} has a unique solution $\phi=T(h)$, where $T$ is a linear operator satisfies the following estimate
\begin{equation}\label{e78}
\|\phi\|_{H^2(\mathfrak{S})}\leq C\|h\|_{L^2(\mathfrak{S})}.
\end{equation}
Here $C>0$ is a constant independent of $\varepsilon$ and the choice of $\mathbf{f}$ and $\mathbf{e}$.

Moreover, $T$ is Lipschitz continously of $\mathbf{f}$ and $\mathbf{e}$:
\begin{equation}\label{e26}
\|T_{\mathbf{f}_1,\mathbf{e}_1}-T_{\mathbf{f}_2,\mathbf{e}_2}\|\le C\left[ \|\mathbf{f}_1-\mathbf{f}_2\|_{H^2(\mathfrak{S})}+\|\mathbf{e}_1-\mathbf{e}_2\|_*\right].
\end{equation}
\end{proposition}


\begin{remark}\label{rm2}
For $\delta>0$ small enough and a function $h$ supported in $|x|\le 2\delta/\varepsilon$, the function $\phi=T(h)$ satisfies the following estimate
\begin{equation}\label{e64}
|\phi(z,x)|+|\nabla \phi(z,x)|\le C\|\phi\|_{L^\infty}e^{-\frac{\sqrt{p}\delta}{3\varepsilon}}.
\end{equation}
\end{remark}
\begin{proof}
According to the definition of $\mathcal{L}$, we get $\phi$ satisfies an equation of the following form
\[\phi_{xx}+a_{11}\tilde{\alpha}^{1-p}\phi_{zz}+(-p+o(1))\phi+O(\delta)(|D^2\phi|+\varepsilon|\nabla\phi|)=0, \qquad \mathrm{for} \qquad |x|\ge \frac{7\delta}{3\varepsilon}.\]
Using the barrier function as the form
\[\varphi(z,x)=\|\phi\|_{L^\infty(\mathfrak{S})}e^{-\frac{\sqrt{p}}2(x-\frac{7\delta}{3\varepsilon})},\]
we get
\[|\phi(z,x)|\le C\|\phi\|_{L^\infty(\mathfrak{S})}e^{-\frac{\sqrt{p}}2(x-\frac{7\delta}{3\varepsilon})}\qquad \mathrm{for} \qquad |x|\ge \frac{7\delta}{3\varepsilon}.\]
Using local elliptic estimate, we get \eqref{e64}.
\end{proof}

\section{Nonlinear problem}\label{sect9}
In this section,  we will solve  \eqref{24}-\eqref{e2} via Proposition \ref{pre2} and Remark \ref{rm2}.
\begin{proposition}\label{pre4}
Fix the constant $p>2$. Then there exists a constant $D>0$ such that for $\varepsilon>0$ small enough and $(\mathbf{f},\mathbf{e})$ satisfying \eqref{66} and \eqref{11}, problem \eqref{24}-\eqref{e2} has a unique solution $\hat{\varphi}=\hat{\varphi}(\mathbf{f},\mathbf{e})$, which satisfies
\[\|\hat{\varphi}\|_{H^2(\mathfrak{S})}\leq D\varepsilon^{\frac32-\mu}\]
as well as
\begin{equation}\label{e68}
\|\hat{\varphi}\|_{L^\infty(|x|>3\delta/\varepsilon)}+\|\nabla\hat{\varphi}\|_{L^\infty(|x|>3\delta/\varepsilon)}\le \|\hat{\varphi}\|_{H^2(\mathfrak{S})}e^{-\frac{\sqrt{p}\delta}{4\varepsilon}}.
\end{equation}
Furthermore,  $\hat{\varphi}$ depends Lipschitz continuously on $\mathbf{f}$ and $\mathbf{e}$, \textit{i.e.}
\begin{equation}\label{e27}
\|\hat{\varphi}_{\mathbf{f}_1, \mathbf{e}_1}-\hat{\varphi}_{\mathbf{f}_2, \mathbf{e}_2}\|_{H^2(\mathfrak{S})}\le C\varepsilon^{\frac32-\mu}\left[\|\mathbf{f}_1-\mathbf{f}_2\|_{H^2(0,1)}+\|\mathbf{e}_1-\mathbf{e}_2\|_*\right].
\end{equation}
\end{proposition}
\begin{proof}
With the help of Proposition \ref{pre1}, \eqref{24}-\eqref{43} is written into the following fixed point problem
\begin{equation}\label{71}
\phi=\mathcal{A}(\phi),
\end{equation}
where
\[\mathcal{A}(\phi):=T\left\{-\eta_\delta^\varepsilon \mathcal{E}_2-\eta_\delta^\varepsilon \hat{N}(\phi+\hat{\psi}(\phi))-p\eta_\delta^\varepsilon\left[\left| \mathcal{V}-\tilde{\alpha}^{-1}\bar{\mathbf{q}}\right|^{p-2}\left( \mathcal{V}- \tilde{\alpha}^{-1}\bar{\mathbf{q}}\right)+\left|\tilde{\alpha}^{-1}\bar{\mathbf{q}}\right|^{p-2}\left(\tilde{\alpha}^{-1}\bar{\mathbf{q}}\right) \right]\hat{\psi}(\phi)\right\}.\]

Consider the following closed, bounded subset of $H^2(\mathfrak{S})$:
\[\mathfrak{B}=\left\{\|\phi\|_{H^2(\mathfrak{S})}\le D\varepsilon^{\frac32-\mu}, \quad \|\phi\|_{L^\infty(|x|>3\delta/\varepsilon)}+\sup_{|x|\ge \frac{3\delta}\varepsilon-1}\|\nabla\phi\|_{L^2(B_1(z,x)}\le \|\phi\|_{H^2(\mathfrak{S})}e^{-\frac{\sqrt{p}\delta}{4\varepsilon}}\right\}.\]
Here $D>0$ is a constant large enough, which we will claim later. We will show $\mathcal{A}$ is a contraction map from $\mathfrak{B}$ into itself.

From the definition of $\hat{N}(\phi)$, we get
\[|\hat{N}(\phi)|\le C\left[|\phi|^2+|\phi|^p\right].\]
Let
\[\tilde{N}_1(\phi)=\eta_\delta^\varepsilon \hat{N}(\phi+\hat{\psi}(\phi)).\]
Then we get
\[\|\tilde{N}_1(\phi)\|_{L^2(\mathfrak{S})}\le C\left[\|\phi\|_{L^4(\mathfrak{S})}^2+\|\phi\|_{L^{2p}(\mathfrak{S})}^p +\|\psi(\phi)\|_{L^4(\mathfrak{S}_\delta)}^2+\|\psi(\phi)\|_{L^{2p}(\mathfrak{S}_\delta)}^p\right].\]
where $\mathfrak{S}_\delta=\mathfrak{S}\cap \{|x|\le 2\delta /\varepsilon\}$. However, for the last two terms in the inequality above, we get
\[
\|\psi(\phi)\|_{L^4(\mathfrak{S}_\delta)}^2+\|\psi(\phi)\|_{L^{2p}(\mathfrak{S}_\delta)}^p\leq Ce^{-\frac{\sqrt{p}\delta}{4\varepsilon}}\left[1+\|\phi\|_{H^2(\mathfrak{S})}^2+\|\phi\|_{H^2(\mathfrak{S})}^p\right]+C\varepsilon^{-1}\|\phi \|_{H^2(\mathfrak{S})}^{2\min\{p,2\}}.
\]
For $\phi\in \mathfrak{B}$, we get
\[\|\tilde{N}_1(\phi)\|_{L^2(\mathfrak{S})}\leq C(\varepsilon^{\frac32-\mu})^{\min\{p,2\}}.\]

Now we estimate the Lipschitz property of $N_1$. From the definition of $\hat{N}(\phi)$ we have
\[|\hat{N}'(\phi)|\leq C|\phi|^{\min\{p-1,1\}}.\]
For $\phi_1, \phi_2\in \mathfrak{B}$,  we have
\begin{eqnarray*}
\|\tilde{N}_1(\phi_1)-\tilde{N}_1(\phi_2)\|_{L^2(\mathfrak{S})}&\leq& A \left[\|\phi_1-\phi_2\|_{L^4(\mathfrak{S})}+\|\phi_1-\phi_2\|_{L^{2p}(\mathfrak{S})} \right.\\ &&\qquad \left.+\|\psi(\phi_1)-\psi(\phi_2)\|_{L^4(\mathfrak{S_\delta})} +\|\psi(\phi_1)-\psi(\phi_2)\|_{L^{2p}(\mathfrak{S_\delta})}\right],
\end{eqnarray*}
where $A=A_1+A_2$ and
\[A_l=\|\phi_l\|_{L^{2p}(\mathfrak{S})}^{p-1}+\|\phi_l\|_{L^2(\mathfrak{S})}+\|\psi(\phi_l)\|_{L^{2p}(\mathfrak{S}_\delta)}^{p-1}+\|\psi(\phi_l)\|_{L^4(\mathfrak{S}_\delta)}, \qquad \mathrm{for} \quad l=1,2.\]
Since
\begin{eqnarray*}
&&\|\psi(\phi_l)\|_{L^{2p}(\mathfrak{S}_\delta)}^{p-1}+\|\psi(\phi_l)\|_{L^4(\mathfrak{S}_\delta)} \\
 &\le& |\mathfrak{S}_\delta|^{1/2}\|\psi(\phi)\|_{L^\infty} +|\mathfrak{S}_\delta|^{\frac{p-1}{2p}}\|\psi(\phi)\|_{L^\infty}^{p-1} \\
&\le&Ce^{-\frac{\sqrt{p}\delta}{4\varepsilon}}\left[1+\|\phi\|_{H^2(\mathfrak{S})}+\|\phi\|_{H^2(\mathfrak{S})}^{p-1}\right] +\varepsilon^{-\frac{p-1}{2p}}\|\phi\|_{H^2(\mathfrak{S})}^{(p-1)\min\{p,2\}}+\varepsilon^{-\frac12}\|\phi\|_{H^2(\mathfrak{S})}^{\min\{p,2\}}.
\end{eqnarray*}
We get
\begin{equation}\label{e28}
\|\tilde{N}_1(\phi_1)-\tilde{N}_1(\phi_2)\|_{L^2(\mathfrak{S})}\le \varepsilon^{\min\{p-1,1\}(\frac32-\mu)}\|\phi_1-\phi_2\|_{H^2(\mathfrak{S})}.
\end{equation}

Let
\[\tilde{N}_2(\phi)=p\eta_\delta^\varepsilon\left[\left| \mathcal{V}-\tilde{\alpha}^{-1}\bar{\mathbf{q}}\right|^{p-2}\left( \mathcal{V}- \tilde{\alpha}^{-1}\bar{\mathbf{q}}\right)+\left|\tilde{\alpha}^{-1}\bar{\mathbf{q}}\right|^{p-2}\left(\tilde{\alpha}^{-1}\bar{\mathbf{q}}\right)\right]\hat{\psi}(\phi).\]
It is apparent that
\begin{equation*}
\|\tilde{N}_2(\phi)\|_{L^2(\mathfrak{S})}\leq C\varepsilon^{-1/2}\left[ e^{-\frac{\sqrt{p}\delta}{2\varepsilon}} +\varepsilon\|\phi\|_{H^2(\mathfrak{S})}e^{-\frac{\sqrt{p}\delta}{4\varepsilon}}+\|\phi\|_{H^2(\mathfrak{S})}^{\min\{p,2\}}\right] \leq C\varepsilon^{-1/2}\varepsilon^{\min\{p,2\}}
\end{equation*}
and
\begin{eqnarray*}
\|\tilde{N}_2(\phi_1)-\tilde{N}_2(\phi_2)\|_{L^2(\mathfrak{S})} &\le&C\varepsilon^{-1/2}\left[\varepsilon\|\phi_1-\phi_2\|_{H^2(\mathfrak{S})}e^{-\frac{\sqrt{p}\delta}{4\varepsilon}} +(\varepsilon^{\frac32-\mu})^{\min\{p-1,1\}}\|\phi_1-\phi_2\|_{H^2(\mathfrak{S})}\right] \\
&\le& C\varepsilon^{-1/2}(\varepsilon^{\frac32-\mu})^{\min\{p-1,1\}}\|\phi_1-\phi_2\|_{H^2(\mathfrak{S})}.
\end{eqnarray*}

Then for $\phi\in \mathfrak{B}$, we have
\[\|\mathcal{A}(\phi)\|_{H^2(\mathfrak{S})}\le C_0\varepsilon^{3/2}|\log\varepsilon|^q.\]
For $\varepsilon>0$ small enough,
\[\|\mathcal{A}(\phi_1)-\mathcal{A}(\phi_2)\|_{H^2(\mathfrak{S})}\le \frac12\|\phi_1-\phi_2\|_{H^2(\mathfrak{S})}.\]

Then the fixed point problem \eqref{71} has a unique solution $\hat{\varphi}$ satisfying $\|\hat{\varphi}\|_{H^2(\mathfrak{S})}\leq C\varepsilon^{\frac32-\mu}$. From Remark \ref{rm2} and Sobolev embedding theorem, we get $\hat{\varphi}\in \mathfrak{B}$ and \eqref{e68} holds.

Now we estimate \eqref{e27}. It is easy to get
\[\partial_{f_k}\hat{N}(\phi)=p\left\{|\mathcal{V}-\tilde{\alpha}\bar{\mathbf{q}}+\phi|^{p-2}(\mathcal{V}-\tilde{\alpha}\bar{\mathbf{q}}+\phi)-|\mathcal{V}-\tilde{\alpha}\bar{\mathbf{q}} |^{p-2}(\mathcal{V}-\tilde{\alpha}\bar{\mathbf{q}})-(p-1)|\mathcal{V}-\tilde{\alpha}\bar{\mathbf{q}} |^{p-2}\phi\right\}\partial_{f_k}\mathcal{V}.\]
Then
\[|\partial_{f_k}\hat{N}(\phi)|\leq C\left[|\phi|^{p-1}+|\phi|^2\right].\]
We can also get the similar estimate of $\partial_{e_k}\hat{N}(\phi)$. Using the similar method as in \eqref{e28}, we get
\begin{eqnarray}\label{e30}
\|\tilde{N}_{1,\mathbf{f}_1,\mathbf{e}_1}-\tilde{N}_{1,\mathbf{f}_2,\mathbf{e}_2}\|_{L^2(\mathfrak{S})}&\leq & C(\varepsilon^{\frac32-\mu})^{\min\{p-1,2\}}\left[\|\mathbf{f}_1- \mathbf{f}_2\|_{H^2(0,1)}+\|\mathbf{e}_1-\mathbf{e}_2\|_*\right] \\
\nonumber && +C\varepsilon^{-1}(\varepsilon^{\frac32-\mu})^{\min\{p,2\}\min\{p-1,2\}} \left[\|\mathbf{f}_1- \mathbf{f}_2\|_{H^2(0,1)}+\|\mathbf{e}_1-\mathbf{e}_2\|_*\right].
\end{eqnarray}
From the same method and \eqref{e29}, we get
\begin{equation}\label{e31}
\|\tilde{N}_{2,\mathbf{f}_1,\mathbf{e}_1}-\tilde{N}_{2,\mathbf{f}_2,\mathbf{e}_2}\|_{L^2(\mathfrak{S})}\leq C\varepsilon^{-1/2}(\varepsilon^{\frac32-\mu})^{\min\{p,2\}}\left[\|\mathbf{f}_1- \mathbf{f}_2\|_{H^2(0,1)}+\|\mathbf{e}_1-\mathbf{e}_2\|_*\right].
\end{equation}
Hence we get \eqref{e27} from \eqref{e30}, \eqref{e31}, \eqref{e32}.
\end{proof}
From Proposition \ref{pre4}, we know the solution $\hat{\varphi}$ satisfies the assumption \eqref{e69} before solving the outer problem \eqref{4}.

\section{Reduced problem}\label{sect7}
To solve the inner problem, it suffices to set the constants $c_j(\varepsilon z)$'s and $d_j(\varepsilon z)$'s in \eqref{24} to zero. We only need to solve the following problems:
\begin{eqnarray}\label{77}
\nonumber&&\int_\R \mathcal{L}(\hat{\varphi})w_{j,x}dx +\int_\R p\eta_\delta^\varepsilon \left[|\mathcal{V}-\tilde{\alpha}^{-1}\bar{\mathbf{q}}|^{p-2}(\mathcal{V}-\tilde{\alpha}^{-1}\bar{\mathbf{q}})+|\tilde{\alpha}^{-1}\bar{\mathbf{q}}|^{p-2}(\tilde{\alpha}^{-1}\bar{\mathbf{q}}) \right]\hat{\psi}( \hat{\varphi})w_{j,x}dx \\
&& \qquad +\int_\R \eta_\delta^\varepsilon S(\mathcal{V})w_{j,x}dx +\int_\R \eta_\delta^\varepsilon \hat{N}(\hat{\varphi}+\hat{\psi}(\hat{\varphi}))w_{j,x}dx=0
\end{eqnarray}
and
\begin{eqnarray}\label{82}
\nonumber&&\int_\R \mathcal{L}(\hat{\varphi})Z_jdx +\int_\R p\eta_\delta^\varepsilon \left[|\mathcal{V}-\tilde{\alpha}^{-1}\bar{\mathbf{q}}|^{p-2}(\mathcal{V}-\tilde{\alpha}^{-1}\bar{\mathbf{q}})+|\tilde{\alpha}^{-1}\bar{\mathbf{q}}|^{p-2}(\tilde{\alpha}^{-1}\bar{\mathbf{q}})\right]\hat{\psi}( \hat{\varphi})Z_jdx \\
&& \qquad +\int_\R \eta_\delta^\varepsilon S(\mathcal{V})Z_jdx +\int_\R \eta_\delta^\varepsilon \hat{N}(\hat{\varphi}+\hat{\psi}(\hat{\varphi}))Z_jdx=0.
\end{eqnarray}
In this section, we treat the terms appearing in \eqref{77} and \eqref{79} as the function of $\theta= \varepsilon z$.

We first consider the equation \eqref{77}. Through direct computation,
\begin{eqnarray}\label{79}
\nonumber\int_\R \mathcal{L}(\hat{\varphi})w_{j,x}dx&=&a_{11}\tilde{\alpha}^{1-p}\int_\R \hat{\varphi}_{zz}w_{j,x}dx +p\int_R\left[|\mathcal{V}-1|(\mathcal{V}-1)-|w_j-1|^{p-2}(w_j-1)\right] \hat{\varphi}w_{j,x}dx \\
\nonumber&& +\int_\R p\eta_{6\delta}^\varepsilon \left[|\eta_{3\delta}^\varepsilon \mathcal{V}-\tilde{\alpha}^{-1}\bar{\mathbf{q}}|^{p-2}(\eta_{3\delta}^\varepsilon \mathcal{V}-\tilde{\alpha}^{-1}\bar{\mathbf{q}})-|\mathcal{V}-1|^{p-2}(\mathcal{V}-1)\right]\hat{\varphi}w_{j,x}dx \\
&& +\int_\R \eta_{6\delta}^\varepsilon (B_4^1(\hat{\varphi})+ B_4^3(\hat{\varphi}))w_{j,x}dx.
\end{eqnarray}
 From \eqref{43}, we get
\begin{equation*}
\int_\R \hat{\varphi}_{zz}w_{j,x}dx=2\varepsilon(\beta f_j)'\int_\R \hat{\varphi}_zw_{j,xx}dx+\varepsilon^2(\beta f_j)''\int_\R\hat{\varphi}w_{j,xx}dx-\varepsilon^2 |(\beta f_j)'|^2\int_\R \hat{\varphi}w_{j,xxx}dx=: I_1+I_2+I_3.
\end{equation*}
Then
\begin{equation*}
\|I_1\|_{L^2(0,1)}\le C\varepsilon\|f_j\|_{H^2(0,1)}\left(\int_0^1d\theta\int_\R \left|\hat{\varphi}_z(\theta/\varepsilon,x)\right|^2dx\right)^{1/2}\le C\varepsilon^{3/2}\|f_j\|_{H^2(0,1)}\|\hat{\varphi}\|_{H^2(\mathfrak{S})}\leq C\varepsilon^{\frac52+\mu_1},
\end{equation*}
where $\mu_1\in(0,1/2)$ is a constant.
By the same method, we also have
\[\|I_3\|_{L^2(0,1)}\leq C\varepsilon^3.\]
For the term $I_2$,
\begin{eqnarray*}
\|I_2\|_{L^2(0,1)}^2&\le& C\varepsilon^4\int_0^1 \left[|f_j|^2 +|f_j'|^2+|f_j''|^2\right]\left|\int_\R\hat{\varphi}(\theta/\varepsilon,x)w_{j,x}dx\right|^2 \\
&\le&  \varepsilon^4\int_0^1 \left[|f_j|^2 +|f_j'|^2+|f_j''|^2\right]\int_\R |\hat{\varphi}(\theta/\varepsilon,x)|^2 dx.
\end{eqnarray*}
Let
\[F(\theta)= \int_\R |\hat{\varphi}(\theta/\varepsilon,x)|^2 dx.\]
Direct computation yields
\[\|F\|_{W^{1,1}(0,1)}\leq C\|\hat{\varphi}\|_{H^2(\mathfrak{S})}^2.\]
Applying the Sobolev embedding theorem(\cite[Theorem 8.8]{Brezis2011}), we deduce
\[\sup_{\theta\in(0,1)}|F(\theta)|\leq \|F\|_{W^{1,1}(0,1)}\leq C\|\hat{\varphi}\|_{H^2(\mathfrak{S})}^2.\]
We thus obtain
\[\|I_2\|_{L^2(0,1)}\leq C\varepsilon^2\|f_j\|_{H^2(0,1)}\|\hat{\varphi}\|_{H^2(\mathfrak{S})}\leq C\varepsilon^3.\]

Next, let
\[I_7=p\int_\R \left[|\mathcal{V}-1|^{p-2}(\mathcal{V}-1)-|w_j-1|^{p-2}(w_j-1)\right]\hat{\varphi}w_{j,x}dx.\]
Then
\[|I_7|\le \left(\int_\R |\hat{\varphi}(\theta/\varepsilon,x)|^2dx\right)^{1/2}\left(\int_\R \left[|\mathcal{V}-1|^{p-2}(\mathcal{V}-1)-|w_j-1|^{p-2}(w_j-1)\right]^2w_{j,x}^2dx\right)^{1/2}.\]
It follows that
\[\|I_7\|_{L^2(0,1)}\le C\varepsilon^{3/2}\|\hat{\varphi}\|_{L^2(\mathfrak{S})}\le C\varepsilon^{\frac52+\mu_1}.\]

We now define
\[I_4=\int_\R \eta_{6\delta}^\varepsilon (B_4^1(\hat{\varphi})+ B_4^3(\hat{\varphi}))w_{j,x}dx.\]
From the expressions of $B_4^1$ and $B_4^3$, we have
\[\|I_4\|_{L^2(0,1)}\leq C\varepsilon^{\frac52+\mu_1}.\]

Finally, let
\[I_5=\int_\R p\eta_{6\delta}^\varepsilon \left[|\eta_{3\delta}^\varepsilon \mathcal{V}-\tilde{\alpha}^{-1}\bar{\mathbf{q}}|^{p-2}(\eta_{3\delta}^\varepsilon \mathcal{V}-\tilde{\alpha}^{-1}\bar{\mathbf{q}})-|\mathcal{V}-1|^{p-2}(\mathcal{V}-1)\right]\hat{\varphi}w_{j,x}dx.\]
It is straightforward to show that  $\|I_5\|_{L^2(0,1)}\leq C\varepsilon^{\frac52+\mu_1}$.

In the expression of \eqref{79}, we isolate all the non-regular terms. The first such term is given by
\[\tilde{I}_2=\varepsilon^2 \beta f_j''\int_\R \hat{\varphi}(\theta/\varepsilon,z)w_{j,x}dx.\]
From \eqref{e27} and direct computation, we obtain
\[\|\tilde{I}_2(\mathbf{e}_1,\mathbf{f}_1)-\tilde{I}_2(\mathbf{e}_2,\mathbf{f}_2)\|_{L^2(0,1)}\leq C\varepsilon^3\left[\|\mathbf{f}_1-\mathbf{f}_2\|_{H^2(0,1)}+\|\mathbf{e}_1-\mathbf{e}_2\|_*\right].\]
The remaining non-regular term $\tilde{I}_4$ in  \eqref{79}, however, arises from those terms in $B_4^1(\phi)$ and $B_4^3(\phi)$ that contain $\hat\varphi_{zz}$. Specifically, we have
\[\tilde{I}_4=\int_\R \eta_{6\delta}^\varepsilon \tilde{\alpha}^{1-p}(X_1-a_{11})\hat{\varphi}_{zz}w_{j,x}dx,\]
where $X_1$ is defined in Lemma \ref{lm8}.
It is easy to get
\[\|\tilde{I}_4(\mathbf{e}_1,\mathbf{f}_1)-\tilde{I}_4(\mathbf{e}_2,\mathbf{f}_2)\|_{L^2(0,1)}\leq C\varepsilon^{\frac52+\mu_1}\left[\|\mathbf{f}_1-\mathbf{f}_2\|_{H^2(0,1)}+\|\mathbf{e}_1-\mathbf{e}_2\|_*\right].\]
Following the same argument as in \cite{delPino_Kowalczyk_Wei2007,Wang_Zhao2019}, we conclude that the operator
\[\int_\R \mathcal{L}(\hat{\varphi})w_{j,x}dx-\tilde{I}_2-\tilde{I}_4\]
is compact.

For the remaining terms in \eqref{77}, we define
\[\Lambda_1=\int_\R p\eta_\delta^\varepsilon \left[|\mathcal{V}-\tilde{\alpha}^{-1}\bar{\mathbf{q}}|^{p-2}(\mathcal{V}-\tilde{\alpha}^{-1}\bar{\mathbf{q}})+|\tilde{\alpha}^{-1}\bar{\mathbf{q}}|^{p-2}(\tilde{\alpha}^{-1}\bar{\mathbf{q}}) \right]\hat{\psi}( \hat{\varphi})w_{j,x}dx\]
and
\[\Lambda_2=\int_\R \eta_\delta^\varepsilon \hat{N}(\hat{\varphi}+\hat{\psi}(\hat{\varphi}))w_{j,x}dx.\]
Direct computation gives
\[\|\Lambda_1\|_{L^2(0,1)}\leq C\|\psi(\hat{\varphi})\|_{L^\infty }\le C(\varepsilon^{3/2-\mu})^{\min\{p,2\}}\le C\varepsilon^{\frac52+\mu_1}\]
and
\[\|\Lambda_2\|_{L^2(0,1)}\leq C\varepsilon^{1/2}(\varepsilon^{3/2-\mu})^{\min\{p,2\}}+(\varepsilon^{3/2-\mu})^{2\min\{p,2\}}\le C\varepsilon^{\frac52+\mu_1}.\]\
From \eqref{e94}, we rewrite \eqref{77}  into the following form
\begin{eqnarray}\label{73}
\nonumber&&\varepsilon^2\alpha^{1-p}\beta\left\{-a_{11}f_j''+\left[a_{22}-b_{11}-a_{11}\left(\beta^{-1}\beta'+2\alpha^{-1}\alpha'\right)\right]f_j'+\left[a_{22}(\beta^{-1}\beta'+2\alpha^{-1}\alpha') +b_{22}-a_{33}\right]f_j\right. \\
\nonumber&&+\left[\frac{p+3}2\alpha^{p-2}\beta^{-2} \mathbf{q}_{tt}+\frac{p+2}{2(p+3)}\alpha^{1-p}\beta^2(a_{32})^2-\frac{p+1}{p+3}\alpha^{1-p}\beta^2 a_{32}b_{21}-\frac2{p+3}\alpha^{1-p}\beta^2 (b_{21})^2\right]f_j \\
\nonumber&&\left.+\left[-2\varepsilon a_{11}e_j' f_j'+2\varepsilon a_{22}e_j' f_j\right]\left(\int_\R w_x^2 dx\right)^{-1}\int_\R Z_x w_xdx\right\} \\
&&C_0 p\alpha_p\left[e^{-\sqrt{p}\beta(f_j-f_{j-1})}-e^{-\sqrt{p}\beta (f_{j+1}-f_j)}\right]=\varepsilon^2 M_{1j\varepsilon}.
\end{eqnarray}
Using the same procedure and \eqref{e95}, we rewrite \eqref{82} as
\begin{equation}\label{e80}
\varepsilon(\varepsilon^2a_{11}\alpha^{1-p}e_j''+\lambda_0 e_j) +p\alpha_p C_1\left[e^{-\sqrt{p}\beta(f_j-f_{j-1})}+e^{-\sqrt{p}\beta(f_{j+1}-f_j)}\right]=\varepsilon^2 M_{2j\varepsilon}.
\end{equation}
In the expression \eqref{73} and \eqref{e80}, $M_{ij\varepsilon}=A_{ij\varepsilon}+K_{ij\varepsilon}$, $i=1,2$, where $K_{ij\varepsilon}$ is a compact operator and $A_{ij\varepsilon}$ is a Lipschitz operator. These operators satisfy the following estimates:
\[\|A_{ij\varepsilon}(\mathbf{f}_1,\mathbf{e}_1)-A_{ij\varepsilon}(\mathbf{f}_2,\mathbf{e}_2)\|_{L^2(0,1)}\leq C\varepsilon^{\frac12+\mu_1}\left[\|\mathbf{f}_1-\mathbf{f}_2\|_{H^2(0,1)}+\|\mathbf{e}_1-\mathbf{e}_2\|_*\right]\]
and
\begin{equation}\label{e49}
\|A_{ij\varepsilon}\|_{L^2(0,1)}\leq C\varepsilon^{\frac12+\mu_1}, \qquad \|K_{ij\varepsilon}\|_{L^2(0,1)}\leq C\varepsilon^{\frac12+\mu_1}.
\end{equation}
\section{Proof of Theorem \ref{th1}}\label{sect8}
To prove Theorem \ref{th1}, it suffices to construct a solution to system \eqref{73}-\eqref{e80}. For this purpose, we require certain a priori estimates.

Let
\begin{equation}\label{e81}
\breve{f}_j(\theta)=\beta(\theta)f_j(\theta).
\end{equation}
Then we have
\[f_j=\beta^{-1}\breve{f}_j,\qquad f_j'=\beta^{-1}\breve{f}_j'-\beta^{-2}\beta' \breve{f}_j\]
and
\[f_j''=\beta^{-1}\breve{f}_j''-2\beta^{-2}\beta'\breve{f}_j'+\left[2\beta^{-3}(\beta')^2-\beta^{-2}\beta''\right]\breve{f}_j.\]
Hence \eqref{73} is transformed into the following problem
\begin{eqnarray}\label{93}
\nonumber&&\varepsilon^2 \alpha^{1-p}\left\{-a_{11}\breve{f}_j''+\left[a_{22}-b_{11}+a_{11}(\beta^{-1}\beta'-2\alpha^{-1}\alpha')\right]\breve{f}_j' -2\varepsilon a_{11}e_j'\breve{f}_j'\frac{\int_\R Z_xw_x dx}{\int_\R w_x^2 dx}\right. \\
\nonumber&&+\left[2\alpha^{-1}\alpha'a_{22}+\beta^{-1}\beta'b_{11}+b_{22}-a_{33}+\frac{p+3}2\alpha^{p-2}\beta^{-2}\mathbf{q}_{tt}+\frac{p+2}{2(p+3)}\alpha^{1-p} \beta^2 (a_{32})^2\right. \\
&&\left.-\frac{p+1}{p+3}\alpha^{1-p}\beta^2 a_{32}b_{21}-\frac{2}{p+3}\alpha^{1-p}\beta^2(b_{21})^2 +a_{11}\left(\beta^{-1}\beta''+2\alpha^{-1}\beta^{-1}\alpha'\beta'-\beta^{-2}(\beta')^2\right)\right]\breve{f}_j' \\
\nonumber&& \left.+\left(2\varepsilon a_{11}\beta^{-1}\beta'e_j'+2\varepsilon a_{22}e_j'\right)\breve{f}_j\frac{\int_\R Z_xw_x dx}{\int_\R w_x^2 dx}\right\} +p\alpha_p C_0 \left[e^{-\sqrt{p}(\breve{f}_j-\breve{f}_{j-1})}-e^{-\sqrt{p}(\breve{f}_{j+1}-\breve{f}_j)}\right]=\varepsilon^2 M_{1j\varepsilon}.
\end{eqnarray}

For notation simplicity, we denote
\begin{equation}\label{e103}
\Upsilon_2=\alpha^{1-p}a_{11}, \qquad  \Upsilon_1=\alpha^{1-p}\left[a_{22}-b_{11}+a_{11}(\beta^{-1}\beta'-2\alpha^{-1}\alpha')\right],
\end{equation}
\begin{eqnarray}\label{e104}
\Upsilon_0&=&-\alpha^{1-p}\left[ 2\alpha^{-1}\alpha'a_{22}+\beta^{-1}\beta'b_{11}+b_{22}-a_{33}+\frac{p+3}2\alpha^{p-2}\beta^{-2}\mathbf{q}_{tt}+\frac{p+2}{2(p+3)}\alpha^{1-p} \beta^2 (a_{32})^2\right. \\
&&\nonumber\left.-\frac{p+1}{p+3}\alpha^{1-p}\beta^2 a_{32}b_{21}-\frac{2}{p+3}\alpha^{1-p}\beta^2(b_{21})^2 +a_{11}\left(\beta^{-1}\beta''+2\alpha^{-1}\beta^{-1}\alpha'\beta'-\beta^{-2}(\beta')^2\right)\right],
\end{eqnarray}
\[\Upsilon_{1j}(\mathbf{e})=\Upsilon_1-2\varepsilon \alpha^{1-p}a_{11}e_j'\frac{\int_\R Z_xw_x dx}{\int_\R w_x^2 dx}\]
and
\[\Upsilon_{0j}(\mathbf{e})=\Upsilon_0-\alpha^{1-p}\left(2\varepsilon a_{11}\beta^{-1}\beta'e_j'+2\varepsilon a_{22}e_j'\right)\frac{\int_\R Z_xw_x dx}{\int_\R w_x^2 dx}.\]
Then \eqref{93} is written into the following form
\begin{equation*}
\varepsilon^2\left(-\Upsilon_2\breve{f}_j''+\Upsilon_{1j}(\mathbf{e})\breve{f}_j'-\Upsilon_{0j}(\mathbf{e})\breve{f}_j \right)+p\alpha_p C_0\left[e^{-\sqrt{p}(\breve{f}_j-\breve{f}_{j-1})}-e^{-\sqrt{p}(\breve{f}_{j+1}-\breve{f}_j)}\right]=\varepsilon^2 M_{1j\varepsilon}.
\end{equation*}

Let
\begin{equation}\label{e82}
\varepsilon^2\rho_\varepsilon=p\alpha_p C_0e^{-\sqrt{p}\rho_\varepsilon}.
\end{equation}
Then we deduce
\begin{equation}\label{e110}
\rho_\varepsilon=\frac2{\sqrt{p}}|\log\varepsilon|-\frac1{\sqrt{p}}\log\left[\frac2{\sqrt{p}}|\log\varepsilon|\right]+\frac1{\sqrt{p}}\log(p\alpha_p C_0)+O\left(\frac{\log|\log\varepsilon|}{|\log\varepsilon|}\right).
\end{equation}
Denote $\sigma=\rho_\varepsilon^{-1}$ and
\[\breve{f}_j(\theta)=\left(j-\frac{N+1}2\right)\rho_\varepsilon +d_j(\theta), \qquad j=1,\cdots,N, \]
Then $d_j$'s satisfy the following equations
\begin{equation}\label{e37}
\tilde{R}_j(\mathbf{d})=\sigma\left(-\Upsilon_2 d_j''+\Upsilon_{1j}(\mathbf{e})d_j'-\Upsilon_{0j}(\mathbf{e})\breve{f}_j\right)+\left[e^{-\sqrt{p}(d_j-d_{j-1})}-e^{-\sqrt{p}(d_{j+1}-d_j)}\right] =\sigma M_{1j\varepsilon},
\end{equation}
where $j=1,2,\cdots, N$ and $\mathbf{d}=(d_1, d_2, \cdots, d_N)^t$.

Using the notation \eqref{e81} and \eqref{e82}, we rewrite \eqref{e80} as the following form
\begin{equation}\label{e83}
\varepsilon^2 a_{11}\alpha^{1-p}e_j'' +\lambda_0 e_j +\varepsilon C_1 C_0^{-1}\rho_\varepsilon\left[e^{-\sqrt{p}(d_j-d_{j-1})}+e^{-\sqrt{p}(d_{j+1}-d_j)}\right]=\varepsilon M_{2j\varepsilon}.
\end{equation}
To solve \eqref{73}-\eqref{e80}, it suffices to find a solution to the problem given by \eqref{e37}-\eqref{e83}.

Let
\begin{equation*}
\mathbf{e}=\begin{bmatrix} e_1\\e_2\\ \vdots \\e_N\end{bmatrix},\qquad \tilde{\mathbf{R}}(\mathbf{d})=\begin{bmatrix}\tilde{R}_1(\mathbf{d}) \\ \tilde{R}_2(\mathbf{d})\\ \vdots \\\tilde{R}_N(\mathbf{d})\end{bmatrix}, \qquad \mathbf{M}_{i\varepsilon}=\begin{bmatrix} M_{i1\varepsilon} \\ M_{i2\varepsilon} \\ \vdots \\ M_{iN\varepsilon}\end{bmatrix}, \quad \mathrm{for} \quad i=1,2,
\end{equation*}
and
\begin{equation*}
\mathbf{B}(\mathbf{d})=\varepsilon C_1 C_0^{-1}\rho_\varepsilon\begin{bmatrix} e^{-\sqrt{p}(d_2-d_1)} \\e^{-\sqrt{p}(d_2-d_1)}+e^{-\sqrt{p}(d_3-d_2)}\\ \vdots \\ e^{-\sqrt{p}(d_{N-1}-d_{N-2})}+ e^{-\sqrt{p}(d_N-d_{N-1})}\\  e^{-\sqrt{p}(d_N-d_{N-1})} \end{bmatrix}.
\end{equation*}
Then \eqref{e37}-\eqref{e83} is written into the following problem
\begin{equation}\label{e86}
\left\{\begin{array}{ll}
\tilde{\mathbf{R}}(\mathbf{d})=\sigma \mathbf{M}_{1\varepsilon}, \\
\varepsilon^2a_{11}\alpha^{1-p}\mathbf{e}''+\lambda_0 \mathbf{e} +\mathbf{B}(\mathbf{d}) =\varepsilon \mathbf{M}_{2\varepsilon}.
\end{array}\right.
\end{equation}

It is apparent that $\tilde{R}_j(\mathbf{d})=R_j(\mathbf{d})-P_j(e_j,d_j)$, where
\[R_j(\mathbf{d}):=\sigma\left(-\Upsilon_2 d_j''+\Upsilon_1 d_j' -\Upsilon_0 \breve{f}_j\right)+\left[e^{-\sqrt{p}(d_j-d_{j-1})} -e^{-\sqrt{p}(d_{j+1}-d_j)}\right]\]
and
\[P_j(e_j,d_j):=2\varepsilon \alpha^{1-p}a_{11}e_j'd_j'\frac{\int_\R Z_xw_x dx}{\int_\R w_x^2dx}-2\varepsilon\alpha^{1-p}\left[a_{11}\beta^{-1}\beta'+a_{22}\right]e_j' \breve{f}_j\frac{\int_\R Z_xw_x dx}{\int_\R w_x^2dx}.\]
Denote $\mathbf{R}(\mathbf{d})=(R_1(\mathbf{d}), \cdots, R_N(\mathbf{d}))^t$.  Let
\[v_j=d_{j+1}-d_j, \qquad j=1, \cdots, N-1, \quad \mathrm{and} \quad v_N=\sum_{j=1}^N d_j.\]
Then
\begin{equation}\label{e38}
d_j=\frac1N v_N-\sum_{k=j}^{N-1} v_k +\frac1N \sum_{k=1}^{N-1} kv_k, \qquad \mathrm{for} \qquad j=1,2,\cdots,N.
\end{equation}
Direct computation yields
\[Q_N(v_N):=\sum_{j=1}^N R_j(\mathbf{d})=\sigma\left[-\Upsilon_2 v_N''+\Upsilon_1 v_N'-\Upsilon_0 v_N\right]\]
and
\begin{eqnarray*}
Q_j(\bar{\mathbf{v}})&=&R_{j+1}(\mathbf{d})-R_j(\mathbf{d}) \\
&=& \sigma\left(-\Upsilon_2 v_{j}''+\Upsilon_1 v_j'-\Upsilon_0(\rho_\varepsilon + v_j)\right)\\
&& +\left\{\begin{array}{ll}
-e^{-\sqrt{p}v_2}+2e^{-\sqrt{p}v_1}, & j=1, \\
-e^{-\sqrt{p}v_{j+1}}+2e^{-\sqrt{p}v_j}-e^{-\sqrt{p}v_{j-1}}, & j=2, \cdots, N-1, \\
2e^{-\sqrt{p}v_{N-1}}-e^{-\sqrt{p}v_{N-2}}, &j=N-1.
\end{array}\right.
\end{eqnarray*}
Denote
\[\mathbf{Q}(\mathbf{v})=\begin{bmatrix}
\bar{\mathbf{Q}}(\bar{\mathbf{v}}) \\
Q_N(v_N)
\end{bmatrix}, \qquad \bar{Q}(\bar{\mathbf{v}})=\begin{bmatrix}
Q_1(\bar{\mathbf{v}})\\
\vdots\\
Q_{N-1}(\bar{\mathbf{v}})
\end{bmatrix}.\]

Then
\[\mathbf{Q}(\mathbf{v})=B\mathbf{R}(B^{-1}(\mathbf{v})),\]\
where
\[B=\begin{bmatrix}
-1 & 1 & & & \\
0 & -1 & 1 & & \\
 & & \ddots &\ddots \\
 & & & -1 & 1 \\
 1& 1 & \cdots & 1& 1
\end{bmatrix}.\]
Then we deduce that $\mathbf{Q}(\mathbf{v})=0$ is equivalent to $\mathbf{R}(\mathbf{d})=0$.
\subsection{Approximate solution}
To solve the first equation in \eqref{e86}, we first construct an approximate solution to the following problem in this subsection:
\begin{equation}\label{e106}
R_j(\mathbf{d}):=\sigma\left(-\Upsilon_2 d_j''+\Upsilon_1 d_j' -\Upsilon_0 \breve{f}_j\right)+\left[e^{-\sqrt{p}(d_j-d_{j-1})} -e^{-\sqrt{p}(d_{j+1}-d_j)}\right] =0,
\end{equation}
where $j=1,2,\cdots, N$. It suffices to instead construct an approximate solution to $\mathbf{Q}(\mathbf{v})=0$.

From Remark \ref{rm3}, we get $Q_N(v_N)=0$ has the only trivial solution $v_N=0$. However, we have
\begin{equation}\label{94}
\bar{\mathbf{Q}}(\bar{\mathbf{v}})=\sigma\left[-\Upsilon_2 \bar{\mathbf{v}}''+\Upsilon_1 \bar{\mathbf{v}}'-\Upsilon_0 \bar{\mathbf{v}}\right]-\Upsilon_0\begin{bmatrix}1 \\ 1\\ \vdots \\1\end{bmatrix} +\bar{\mathbf{Q}}_0(\bar{\mathbf{v}})=0,
\end{equation}
where
\[\bar{\mathbf{Q}}_0(\bar{\mathbf{v}})=M\begin{bmatrix}e^{-\sqrt{p}v_1} \\ \vdots \\ e^{-\sqrt{p}v_{N-1}} \end{bmatrix},\qquad \mathrm{and} \qquad M=\begin{bmatrix}
2 &-1 & & & \\
-1 & 2 & -1 & & \\
& \ddots & \ddots & \ddots & \\
& & -1 &2 &-1 \\
& & & -1 &2
\end{bmatrix}.\]
Now we construct an approximate solution of \eqref{94}.

\begin{proposition}\label{pre3}
For any integer $k\ge1$, there exists a function $\bar{\mathbf{v}}_k(y,\sigma)=\bar{\mathbf{v}}_1+\sigma \eta_k(y,\sigma)$ such that $\bar{\mathbf{Q}}(\bar{\mathbf{v}}_k)=O(\sigma^k)$, where $\eta_1\equiv 0$, $\bar{\mathbf{v}}_1=-\frac1{\sqrt{p}}\log\left[\frac{\Upsilon_0}2(N-i)i\right]$ and $\eta_k$ is continuous on $\Gamma \times [0, +\infty)$.

Let
\begin{equation}\label{e43}
\mathbf{h}_k=B^{-1}\begin{bmatrix}\bar{\mathbf{v}}_k \\0\end{bmatrix}.
\end{equation}
There holds that $\mathbf{R}(\mathbf{h}_k)=O(\sigma^k)$.
\end{proposition}
\begin{proof}

Let $\bar{\mathbf{v}}_1=(\bar{\mathbf{v}}_{11}, \bar{\mathbf{v}}_{12}, \cdots, \bar{\mathbf{v}}_{1(N-1)})^t$ be the unique solution of the following linear problem
\begin{equation*}
\bar{\mathbf{Q}}_0(\bar{\mathbf{v}}_1)=M\begin{bmatrix}
e^{-\sqrt{p}\bar{\mathbf{v}}_{11}}\\ e^{-\sqrt{p}\bar{\mathbf{v}}_{12}}\\  \vdots\\ e^{-\sqrt{p}\bar{\mathbf{v}}_{1(N-1)}}\end{bmatrix}=\Upsilon_0 \begin{bmatrix}1 \\ 1\\ \vdots \\1\end{bmatrix}.
\end{equation*}
In fact, we have
\begin{equation}\label{e107}
\bar{\mathbf{v}}_{1i}=-\frac1{\sqrt{p}}\log\left(\frac{\Upsilon_0}2(N-i)i\right), \qquad i=1, \cdots, N-1.
\end{equation}
It follows that
\[\bar{\mathbf{Q}}(\bar{\mathbf{v}}_1)=\sigma[-\Upsilon_2\bar{\mathbf{v}}_1''+\Upsilon_1 \bar{\mathbf{v}}_1'-\Upsilon_0 \bar{\mathbf{v}}_1]=O(\sigma).\]
Thus, we have
\begin{equation}\label{97}
\bar{\mathbf{Q}}(\bar{\mathbf{v}}_1+\mathbf{w})=\sigma[-\Upsilon_2 \mathbf{w}''+\Upsilon_1\mathbf{w}'-\Upsilon_0 \mathbf{w}]+ \sigma [-\Upsilon_2\bar{\mathbf{v}}_1''+\Upsilon_1 \bar{\mathbf{v}}_1'-\Upsilon_0 \bar{\mathbf{v}}_1] +D\bar{\mathbf{Q}}_0(\bar{\mathbf{v}}_1)\mathbf{w} +N_1(\mathbf{w}),
\end{equation}
where
\begin{eqnarray*}
D\bar{\mathbf{Q}}_0(\bar{\mathbf{v}}_1)&=&-\sqrt{p}M\begin{bmatrix}e^{-\sqrt{p}\bar{\mathbf{v}}_{11}}, & & & \\
& e^{-\sqrt{p}\bar{\mathbf{v}}_{12}} && \\
& & \ddots & \\
& & & e^{-\sqrt{p}\bar{\mathbf{v}}_{1(N-1)}}\end{bmatrix} \\
 &=&-\frac{\sqrt{p}}2\Upsilon_0\begin{bmatrix}
2r_1 & -r_2 & & & & \\
-r_1 & 2r_2 & -r_3 & & & \\
& -r_2 & \ddots & \ddots & &  \\
&&\ddots&&& \\
 & &  & -r_{N-3} & 2r_{N-2} & -r_{N-1} \\
 & & & & -r_{N-2} & 2r_{N-1}
\end{bmatrix},
\end{eqnarray*}
$r_i=(N-i)i$ and
\[N_1(\mathbf{w})=\frac{\Upsilon_0}2M\begin{bmatrix}
r_1\left(e^{-\sqrt{p}w_1}-1+\sqrt{p}w_1\right) \\
r_2\left(e^{-\sqrt{p}w_2}-1+\sqrt{p}w_2\right) \\
\vdots \\
r_{N-1}\left(e^{-\sqrt{p}w_{N-1}}-1+\sqrt{p}w_{N-1}\right)
\end{bmatrix}\qquad \mathrm{and} \qquad  \mathbf{w}=
\begin{bmatrix}
w_1 \\
w_2 \\
\vdots \\
w_{N-1}
\end{bmatrix}.\]

Let $\mathbf{w}_1=O(\sigma)$ be unique solution of the following problem
\[-D\bar{\mathbf{Q}}_0(\bar{\mathbf{v}}_1)\mathbf{w}_1=\sigma[-\Upsilon_2 \bar{\mathbf{v}}_1''+ \Upsilon_1 \bar{\mathbf{v}}_1' -\Upsilon_0 \bar{\mathbf{v}}_1 ].\]
Then  $\bar{\mathbf{v}}_2:=\bar{\mathbf{v}}_1+\mathbf{w}_1$ satisfies
\[\bar{\mathbf{Q}}(\bar{\mathbf{v}}_2)=\sigma[-\Upsilon_2 \mathbf{w}_1''+\Upsilon_1\mathbf{w}_1'-\Upsilon_0 \mathbf{w}_1] +N_1(\mathbf{w}_1)=O(\sigma^2)\]
and
\begin{eqnarray*}
\bar{\mathbf{Q}}(\bar{\mathbf{v}}_2+\mathbf{w})&=&\bar{\mathbf{Q}}(\bar{\mathbf{v}}_1+\mathbf{w}_1+\mathbf{w}) \\
&=& \sigma[-\Upsilon_2 \mathbf{w}_1''+\Upsilon_1 \mathbf{w}_1' -\Upsilon_0 \mathbf{w}_1]+N_1(\mathbf{w}_1) +D\bar{\mathbf{Q}}_0(\bar{\mathbf{v}}_1)\mathbf{w} \\
&&+\sigma[-\Upsilon_2 \mathbf{w}''+\Upsilon_1 \mathbf{w}' -\Upsilon_0 \mathbf{w}]+N_1(\mathbf{w}_1+\mathbf{w})-N_1(\mathbf{w}_1).
\end{eqnarray*}
Denote $\mathbf{w}_2=O(\sigma^2)$ as the unique solution to the following equation
\[-D\bar{\mathbf{Q}}_0(\bar{\mathbf{v}}_1)\mathbf{w}_2=\sigma[-\Upsilon_2 \mathbf{w}_1''+\Upsilon_1 \mathbf{w}_1' -\Upsilon_0 \mathbf{w}_1]+N_1(\mathbf{w}_1).\]
Then the function $\bar{\mathbf{v}}_3=\bar{\mathbf{v}}_2+\mathbf{w}_2$ satisfies
\[\bar{\mathbf{Q}}(\bar{\mathbf{v}}_3)=\sigma[-\Upsilon_2 \mathbf{w}_2''+\Upsilon_1 \mathbf{w}_2' -\Upsilon_0 \mathbf{w}_2]+N_1(\mathbf{w}_1+\mathbf{w}_2)-N(\mathbf{w}_1).\]
Assume for $k\ge3$, the function $\bar{\mathbf{v}}_{k-1}=\bar{\mathbf{v}}_1+\sum_{j=1}^{k-1}\mathbf{w}_j$ satisfies
\[\bar{\mathbf{Q}}(\bar{\mathbf{v}}_{k-1})=\sigma[-\Upsilon_2 \mathbf{w}_{k-1}''+\Upsilon_1 \mathbf{w}_{k-1}' -\Upsilon_0 \mathbf{w}_{k-1}]+N_1(\mathbf{w}_1+\cdots +\mathbf{w}_{k-1})-N_1(\mathbf{w}_1+\cdots +\mathbf{w}_{k-2})\]
and $\mathbf{w}_j=O(\sigma^j)$, $j=1, \cdots, k-1$.

Let $\mathbf{w}_k=O(\sigma^k)$ be the unique solution of the following problem
\[-D\bar{\mathbf{Q}}_0(\bar{\mathbf{v}}_1)\mathbf{w}_k=\sigma [-\Upsilon_2 \mathbf{w}_{k-1}''+\Upsilon_1 \mathbf{w}_{k-1}' -\Upsilon_0 \mathbf{w}_{k-1}]+N_1(\mathbf{w}_1+\cdots+\mathbf{w}_{k-1})-N_1(\mathbf{w}_1+\cdots+\mathbf{w}_{k-2})\]
and $\bar{\mathbf{v}}_k=\bar{\mathbf{v}}_{k-1}+\mathbf{w}_k$.
From \eqref{97}, we get
\begin{eqnarray*}
\bar{\mathbf{Q}}(\bar{\mathbf{v}}_k)&=&\bar{\mathbf{Q}}(\bar{\mathbf{v}}_{k-1}+\mathbf{w}_k) \\
&=& \bar{\mathbf{Q}}(\bar{\mathbf{v}}_{k-1})+\sigma[-\Upsilon_2 \bar{\mathbf{w}}_k''+\Upsilon_1 \bar{\mathbf{w}}_k' -\Upsilon_0 \bar{\mathbf{w}}_k] +D\bar{\mathbf{Q}}_0(\bar{\mathbf{v}}_1)\mathbf{w}_k\\
&& +N_1(\mathbf{w}_1+\cdots+\mathbf{w}_k)- N_1(\mathbf{w}_1+\cdots+\mathbf{w}_{k-1})\\
&=& \sigma[-\Upsilon_2 \bar{\mathbf{w}}_k''+\Upsilon_1 \bar{\mathbf{w}}_k' -\Upsilon_0 \bar{\mathbf{w}}_k]+ N_1(\mathbf{w}_1+\cdots+\mathbf{w}_k)- N_1(\mathbf{w}_1+\cdots+\mathbf{w}_{k-1}).
\end{eqnarray*}
Hence $\bar{\mathbf{Q}}(\bar{\mathbf{v}}_k)=O(\sigma^{k+1})$. This proposition follows.

\end{proof}

From \eqref{e38}, we get
\begin{eqnarray}\label{e40}
\nonumber S_j(\mathbf{v})&:=&P_{j+1}(e_{j+1}, d_{j+1})-P_j(e_j, d_j) \\
\nonumber&=&2\varepsilon\alpha^{1-p}a_{11}\frac{\int_\R Z_xw_x dx}{\int_\R w_x^2dx}\left[e_{j+1}'\left(\frac1N v_N'-\sum_{k=j+1}^{N-1}v_k'+\frac1N\sum_{k=1}^{N-1} kv_k'\right)\right. \\
\nonumber&&\left.-e_j'\left(\frac1N v_N'-\sum_{k=j}^{N-1}v_k'+\frac1N\sum_{k=1}^{N-1} kv_k'\right)\right] -2\varepsilon\alpha^{1-p}[a_{11}\beta^{-1}\beta'+a_{22}]\frac{\int_\R Z_xw_x dx}{\int_\R w_x^2dx}\times\\
\nonumber&&\left\{e_{j+1}'\left[(j-\frac{N-1}2)\rho_\varepsilon +\frac1N v_N-\sum_{k=j+1}^{N-1}v_k+\frac1N\sum_{k=1}^{N-1} kv_k\right]\right. \\
&&\left.-e_j'\left[(j-\frac{N+1}2)\rho_\varepsilon +\frac1N v_N-\sum_{k=j}^{N-1}v_k+\frac1N\sum_{k=1}^{N-1} kv_k\right]\right\},
\end{eqnarray}
and
\begin{eqnarray*}
S_N(\mathbf{v})&:=&\sum_{j=1}^N P_j(e_j,d_j) \\
&=& 2\varepsilon \alpha^{1-p}a_{11}\frac{\int_\R Z_xw_x dx}{\int_\R w_x^2dx}\sum_{j=1}^N e_j'\left(\frac1N v_N'-\sum_{k=j}^{N-1}v_k'+\frac1N\sum_{k=1}^{N-1} kv_k'\right) \\
&& -2\varepsilon \alpha^{1-p}\left[a_{11}\beta^{-1}\beta'+a_{22}\right]\sum_{j=1}^N e_j'\left[(j-\frac{N+1}2)\rho_\varepsilon +\frac1N v_N-\sum_{k=j}^{N-1}v_k+\frac1N\sum_{k=1}^{N-1} kv_k\right].
\end{eqnarray*}
Denote
\begin{equation}\label{e79}
\mathbf{S}(\mathbf{v})=\begin{bmatrix}\bar{\mathbf{S}}(\mathbf{v})\\ S_N(\mathbf{v})\end{bmatrix}, \qquad \bar{\mathbf{S}}(\mathbf{v})=\begin{bmatrix}S_1(\mathbf{v})\\ S_2(\mathbf{v})\\  \vdots \\ S_{N-1}(\mathbf{v}))\end{bmatrix}.
\end{equation}
\subsection{Related problems}
Now we consider the problem
\begin{equation}\label{e42}
\tilde{\mathbf{R}}(\mathbf{d})=\mathbf{g}.
\end{equation}
From the calculation above,  we see that \eqref{e42} is equivalent to the following problem
\[\mathbf{Q}(\mathbf{u})=\bar{\mathbf{g}}+\mathbf{S}(\mathbf{u}),\]
where $\bar{\mathbf{g}}=B\mathbf{g}$.
The equation above is  equivalent to the following system
\begin{equation}\label{98}
\sigma[-\Upsilon_2 \mathbf{u}_N''+\Upsilon_1 \mathbf{u}_N' -\Upsilon_0 \mathbf{u}_N]=g_N+S_N(\mathbf{u})
\end{equation}
and
\begin{equation}\label{102}
\bar{\mathbf{Q}}(\bar{\mathbf{u}})=\sigma[-\Upsilon_2\bar{\mathbf{u}}''+\Upsilon_1 \bar{\mathbf{u}}' -\Upsilon_0 \bar{\mathbf{u}}]-\Upsilon_0\begin{bmatrix}1\\\vdots\\1\end{bmatrix}+\bar{\mathbf{Q}}_0(\bar{\mathbf{u}})=\bar{\mathbf{g}}+\bar{\mathbf{S}}(\mathbf{u}),
\end{equation}
where
\[\mathbf{u}=\begin{bmatrix}\bar{\mathbf{u}}\\ u_N\end{bmatrix}, \qquad  \mathrm{and} \qquad \mathbf{g}=\begin{bmatrix}\bar{\mathbf{g}}\\ g_N\end{bmatrix}.\]

To solve problem \eqref{98}-\eqref{102}, we first consider \eqref{102} for given $u_N$. To get a solution of the form $\bar{\mathbf{u}}=\bar{\mathbf{v}}_k+\mathbf{w}$ of \eqref{102}, we consider the following equation
\begin{equation}\label{108}
\mathbf{J}_\sigma(\mathbf{w}):=\sigma[-\Upsilon_2 \mathbf{w}''+\Upsilon_1 \mathbf{w}' -\Upsilon_0 \mathbf{w}]+D\bar{\mathbf{Q}}_0(\bar{\mathbf{v}}_k)\mathbf{w}=\bar{\mathbf{g}}+\bar{\mathbf{S}}(\bar{\mathbf{v}}_k+\mathbf{w})-N_2(\mathbf{w})-\bar{\mathbf{Q}}(\bar{\mathbf{v}}_k),
\end{equation}
where
\[N_2(\mathbf{w})=\bar{\mathbf{Q}}_0(\bar{\mathbf{v}}_k+\mathbf{w})-\bar{\mathbf{Q}}_0(\bar{\mathbf{v}}_k)-D\bar{\mathbf{Q}}_0(\bar{\mathbf{v}}_k) \mathbf{w}.\]

For the solvability of \eqref{108}, we first consider its associated linear problem:
\begin{equation}\label{103}
\mathbf{J}_\sigma(\mathbf{w})=\tilde{\mathbf{g}}, \qquad \mathrm{in} \qquad (0,1).
\end{equation}
\begin{lemma}\label{lme5}
For  sufficiently small $\varepsilon>0$ satisfying the condition
\begin{equation}\label{e41}
\left|\frac2{\sqrt{p}}\Lambda_i|\log\varepsilon|-\frac{4\pi^2j^2}{l_1^2}\right|>c_1\left(\frac2{\sqrt{p}|log\varepsilon|}\right)^{1/2}, \qquad \mathrm{for} \qquad  i=1,\cdots,N,
\end{equation}
\eqref{103}  has a unique solution $\mathbf{w}=F(\tilde{\mathbf{g}})$ provided $\tilde{\mathbf{g}}\in L^2(0,1)$. Moreover, $\mathbf{w}$ satisfies the following estimate
\[\sigma\|\mathbf{w}''\|_{L^2(0,1)}+\sigma^{\frac12}\|\mathbf{w}'\|_{L^2(0,1)}+\|\mathbf{w}\|_{L^2(0,1)}\leq C\sigma^{-\frac12}\|\tilde{\mathbf{g}}\|_{L^2(0,1)}.\]
\end{lemma}
\begin{proof}
Let $\phi=M^{-\frac12}\mathbf{w}$ and $\mathbf{g}_0=M^{-\frac12}\tilde{\mathbf{g}}$. Then \eqref{103} is written into the following problem
\[\sigma[-\Upsilon_2 \phi''+\Upsilon_1 \phi' -\Upsilon_0 \phi]-\sqrt{p}M^{\frac12}
\begin{bmatrix}
e^{-\sqrt{p}\bar{\mathbf{v}}_{k1}}&& \\
 & \ddots & \\
  && e^{-\sqrt{p}\bar{\mathbf{v}}_{k(N-1)}}
\end{bmatrix}
M^{\frac12}\phi=\mathbf{g}_0,\]
where $\bar{\mathbf{v}}_k=(\bar{\mathbf{v}}_{k1},\bar{\mathbf{v}}_{k2},\cdots, \bar{\mathbf{v}}_{k(N-1)})^t$.
We further rewrite this equation above into the following form
\begin{equation}\label{104}
\sigma[-\Upsilon_2 \phi''+\Upsilon_1 \phi']-\mathbf{C}(y, \sigma)\phi=\mathbf{g}_0,
\end{equation}
where
\[C(y, \sigma)=\sigma\Upsilon_0I_{N-1}+\sqrt{p}M^{\frac12}
\begin{bmatrix}
e^{-\sqrt{p}\bar{\mathbf{v}}_{k1}}&& \\
 & \ddots & \\
  && e^{-\sqrt{p}\bar{\mathbf{v}}_{k(N-1)}}
\end{bmatrix}
M^{\frac12}.\]

To solve \eqref{104}, we first consider the following problem
\begin{equation}\label{e34}
\sigma[-\Upsilon_2 \phi''+\Upsilon_1 \phi']-\mathbf{C}(y, 0)\phi=\mathbf{g}.
\end{equation}
From Proposition \ref{pre3} and \eqref{e107}, we get
\[C(y,0)=\sqrt{p}M^{\frac12}
\begin{bmatrix}
e^{-\sqrt{p}\bar{\mathbf{v}}_{11}}&& \\
 & \ddots & \\
  && e^{-\sqrt{p}\bar{\mathbf{v}}_{1(N-1)}}
\end{bmatrix}=\frac{\sqrt{p}}2\Upsilon_0(y)M^{\frac12}\begin{bmatrix}
r_1 && \\
& \ddots & \\
&& r_{N-1}
\end{bmatrix}
M^{\frac12}.\]
Let $\mathbf{g}=(g_1,g_2, \cdots, g_N)$,
\begin{equation*}
\mathbf{D}=\frac{\sqrt{p}}2M^{\frac12}\begin{bmatrix}
r_1 && \\
& \ddots & \\
&& r_{N-1}
\end{bmatrix}M^{\frac12},
\end{equation*}
and $\Lambda_1, \cdots, \Lambda_{N-1}$ be the eigenvalue of the matrix $\mathbf{D}$. Then \eqref{e34} is equivalent to the following system:
\begin{equation}\label{107}
\sigma[-\Upsilon_2 \phi''+\Upsilon_1\phi']-\Lambda_i \Upsilon_0\phi=g_i, \qquad i=1, \cdots, N-1.
\end{equation}
For simplicity, we write this problem into the generic form
\begin{equation}\label{105}
\sigma[-\Upsilon_2 \phi''+\Upsilon_1\phi']-\mu \Upsilon_0 \phi=g.
\end{equation}
It is well known that \eqref{105} has a unique solution provided that $\mu/\sigma$ differs from all the eigenvalues $\lambda=\tilde{\lambda}_j$ of the following problem
\begin{equation}\label{106}
\left\{\begin{array}{ll}
-\Upsilon_2 \phi''+\Upsilon_1 \phi'=\lambda \Upsilon_0 \phi, \qquad \mathrm{in} \quad(0,1), \\
\phi(0)=\phi(1), \quad \phi'(0)=\phi'(1).
\end{array}\right.
\end{equation}
Moreover, its solution has the following estimate
\[\|\phi\|_{L^2(0,1)}\leq \frac{C\sigma^{-1}}{\min_j\left|\mu/\sigma-\tilde{\lambda}_j\right|}\|g\|_{L^2(0,1)}.\]

Using the following Liouville transformation
\[l_1=\int_0^1 \sqrt{\frac{\Upsilon_0(t)}{\Upsilon_2(t)}}dt, \quad s=\frac{\pi}{l_1}\int_0^t \sqrt{\frac{\Upsilon_0(\theta)}{\Upsilon_2(\theta)}}d\theta, \quad \psi(s)=\Phi(t)\phi(t),\]
\[\Phi(t)=\sqrt[4]{\frac{\Upsilon_0(t)}{\Upsilon_2(t)}}\exp\left(-\frac12\int_0^t \frac{\Upsilon_1(t)}{\Upsilon_2(t)}dt\right),\]
\eqref{106} is written into the following form
\[\psi''+q(s)\psi+\frac{\lambda l_1^2}{\pi^2}\psi=0, \qquad \mathrm{in} \quad (0,\pi), \]
where $q(s)$ is a smooth function. From \cite{Levitan_Sargsjan1991}, we get the following estimate
\[\tilde{\lambda}_j=\frac{4\pi^2j^2}{l_1^2}+o(j^{-2})\qquad \mathrm{as} \qquad j\to\infty.\]
From \eqref{e110} and \eqref{e41}, we get
problem \eqref{107} has a unique solution
\[\phi=S(g)=(\phi_1,\phi_2, \cdots, \phi_{N-1})^t,\]
 which satisfies the following estimate
\begin{equation}\label{e35}
\|\phi\|_{L^2(0,1)}\leq C\sigma^{-\frac12}\|g\|_{L^2(0,1)}.
\end{equation}

Now we estimate the derivatives of $\phi$. Multiplying both sides of \eqref{105} by $\phi$ and integrating, we have
\[\sigma\int_0^1\left[-\Upsilon_2\phi''+\Upsilon_1 \phi'\right]\phi dt -\mu\int_0^1\Upsilon_0\phi^2 dt =\int_0^1 g\phi dt.\]
Thus,
\[\sigma\int_0^1 \Upsilon_2|\phi'|^2 dt +\sigma\int_0^1 (\Upsilon_1+\Upsilon_2')\phi \phi' dt=\mu\int_0^1\Upsilon_0\phi^2 dt+\int_0^1 g\phi dt.\]
Then we get
\[\sigma\|\phi'\|_{L^2(0,1)}^2\leq C\sigma\int_0^1 \Upsilon_2|\phi'|^2 dt\leq C\left[\|\phi\|_{L^2(0,1)}^2+\|g\|_{L^2(0,1)}\|\phi\|_{L^2(0,1)}+\sigma \|\phi\|_{L^2(0,1)}\|\phi'\|_{L^2(0,1)}\right].\]
From \eqref{e35}, we obtain
\[\sigma\|\phi'\|_{L^2(0,1)}^2\le C\sigma^{-1}\|g\|_{L^2(0,1)}^2.\]
That is
\begin{equation}\label{e36}
\sigma^{\frac12}\|\phi'\|_{L^2(0,1)}\leq C\sigma^{-\frac12}\|g\|_{L^2(0,1)}.
\end{equation}
From \eqref{105}, \eqref{e35} and \eqref{e36}, we get \eqref{105} has a unique solution, with $\phi$ satisfying the following estimate
\[\sigma\|\phi''\|_{L^2(0,1)}+\sigma^{\frac12}\|\phi'\|_{L^2(0,1)}+\|\phi\|_{L^2(0,1)}\leq C\sigma^{-\frac12}\|g\|_{L^2(0,1)}.\]

We write \eqref{104} into the following form
\[\sigma[-\Upsilon_2 \phi''+\Upsilon_1 \phi']-C(y,0)\phi=\mathbf{g}_0+\left[C(y, \sigma)-C(y,0)\right]\phi.\]
Using fixed point theorem, we get \eqref{104} has  a unique solution satisfying the  estimate
\[\sigma\|\phi''\|_{L^2(0,1)}+\sigma^{\frac12}\|\phi'\|_{L^2(0,1)}+\|\phi\|_{L^2(0,1)}\leq C\sigma^{-\frac12}\|\mathbf{g}_0\|_{L^2(0,1)}.\]
Hence this lemma follows.

\end{proof}

For the function $\phi \in H^2(0,1)$, we define the norm
\[\|\phi\|_b=\sigma\|\phi''\|_{L^2(0,1)}+\sigma^{\frac12}\|\phi'\|_{L^2(0,1)}+\|\phi\|_{L^2(0,1)}.\]
Recall $\mathbf{h}_k$ is defined in \eqref{e43}. Then we get the following lemma.
\begin{lemma}\label{lme6}
Let  $k>2$ and  $\varepsilon>0$ satisfies \eqref{e41}. For all the functions  $\mathbf{g}$ with $\|\mathbf{g}\|_{L^2(0,1)}\le \sigma^k$, \eqref{e42} has a unique solution of the form
\[\mathbf{d}=\mathbf{h}_k+H(\mathbf{g}),\]
where $H(\cdot)$ satisfies
\[\|H(\mathbf{g})\|_b\le C\sigma^{\frac{k+1}2}\]
and
\[\|H(\mathbf{g}_1)-H(\mathbf{g}_2)\|_b\le C\sigma^{-1}\|\mathbf{g}_1-\mathbf{g}_2\|_{L^2(0,1)}\]
\end{lemma}

\begin{proof}
Recall that \eqref{e42} is equivalent to \eqref{98} and \eqref{102}. We first consider the solution $\bar{\mathbf{u}}=\bar{\mathbf{v}}_k+\mathbf{w}$. With the help of  Lemma \ref{lme5}, \eqref{102} is written into the following fixed point problem
\begin{equation}\label{e39}
\mathbf{w}=F[\bar{\mathbf{g}}+\bar{S}(\bar{\mathbf{v}}_k+\mathbf{w})-N_2(\mathbf{w})-\bar{\mathbf{Q}}(\bar{\mathbf{v}}_k)]=\tilde{T}_1(\mathbf{w}).
\end{equation}
Let
\[\mathcal{D}=\left\{\mathbf{w}\in H^2(0,1)\;:\; \|\mathbf{w}\|_b\le \mu \sigma^{\frac{k+1}2}\right\}.\]
We solve the fixed point problem \eqref{102} in $\mathcal{D}$.

From Proposition \ref{pre3}, we get  $\|\bar{\mathbf{Q}}(\bar{\mathbf{v}}_k)\|_{L^2(0,1)}\leq C\sigma^k$. Using the definition of $\bar{\mathbf{S}}(\cdot)$ in \eqref{e40} and \eqref{e79}, we get
\begin{eqnarray*}
\|\bar{\mathbf{S}}(\bar{\mathbf{v}}_k+\mathbf{w})\|_{L^2(0,1)}&\le& C\varepsilon^{3/4}\sum_{j=1}^N\|e_j'\|\left(C+\|u_N\|_{H^2(0,1)}+\|\mathbf{w}\|_{H^2(0,1)}\right) \\
&\le & \varepsilon^{1/4}\left(C+\|u_N\|_{H^2(0,1)}+\|\mathbf{w}\|_b\right).
\end{eqnarray*}
From the definition of $N_2$, we get $|N_2(\mathbf{w})|\le C|\mathbf{w}|^2$. Hence
\[\|N_2(\mathbf{w})\|_{L^2(0,1)}\le C\|\mathbf{w}\|_{L^4(0,1)}^2\le C\sigma^{-1}\|\mathbf{w}\|_b^2.\]
For $\mathbf{w}\in \mathcal{D}$, we have
\begin{eqnarray*}
\|\tilde{T}(\mathbf{w})\|_b&\le& C\sigma^{-1/2}\left[\|\bar{\mathbf{g}}\|_{L^2(0,1)}+\varepsilon^{1/4}\left(C+\|u_N\|_{H^2(0,1)}+\sigma^{-1}\|\mathbf{w}\|_b\right)+\sigma^{-1}\| \mathbf{w}\|^2 +\sigma^k\right] \\
&\le& C\sigma^{-1/2}\left(\|\bar{\mathbf{g}}\|_{L^2(0,1)}+\sigma^k +\varepsilon^{1/4}\|u_N\|_{H^2(0,1)}\right).
\end{eqnarray*}
Then $\tilde{T}(\mathbf{w})\in \mathcal{D}$ for $\mu$ large enough.

However, for $\mathbf{w}_1, \mathbf{w}_2\in \mathcal{D}$, we get $|N_2(\mathbf{w}_1)-N_2(\mathbf{w}_2)|\le C\sigma^{\frac k2}|\mathbf{w}_1-\mathbf{w}_2|$. Hence
\[\|N_2(\mathbf{w}_1)-N_2(\mathbf{w}_2)\|_{L^2(0,1)}\le C\sigma^{\frac k2}\|\mathbf{w}_1-\mathbf{w}_2\|_{L^2(0,1)}.\]
From \eqref{e40}, we get
\[\|\bar{\mathbf{S}}(\bar{\mathbf{v}}_k+\mathbf{w}_1)-\bar{\mathbf{S}}(\bar{\mathbf{v}}_k+\mathbf{w}_2)\|_{L^2(0,1)}\le C\varepsilon^{1/2}\|\mathbf{w}_1-\mathbf{w}_2\|_{H^2(0,2)}\le C\varepsilon^{1/4}\|\mathbf{w}_1-\mathbf{w}_2\|_b.\]
Then
\begin{eqnarray*}
\|\tilde{T}_1(\mathbf{w}_1)-\tilde{T}_1(\mathbf{w}_2)\|_b&\le& C\sigma^{-1/2}\left(\sigma^{\frac k2}\|\mathbf{w}_1-\mathbf{w}_2\|_b+\varepsilon^{1/4}\|\mathbf{w}_1-\mathbf{w}_2\|_b\right) \\
&\le& C\sigma^{\frac {k-1}2}\|\mathbf{w}_1-\mathbf{w}_2\|_b.
\end{eqnarray*}

From fixed point theorem, we get \eqref{102} has a unique solution, which we denote by $\mathbf{w}=\tilde{\Omega}(\bar{\mathbf{g}}, u_N)$.

Now we estimate dependence of $\tilde{\Omega}(\cdot,\cdot)$ on its parameters. Let $\mathbf{w}_i=\tilde{\Omega}(\bar{\mathbf{g}}_i, u_{Ni})$, where $i=1,2$. Thus,
\[\mathbf{w}_1-\mathbf{w}_2=F[\bar{\mathbf{g}}_1-\bar{\mathbf{g}}_2+\bar{S}(\bar{\mathbf{v}}_k+\mathbf{w}_1,u_{N1})-\bar{S}(\bar{\mathbf{v}}_k+ \mathbf{w}_2,u_{N2})-(N_2(\mathbf{w}_1))-N_2( \bar{\mathbf{w}}_2)].\]
From \eqref{e40},
\[\|\bar{\mathbf{S}}(\bar{\mathbf{v}}_k+\mathbf{w}_1,u_{N1})-\bar{\mathbf{S}}(\bar{\mathbf{v}}_k+\mathbf{w}_2,u_{N2})\|_{L^2(0,1)}\le C\varepsilon^{1/2}\left(\|\mathbf{w}_1-\mathbf{w}_2\|_{H^2(0,1)}+\|u_{N1}-u_{N2}\|_{H^2(0,1)}\right).\]
Then we get
\begin{eqnarray*}
\|\mathbf{w}_1-\mathbf{w}_2\|_b &\le & C\sigma^{-1/2}\left[\|\bar{\mathbf{g}}_1-\bar{\mathbf{g}}_2\|_{L^2(0,1)} +\sigma^{\frac k2}\|\mathbf{w}_1-\mathbf{w}_2\|\right. \\
&&\left.+\varepsilon^{1/2}\left(\sigma^{-1}\|\mathbf{w}_1-\mathbf{w}_2\|_b+\|u_{N1}-u_{N2}\|_{H^2(0,1)}\right)\right].
\end{eqnarray*}
Hence
\begin{equation}\label{e44}
\|\tilde{\Omega}(\bar{\mathbf{g}}_1, u_{N1})-\tilde{\Omega}(\bar{\mathbf{g}}_2, u_{N2})\|_b\le C\left[\sigma^{-1/2}\|\bar{\mathbf{g}}_1-\bar{\mathbf{g}}_2\|_{L^2(0,1)}+\varepsilon^{1/4}\|u_{N1}-u_{N2}\|_{H^2(0,1)}\right].
\end{equation}

After solving the problem \eqref{102},  we consider \eqref{98}. From fixed point theorem, we get \eqref{98} has a unique solution $u_N=\tilde{T}_2(g_N)$ satisfying
\[\|u_N\|_{H^2(0,1)}\le  C\sigma^{-1}\left[\|g_N\|_{L^2(0,1)}+\varepsilon^{1/4}\right].\]
Let $u_{Ni}=\tilde{T}_2(g_{Ni})$, $i=1,2$. Using \eqref{e44}, we get
\begin{eqnarray*}
\|u_{N1}-u_{N2}\|_{H^2(0,1)}&\le& C\sigma^{-1}\left[\|g_{N1}-g_{N2}\|_{L^2(0,1)}+\varepsilon^{1/2}\left(\|u_{N1}-u_{N2}\|_{H^2(0,1)}+\sigma^{-1}\|\mathbf{w}_1-\mathbf{w}_2\|_b\right)\right]  \\
&\le& C\sigma^{-1}\left[\|g_{N1}-g_{N2}\|_{L^2(0,1)}+\varepsilon^{1/2}\|u_{N1}-u_{N2}\|_{H^2(0,1)}\right]+C\varepsilon^{1/4}\|\bar{\mathbf{g}}_1-\bar{\mathbf{g}}_2\|_{L^2(0,1)}.
\end{eqnarray*}
Hence
\[\|u_{N1}-u_{N2}\|_{H^2(0,1)}\le C\sigma^{-1}\|\mathbf{g}_1-\mathbf{g}_2\|_{L^2(0,1)}.\]
Therefore \eqref{e41} has a unique solution of the form
\[\mathbf{d}=B^{-1}\begin{bmatrix}\bar{\mathbf{v}}_k+\mathbf{w}\\ u_N\end{bmatrix}.\]
This lemma follows.
\end{proof}

To solve second equation in  \eqref{e86}, we first consider the following problem
\begin{equation}\label{88}
\varepsilon^2 a_{11}\alpha^{1-p}e''+\lambda_0 e=h, \quad \mathrm{in} \quad (0,1).
\end{equation}
From the same method in \cite[Lemma 8.1]{delPino_Kowalczyk_Wei2007}, we obtain the following lemma.

\begin{lemma}\label{lme7}
Assume
\begin{equation}\label{e108}
|\varepsilon^2 k^2 -\lambda_*|>c_2\varepsilon \qquad \mathrm{for} \qquad\forall k\in Z_+,
\end{equation}
where
\[\lambda_*=\lambda_0l_2^2/4\pi \quad \mathrm{and} \quad l_2=\int_0^1 \frac1{\sqrt{\Upsilon_2(t)}}dt.\]
If $h\in L^2(0,1)$, \eqref{88} has  unique solution $e=G(h)$, which satisfies
\[\varepsilon^2\|e''\|_{L^2(0,1)}+\varepsilon\|e'\|_{L^2(0,1)}+\|e\|_{L^\infty(0,1)}\leq C\varepsilon^{-1}\|h\|_{L^2(0,1)}.\]
However, if $d\in H^2(0,1)$, we have
\begin{equation}\label{e84}
\varepsilon^2\|e''\|_{L^2(0,1)}+\varepsilon\|e'\|_{L^2(0,1)}+\|e\|_{L^\infty(0,1)}\leq C \|h\|_{H^2(0,1)}.
\end{equation}
\end{lemma}

\subsection{Solving \eqref{e86}}
From now on, we impose the condition
\begin{equation}\label{e85}
\|\mathbf{d}\|_{H^2(0,1)}\le M_1,
\end{equation}
where $M_1>0$ is a constant large enough.
Let
\[\mathbf{e}_0(\mathbf{d})=G(\mathbf{B}(\mathbf{d})).\]
According to the estimate \eqref{e84} and the condition \eqref{e85}, we get $\|\mathbf{e}_0\|_*\le C\varepsilon |\log\varepsilon|$ and
\[\|\mathbf{e}_0(\mathbf{d}_1)-\mathbf{e}_0(\mathbf{d}_2)\|_*\le C\varepsilon |\log\varepsilon|\|\mathbf{d}_1-\mathbf{d}_2\|_{H^2(0,1)}.\]

Let $\mathbf{e}=\mathbf{e}_0+\tilde{\mathbf{e}}$. Then \eqref{e86} is transformed into the following equation
\begin{equation}\label{e87}
\left\{\begin{array}{ll}
\tilde{\mathbf{R}}(\mathbf{d})=\sigma \tilde{\mathbf{M}}_{1\varepsilon}(\mathbf{d},\tilde{\mathbf{e}}), \\
\varepsilon^2 a_{11}\alpha^{1-p}\tilde{\mathbf{e}}''+\lambda_0 \tilde{\mathbf{e}}=\varepsilon\tilde{\mathbf{M}}_{2\varepsilon}(\mathbf{d},\tilde{\mathbf{e}}).
\end{array}\right.
\end{equation}
where $\tilde{\mathbf{M}}_{i\varepsilon}(\mathbf{d},\tilde{\mathbf{e}})=\tilde{\mathbf{A}}_{i\varepsilon}(\mathbf{d},\tilde{\mathbf{e}})+\tilde{\mathbf{K}}_{i\varepsilon}( \mathbf{d},\tilde{\mathbf{e}})$, $i=1,2$. Here $\tilde{\mathbf{A}}_{i\varepsilon}$ is a Lipschitz operator satisfying
\[\|\tilde{\mathbf{A}}_{i\varepsilon}(\mathbf{d}_1,\mathbf{e}_1)-\tilde{\mathbf{A}}_{i\varepsilon}(\mathbf{d}_2,\mathbf{e}_2)\|_{L^2(0,1)}\leq C\varepsilon^{\frac12+\mu_1}\left[\|\mathbf{d}_1-\mathbf{d}_2\|_{H^2(0,1)}+\|\mathbf{e}_1-\mathbf{e}_2\|_*\right],\]
and $\tilde{\mathbf{K}}_{i\varepsilon}$ is a compact operator. There also hold that
\begin{equation*}
\|\tilde{\mathbf{A}}_{i\varepsilon}\|_{L^2(0,1)}\leq C\varepsilon^{\frac12+\mu_1}, \qquad \|\tilde{\mathbf{K}}_{i\varepsilon}\|_{L^2(0,1)}\leq C\varepsilon^{\frac12+\mu_1}.
\end{equation*}

Then we consider the following system
\begin{equation}\label{e45}
\left\{\begin{array}{ll}
\tilde{\mathbf{R}}(\mathbf{d})-\sigma \tilde{\mathbf{A}}_{1\varepsilon}(\mathbf{d},\tilde{\mathbf{e}}) =\tilde{\mathbf{h}}, \\
\varepsilon^2a_{11}\alpha^{1-p}\tilde{\mathbf{e}}''+\lambda_0 \tilde{\mathbf{e}}' -\varepsilon \tilde{\mathbf{A}}_{2\varepsilon}(\mathbf{d},\tilde{\mathbf{e}})=\tilde{\mathbf{g}}.
\end{array}\right.
\end{equation}
\begin{lemma}\label{lme8}
Assume the small constant $\varepsilon>0$ satisfy \eqref{e41} and \eqref{e108}.
Under the condition  $\|\tilde{\mathbf{h}}\|_{L^2(0,1)}\le C\sigma \varepsilon^{\frac12+\mu_2}$ and $\|\tilde{\mathbf{g}}\|_{L^2(0,1)}\le C\varepsilon^{\frac32+\mu_2}$, with $\mu_2\in(0,1/2)$, \eqref{e45} has a unique solution $(\mathbf{d},\tilde{\mathbf{e}})=(\mathbf{h}_k+\mathcal{R}_1(\tilde{\mathbf{h}},\tilde{\mathbf{g}}),\mathcal{R}_2(\tilde{\mathbf{h}},\tilde{\mathbf{g}}))$, which satisfies
\[\|\mathcal{R}_1(\tilde{\mathbf{h}},\tilde{\mathbf{g}})\|_b\le C\sigma^{\frac{k+1}2}, \qquad \mathrm{and} \qquad \|\mathcal{R}_2(\tilde{\mathbf{h}},\tilde{\mathbf{g}})\|_*\le C \varepsilon^{1/2+\mu_2}.\]
Moreover,
\begin{eqnarray}\label{e48}
\nonumber&& \|\mathcal{R}_1(\mathbf{h}_1,\mathbf{g}_1)-\mathcal{R}_1(\mathbf{h}_2,\mathbf{g}_2)\|_b+\|\mathcal{R}_2(\mathbf{h}_1,\mathbf{g}_1)-\mathcal{R}_1(\mathbf{h}_2,\mathbf{g}_2)\|_* \\
&\le& C\sigma^{-1}\|\mathbf{h}_1-\mathbf{h}_2\|_{L^2(0,1)}+C\varepsilon^{-1}\|\mathbf{g}_1-\mathbf{g}_2\|_{L^2(0,1)}.
\end{eqnarray}
\end{lemma}
\begin{proof}
Substituting $\mathbf{d}=\mathbf{h}_k+\mathbf{p}$ into \eqref{e45}, and with the help of Lemmata \ref{lme6} and \ref{lme7}, we only need to solve the following fixed point problem
\begin{equation}\label{e47}
\left\{\begin{array}{ll}
\mathbf{p}=H\left[\sigma \tilde{\mathbf{A}}_{1\varepsilon}(\mathbf{h}_k+\mathbf{p},\tilde{\mathbf{e}})+\tilde{\mathbf{h}}\right], \\
\tilde{\mathbf{e}}=G\left[\varepsilon \tilde{\mathbf{A}}_{2\varepsilon}(\mathbf{h}_k+\mathbf{p},\tilde{\mathbf{e}})+\tilde{\mathbf{g}}\right].
\end{array}\right.
\end{equation}
We define
\[\mathcal{D}_1=\left\{(\mathbf{p},\mathbf{e})\;:\; \|\mathbf{p}\|_b\leq \kappa\sigma^{\frac{k+1}2},\; \|\mathbf{e}\|_*\le \kappa_1\varepsilon^{1/2+\mu_2}\right\},\]
where $\kappa$ and $\kappa_1$ are positive constants large enough.
Let
\[\mathcal{H}(\tilde{\mathbf{h}},\mathbf{p},\tilde{\mathbf{e}})=H\left[\sigma \tilde{\mathbf{A}}_{1\varepsilon}(\mathbf{h}_k+\mathbf{p},\tilde{\mathbf{e}})+\tilde{\mathbf{h}}\right]\]
and
\[\mathcal{G}(\tilde{\mathbf{g}},\mathbf{p},\tilde{\mathbf{e}})=G\left[\varepsilon \tilde{\mathbf{A}}_{2\varepsilon}(\mathbf{h}_k+\mathbf{p},\tilde{\mathbf{e}})+\tilde{\mathbf{g}}\right].\]

For $(\mathbf{p},\mathbf{e})\in \mathcal{D}_1$, we get
\[\|\sigma \mathbf{A}_{1\varepsilon}(\mathbf{h}_k+\mathbf{p},\mathbf{e})+\tilde{\mathbf{h}}\|_{L^2(0,1)}\le C\sigma\varepsilon^{\frac12+\mu_2} +\|\tilde{\mathbf{h}}\|_{L^2(0,1)}\le \sigma^k.\]
From Lemma \ref{lme6}, we get $\|\mathcal{H}(\tilde{\mathbf{h}},\mathbf{p},\mathbf{e})\|_b \le C\sigma^{\frac{k+1}2}$. However
\[\|\mathcal{G}(\tilde{\mathbf{g}},\mathbf{p},\mathbf{e})\|_*\le C\varepsilon^{-1}\left(\varepsilon^{\frac32+\mu_2}+\|\mathbf{g}\|_{L^2(0,1)}\right) \le C\varepsilon^{\frac12+\mu_1}.\]
Hence $(\mathcal{H}(\tilde{\mathbf{h}},\mathbf{p},\mathbf{e}),\mathcal{G}(\tilde{\mathbf{g}},\mathbf{p},\mathbf{e}) )\in \mathcal{D}_1$ for sufficient large $\kappa>0$ and $\kappa_1>0$.

For $(\mathbf{p}_1,\mathbf{e}_1), (\mathbf{p}_2,\mathbf{e}_2)\in \mathcal{D}_1$, we get
\begin{eqnarray*}
\|\mathcal{H}(\tilde{\mathbf{h}},\mathbf{p}_1,\mathbf{e}_1)-\mathcal{H}(\tilde{\mathbf{h}},\mathbf{p}_2,\mathbf{e}_2)\|_b
&\le& C \varepsilon^{\frac12+\mu_1}\left(\|\mathbf{p}_1-\mathbf{p}_2\|_{H^2(0,1)}+\|\mathbf{e}_1-\mathbf{e}_2\|_*\right) \\
&\le& C\varepsilon^{\frac12} \left(\|\mathbf{p}_1-\mathbf{p}_2\|_b+\|\mathbf{e}_1-\mathbf{e}_2\|_*\right),
\end{eqnarray*}
and
\begin{eqnarray*}
\|\mathcal{G}(\tilde{\mathbf{g}},\mathbf{p}_1,\mathbf{e}_1)-\mathcal{G}(\tilde{\mathbf{g}},\mathbf{p}_2,\mathbf{e}_2)\|_* &\le& C\varepsilon^{\frac12+\mu_1}\left(\|\mathbf{p}_1-\mathbf{p}_2\|_{H^2(0,1)}+\|\mathbf{e}_1-\mathbf{e}_2\|_*\right) \\
&\le& C\varepsilon^{\frac12} \left(\|\mathbf{p}_1-\mathbf{p}_2\|_b+\|\mathbf{e}_1-\mathbf{e}_2\|_*\right).
\end{eqnarray*}
From fixed point theorem, we get \eqref{e47} has a unique solution in $\mathcal{D}_1$. \eqref{e48} follows from   Lemma \ref{lme6} and Lemma \ref{lme7}.
\end{proof}

%
%

\noindent\textit{Proof of Theorem \ref{th1}.}
With the help of Lemma \ref{lme8}, we only need to consider the following problem
\begin{equation}\label{e112}
\left\{\begin{array}{ll}
\tilde{\mathbf{p}}=\mathcal{R}_1(\sigma \tilde{\mathbf{K}}_{1\varepsilon}(\mathbf{h}_k+\tilde{\mathbf{p}},\tilde{\mathbf{e}}), \varepsilon\tilde{\mathbf{K}}_{2\varepsilon}(\mathbf{h}_k+\tilde{\mathbf{p}},\tilde{\mathbf{e}})), \\
\tilde{\mathbf{e}}=\mathcal{R}_2(\sigma \tilde{\mathbf{K}}_{1\varepsilon}(\mathbf{h}_k+\tilde{\mathbf{p}},\tilde{\mathbf{e}}), \varepsilon\tilde{\mathbf{K}}_{2\varepsilon}(\mathbf{h}_k+\tilde{\mathbf{p}},\tilde{\mathbf{e}})).
\end{array}\right.
\end{equation}
Denote
\[\mathcal{R}(\tilde{\mathbf{p}},\tilde{\mathbf{e}})=\left(\mathcal{R}_1(\sigma \tilde{\mathbf{K}}_{1\varepsilon}(\mathbf{h}_k+\tilde{\mathbf{p}},\tilde{\mathbf{e}}), \varepsilon\tilde{\mathbf{K}}_{2\varepsilon}(\mathbf{h}_k+\tilde{\mathbf{p}},\tilde{\mathbf{e}})),\mathcal{R}_2(\sigma \tilde{\mathbf{K}}_{1\varepsilon}(\mathbf{h}_k+\tilde{\mathbf{p}},\tilde{\mathbf{e}}), \varepsilon\tilde{\mathbf{K}}_{2\varepsilon}(\mathbf{h}_k+\tilde{\mathbf{p}},\tilde{\mathbf{e}}))\right).\]
From \eqref{e48} and the fact $\tilde{\mathbf{K}}_{1\varepsilon}$, and $\tilde{\mathbf{K}}_{2\varepsilon}$ are compact operators, we get $\mathcal{R}(\tilde{\mathbf{p}},\tilde{\mathbf{e}})$ is a compact operator.
Then we rewrite \eqref{e112} into the following fixed point problem
\begin{equation}\label{e111}
(\tilde{\mathbf{p}},\tilde{\mathbf{e}})=\mathcal{R}(\tilde{\mathbf{p}},\tilde{\mathbf{e}})
\end{equation}
in
\[\mathcal{D}_2=\left\{(\mathbf{p},\mathbf{e})\;:\; \|\mathbf{p}\|_b\leq \frac1{|\log\varepsilon|^{\frac32}},\; \|\mathbf{e}\|_*\le \varepsilon^{1/2}\right\}.\]
By virtue of the Schauder fixed point theorem, we get \eqref{e111} has a unique solution provided $\varepsilon>0$ satisfying \eqref{e41} and \eqref{e108}. Hence  Theorem \ref{th1} follows.
$\hfill\Box$

\appendix
\section{Decay estimate of the solution to \eqref{29}}\label{sectB}
In this section, we assume the constant $p>1$ and estimate the decay property of the function $w$.
\begin{lemma}\label{lm9}
The unique solution to \eqref{29} satisfies the following estimate
\[w(t)=\alpha_p e^{-\sqrt{p}|t|}+O(e^{-\min\{p,2\}\sqrt{p}|t|}), \qquad \mathrm{as}  \qquad |t|\to \infty,\]
where the constant $\alpha_p>0$ is of the following form:
\begin{equation}\label{e115}
\alpha_p=\frac{\sqrt{p}c_p}2\int_\R \left[|w-1|^{p-2}(w-1)+1\right]\left(e^{\sqrt{p}t}-e^{-\sqrt{p}t}\right)w'(t)dt.
\end{equation}
\end{lemma}
\begin{proof}
We use the method in \cite[Section 4]{Gidas_Ni_Nirenberg1981} to prove this lemma. Since $w(t)$ is an even function, we only need to consider the asymptotic behaviours of $w(t)$ as $t\to +\infty$.

Let $g(s)=|s-1|^p-1+ps$. Then we write \eqref{29} into the following form
\begin{equation}\label{84}
-w''+pw=g(w), \quad \mathrm{in} \quad \R, \qquad w\to 0 \quad \mathrm{as} \quad |t|\to \infty.
\end{equation}
It is obvious that $g(w)=O(|w|^{\min\{p,2\}})$. For any constant $\varepsilon>0$, there exists $r_0>0$ large enough, such that $|g(w(t))|\leq \varepsilon w(t)$ for $t>r_0$, which implies
\begin{equation}\label{83}
-w''+pw<\varepsilon w, \qquad \mathrm{for} \quad t>r_0.
\end{equation}
Hence $w''>(p-\varepsilon)w>0$. So $w'$ is an increasing function for $t>r_0$. Then we conclude  $w'(t)<0$ for $t>r_0$.

Multiplying the both side of \eqref{83} by $2w'$, we get
\[\left(\left(w'\right)^2-(p-\varepsilon)w^2\right)'<0, \quad \mathrm{for} \quad t>r_0.\]
So $y(t):=\left(w'\right)^2-(p-\varepsilon)w^2$ is a decreasing function. Then  $w'(t)\to 0$ as $t\to +\infty$; otherwise, $y(t)\to c^2>0$ as $t\to +\infty$ which implies that $w'(t)\to -c$ as $t\to +\infty$. This contradicts  the fact that $w(t)\to 0$ as $t\to \infty$. Hence $\left(w'\right)^2-(p-\varepsilon)w^2\geq 0$ for $t>r_0$. Then
\[w'+\sqrt{p-\varepsilon}w\leq 0\qquad \mathrm{for} \qquad t>r_0.\]
Hence
\begin{equation}\label{46}
w(t)=O(e^{-a|t|})\qquad \mathrm{where} \qquad 0<a<\sqrt{p}.
\end{equation}

By means of the Green function and \eqref{84}, we get
\begin{equation*}
w(t)=c_p\int_\R e^{-\sqrt{p}|t-s|}g(w(s))ds.
\end{equation*}
From \eqref{46}, we have
\[e^{\sqrt{p}|t|}w(t)\leq c_p\int_\R e^{\sqrt{p}|s|}g(w(s))ds\leq C.\]
Hence $w(t)=O(e^{-\sqrt{p}|t|})$. Then
\begin{equation}\label{87}
g(w(t))=O(e^{-\min\{p,2\}\sqrt{p}|t|}).
\end{equation}

Now we consider the solution of the following problem
\begin{equation}\label{85}
-u''+pu=f(t), \quad \mathrm{in} \quad \R.
\end{equation}

\noindent \textbf{Claim:} If the function $f(t)$ satisfies the decay estimate $f(t)=O(e^{-\min\{p,2\}\sqrt{p}|t|})$ as $|t|\to \infty$, we get
\[\lim_{t\to +\infty} e^{\sqrt{p}|t|}u(t)=c_p\int_\R e^{\sqrt{p}s}f(s)ds.\]

If $f$ is a smooth function with compact support, we get
\[\lim_{t\to +\infty}e^{\sqrt{p}|t|}u(t)=\lim_{t\to +\infty} c_p\int_\R e^{\sqrt{p}(|t|-|t-s|)}f(s)ds=c_p\int_\R e^{\sqrt{p}s}f(s)ds.\]
However in the case of $f(t)=O(e^{-\min\{p,2\}\sqrt{p}|t|})$, we define the Banach space $B_{\tilde{\gamma}}$( $\sqrt{p}<\tilde{\gamma}<\min\{p,2\}\sqrt{p}$) with the norm
\[\|u\|_{\tilde{\gamma}} =\sup\{e^{\tilde{\gamma}|t|}|u(t)|\}.\]
Then $f\in B_{\tilde{\gamma}}$. There exists a sequence of compactly supported smooth functions  $\{f_n\}\in C^\infty_0$ such that $\|f_n-f\|_{\tilde{\gamma}}\to 0$. Denote the solution of \eqref{85} corresponding to $f=f_n$ by $u_n$. Then we get
\begin{eqnarray*}
e^{\sqrt{p}|t|}|u(t)-u_n(t)|&\leq& C\int_\R e^{\sqrt{p}(|t|-|t-s|)}|f(s)-f_n(s)|ds  \\
&\leq& C\|f-f_n\|_{\tilde{\gamma}}\int_\R e^{(\sqrt{p}-\tilde{\gamma})|s|}ds\leq C\|f-f_n\|_{\tilde{\gamma}}.
\end{eqnarray*}
Let $t\to +\infty$. Then we get
\begin{equation}\label{86}
\left|\lim_{t\to +\infty} e^{\sqrt{p}|t|}u(t)-c_p\int_\R e^{\sqrt{p}s}f_n(s)ds\right|\leq C\|f-f_n\|_{\tilde{\gamma}}.
\end{equation}
Using the dominated convergence theorem and \eqref{86}, we get the claim.

From \eqref{87} and the claim, we get the solution $w$ of \eqref{84} satisfies the estimate
\[\lim_{|t|\to \infty} e^{\sqrt{p}|t|}w(t)=\alpha_p,\]
where  the positive constant $\alpha_p$ is of the following form
\begin{eqnarray*}
\alpha_p&=&c_p \int_\R e^{\sqrt{p}s}\left(|w-1|^p-1+pw\right)ds \\
&=& \frac{c_p}2\int_\R \left(e^{\sqrt{p}s}+e^{-\sqrt{p}s}\right)\left(|w-1|^p-1+pw\right)ds \\
&=& \frac{\sqrt{p}c_p}2\int_\R \left(|w-1|^{p-2}(w-1)+1\right)\left(e^{\sqrt{p}s}+e^{-\sqrt{p}s}\right)w'(s)ds.
\end{eqnarray*}

Let
\[v(t)=w(t)-\alpha_p e^{-\sqrt{p}|t|}.\]
Then $v$ satisfies
\[-v''+pv=g(w(t)), \quad \mathrm{for} \quad t\not=0,\]
and
\[\lim_{|t|\to \infty} e^{\sqrt{p}|t|}v(t)=0.\]
From L'H\^{o}pital's rule, we get
\[\lim_{|t|\to \infty} e^{\sqrt{p}|t|}v'(t)=0.\]

Denote $\tilde{w}_1(t)=e^{\sqrt{p}|t|}v(t)$. This function satisfies
\[-\tilde{w}_1''+2\sqrt{p}\left(\mathrm{sgn}\;(t)\right)\tilde{w}_1'=e^{\sqrt{p}|t|}g(w(t)), \quad \mathrm{for} \quad t\not=0, \qquad \mathrm{and} \qquad \tilde{w}_1(\pm\infty)=\tilde{w}_1'(\pm\infty)=0.\]
From theory in ordinary differential equation and \eqref{87}, we have
\[\tilde{w}_1(t)=\int_{+\infty}^te^{2\sqrt{p}x}dx\int_{+\infty}^x e^{-\sqrt{p}|s|}g(w(s))ds=O(e^{-\min\{p-1,1\}\sqrt{p}|t|}),\]
for $t>0$ large enough. Hence this lemma follows.
\end{proof}

\section{Property of the negative solutions of \eqref{1}}\label{sectA}
In this section, we investigate the negative solutions to \eqref{1}.
\begin{proposition}\label{pro1}
Assume that $\Omega$ is a smooth bounded domain in $\R^{\mathcal{N}}(\mathcal{N}\ge 2)$ and the constant $p\in (1,\frac{\mathcal{N}+2}{\mathcal{N}-2})$ for $\mathcal{N}\ge 3$ and $p>1$ for $\mathcal{N}=1,2$. Let $\mathbf{\Psi}(x)$ be an eigenfunction corresponding to the first Dirichlet eigenvalue of the operator $\mathfrak{L}(u)=-\Div (A(x)\nabla u)$ on $\Omega$, where  $A(x)=\{A_{ij}(x)\}_{2\times 2}$ is a symmetric positive definite matrix-valued function satisfying \eqref{e98}.
There exists $\varepsilon_0>0$, such that for any  $\varepsilon\in (0,\varepsilon_0)$,   \eqref{1} has a unique negative solution $\bar{u}_\varepsilon> -\mathbf{\Psi}^{\frac1p}$ and  the following estimate holds on any compact sets in $\Omega$:
\begin{equation}\label{e101}
\bar{u}_\varepsilon (x)=-\mathbf{\Psi}^{\frac1p}(x)-\varepsilon^2\left(\frac{\Div(A(x)\nabla \mathbf{\Psi}^{\frac1p})}{p\mathbf{\Psi}^{\frac{p-1}p}(x)}+o(1)\right), \quad \mathrm{as} \quad \varepsilon \to 0^+.
\end{equation}
\end{proposition}
We will prove this Proposition \ref{pro1} via a similar method in \cite[Theorem 1.1]{Li_Yang_Yan2004} and \cite[Theorem 2.1]{Dancer_Yan2005}.

Let $u=-w$. Then  \eqref{1} reduces to the following problem:
\begin{equation}\label{3}
\left\{\begin{array}{ll}
-\varepsilon^2\Div (A(x)\nabla w)=\mathbf{\Psi}(x)-|w|^p, &\mbox{in $\Omega$,} \\
w=0, & \mbox{on $\partial \Omega$}.
\end{array}
\right.
\end{equation}
To solve \eqref{3}, we consider the following auxiliary problem problem first
\begin{equation}\label{5}
\left\{\begin{array}{ll}
-\varepsilon^2\Div (A(x)\nabla \tilde{w})=h(x,\tilde{w}), &\mbox{in $\Omega$,} \\
\tilde{w}=0, & \mbox{on $\partial \Omega$},
\end{array}
\right.
\end{equation}
where
\begin{equation*}
h(x,t)=\left\{
\begin{array}{ll}
\mathbf{\Psi}(x), &\mbox{for $t<0$,} \\
\mathbf{\Psi}(x)-|t|^p, &\mbox{for $t \geq 0$.}
\end{array}\right.
\end{equation*}
The energy functional of \eqref{5} is
\[J_\varepsilon(\tilde{w})=\frac{\varepsilon^2}2\int_\Omega \langle A(x)\nabla \tilde{w}, \nabla \tilde{w} \rangle dx-\int_\Omega H(x, \tilde{w})dx, \qquad \tilde{w}\in H^1_0(\Omega),\]
where $H(x,t)=\int_0^t h(x, \tau)d\tau$.
It is apparent that $J_\varepsilon$ is bounded from below on $H^1_0(\Omega)$. Let $\underline{u}_\varepsilon$ be its minimizer.

From the definition of $h(x,t)$, we get $H(x, t)<0$ for $t<0$ or $t>M_3$, where $M_3$ is a large positive constant. Consequently, we have $0< \underline{u}_\varepsilon<M_3$ and $\underline{u}_\varepsilon$ solves \eqref{5}.

Through direct computation, $0$ is a subsolution of \eqref{5} whereas $\mathbf{\Psi}(x)^{\frac1p}$ is a supersolution of \eqref{5}. From the same argument of \cite[Lemma A.1]{Clement_Sweers1987}, we find a solution $\tilde{w}$ of \eqref{5} satisfies $0\leq \tilde{w}(x)\leq \mathbf{\Psi}(x)^{\frac1p}$. Hence $\tilde{w}$ solves \eqref{3}.

A direct computation yields that the positive solution of \eqref{3} is unique. Hence $\underline{u}_\varepsilon=\tilde{w}$ and it satisfies
\begin{equation}\label{e97}
0\leq \underline{u}_\varepsilon(x)\leq \mathbf{\Psi}(x)^{\frac1p}.
\end{equation}

To investigate  the asymptotic behaviors of $\underline{u}_\varepsilon$, we consider the following minimization  problem
\begin{equation}\label{6}
\inf\{I(u)\;:\;u-c\in H^1_0(B_1(0))\},
\end{equation}
where $c\ge 0$ is a constant,
\[I(u)=\frac{\varepsilon^2}2\int_{B_1(0)} \langle A(x)\nabla u, \nabla u\rangle dx-\int_{B_1(0)} \bar{H}(u)dx, \quad \mathrm{and} \quad \bar{H}(t)=\int_0^t \bar{h}(\tau)d\tau.\]
Here $\bar{h}(t)$ is a non-increasing function satisfies $\bar{h}(t)>0$ for $t\in(-\infty, a)$ and $\bar{h}(t)\leq0$ for $t>a$, where $a\ge 0$ is a constant.

\begin{lemma}\label{lm1}
If $a\ge c$ and $u_\varepsilon$ is the minimizer of problem \eqref{6},  $u_\varepsilon(x)\to a$ on compact sets in $B_1(0)$, as $\varepsilon \to 0^+$.
\end{lemma}
\begin{proof}
To prove this lemma, we only need to modify the argument in  \cite[Lemma 2.2]{Li_Yang_Yan2004}. Since $\bar{H}(t)<0$ for $t<0$, it follows that $u_\varepsilon\geq 0$. Otherwise, $u_\varepsilon^+=\max\{u_\varepsilon,0\}$ would have  lower energy.

It is apparent that $u_\varepsilon$ is a solution to the following boundary value problem
\begin{equation}\label{7}
\left\{\begin{array}{ll}
-\varepsilon^2 \Div(A(x)\nabla u_\varepsilon)=\bar{h}(u_\varepsilon), &\mbox{in $B_1(0)$,} \\
u_\varepsilon=c, &\mbox{on $\partial B_1(0)$.}
\end{array}\right.
\end{equation}
Multiplying the both side of \eqref{7} by $(u_\varepsilon-a)_+$ and integrating by part, we obtain $u_\varepsilon\leq a$. Consequently, $0\leq u_\varepsilon \leq a$, which implies  $\bar{h}(u_\varepsilon)\geq 0$. By virtue of elliptic estimates, we have  $u_\varepsilon \in C^{1, \gamma}(B_1(0))$ for some $\gamma \in (0,1)$.

Employing the same method as above, together with \eqref{7} and the non-increasing property of $\bar{h}$, we derive
\begin{equation}\label{e96}
u_{\varepsilon_1}(x)\leq u_{\varepsilon_2}(x)\quad \mathrm{for} \quad 0<\varepsilon_2<\varepsilon_1,\quad \mathrm{and} \quad  x\in B_1(0).
\end{equation}

For any fixed point $x_0\in B_1(0)$, we define $\tilde{v}_\varepsilon(x)=u_\varepsilon(x_0+\varepsilon x)$, which satisfies the following equation:
\begin{equation*}
\left\{
\begin{array}{ll}
-\Div(A(x_0+\varepsilon x)\nabla \tilde{v}_\varepsilon)=\bar{h}(\tilde{v}_\varepsilon), &\mbox{in $D_1$,} \\
\tilde{v}_\varepsilon(x)=c, &\mbox{on $\partial D_1$.}
\end{array}
\right.
\end{equation*}
where $D_1=\left(B_1(0)-x_0\right)/\varepsilon$. From elliptic estimate, we get $\tilde{v}_\varepsilon$ converges to a $C^{1,\gamma}$ function $\tilde{v}$ on compact sets. The limit function $\tilde{v}(x)$ satisfies
\begin{equation}\label{20}
-\Div(A(x_0)\nabla \tilde{v})=\bar{h}(\tilde{v}),\qquad \mathrm{in} \qquad \R^n.
\end{equation}
For any constant $b>1$, $\varepsilon/b<\varepsilon$. From \eqref{e96}, we get
\[\tilde{v}_\varepsilon(x)=u_\varepsilon(x_0+\varepsilon x)\le u_{\varepsilon/b}(x_0+\varepsilon x)=\tilde{v}_{\varepsilon/b}(bx).\]
Letting $\varepsilon \to 0^+$, we obtain $\tilde{v}(x)\le \tilde{v}(bx)$. Thus, the minimum of $\tilde{v}$ is attained at $0$. Applying the maximum principle, we conclude that $\tilde{v}$ is a constant function. Consequently, we get $\tilde{v}(x)\equiv a$ from \eqref{20}. Therefore, $u_\varepsilon(x_0)=\tilde{v}_\varepsilon(0)\to a$ as $\varepsilon \to 0^+$. Since $x_0$ is arbitrarily chosen in $B_1(0)$, we get  $u_\varepsilon(x)\to a$ as $\varepsilon \to 0^+$ for all $x\in B_1(0)$.

For any compact set $K\subset B_1(0)$, the maximum principle implies
\begin{equation}\label{15}
\min_{\partial K} u_\varepsilon \leq u_\varepsilon(x)\leq a,\qquad \mathrm{for}\qquad \forall x\in K.
\end{equation}

Next, we prove that $\min_{\partial K} u_\varepsilon\to a$ as $\varepsilon \to 0^+$. There is $x_\varepsilon \in \partial K$ such that $u_\varepsilon(x_\varepsilon)=\min_{\partial K} u_\varepsilon$. Hence $x_\varepsilon \to z_0\in \partial K$. Then for $0<\varepsilon<\varepsilon_1$, we have $u_{\varepsilon_1}(x_\varepsilon)\leq u_{\varepsilon}(x_\varepsilon)$. Hence
\[u_{\varepsilon_1}(z_0)=\lim_{\varepsilon\to 0^+} u_{\varepsilon_1}(x_\varepsilon)\leq \lim_{\varepsilon\to 0^+}u_\varepsilon(x_\varepsilon)\leq a.\]
Letting $\varepsilon_1\to 0^+$,  we obtain $\min_{\partial K} u_\varepsilon \to a$ as $\varepsilon \to 0^+$, as desired.

Combining this result with \eqref{15}, we conclude that the lemma holds.

\end{proof}

Next, we consider the following minimization problem
\begin{equation}\label{9}
\inf\left\{\bar{J}_\varepsilon(u, D)=\frac{\varepsilon^2}2\int_D \langle A(x)\nabla u, \nabla u \rangle-\int_D \bar{G}(x,u)\;:\;u-\eta\in H^1_0(D)\right\},
\end{equation}
where $D\subset\Omega$, $\eta\in H^1(D)$ and $\bar{G}(x,t)=\int_0^t \bar{g}(x, \tau)d\tau$. By repeating the argument presented in \cite[Lemma 2.3]{Dancer_Yan2004}, we obtain the following lemma:
\begin{lemma}\label{lm2}
Let $u_{\varepsilon i}$ be the minimizer of \eqref{9} corresponding to $\bar{g}=g_i$ and $\eta=\eta_i$, $i=1,2$. If $\eta_1\geq \eta_2$ and
\[g_1(x,t)\geq g_2(x,t),\quad \mathrm{for} \quad x\in D \quad \mathrm{and} \quad t\in \left[\min_{i=1,2}\min_{z\in D}u_{\varepsilon i}(z), \max_{i=1,2}\max_{z\in D}u_{\varepsilon i}(z)\right], \]
We obtain $u_{\varepsilon1}\geq u_{\varepsilon2}$ in $D$.
\end{lemma}

With the help of Lemma \ref{lm1}, Lemma \ref{lm2} and \eqref{e97}, we conclude that $\underline{u}_\varepsilon$ converges to $\mathbf{\Psi}^{\frac1p}$ on compact sets in $\Omega$ as $\varepsilon\to 0^+$, following the same argument  \cite[Lemma 2.3]{Li_Yang_Yan2004}.

Note that  $\mathbf{\Psi}^{\frac{p-1}p}$ is a positive continuous function on $\Omega$ and
\[f(x):=\frac{\Div(A(x)\nabla \mathbf{\Psi}^{\frac1p})}{p\mathbf{\Psi}^{\frac{p-1}p}}\]
is a negative valued continuous function.

Fix an arbitrary point $x_0\in \Omega$. For any $\eta>0$ sufficient small, there exist $\eta_1>0$ such that $(-f(x_0)+\eta)\eta_1<\eta$. So there exists a constant $\delta>0$ such that  $B_\delta(x_0)\subset \Omega$, and  for all $x$ satisfying $|x-x_0|<\delta$, the following inequalities hold:
\[\mathbf{\Psi}^{\frac{p-1}p}(x)>\frac12\mathbf{\Psi}^{\frac{p-1}p}(x_0),\quad |f(x)-f(x_0)|<\eta, \quad\mathrm{and} \quad \left|\mathbf{\Psi}^{\frac{p-1}p}(x)-\mathbf{\Psi}^{\frac{p-1}p}(x_0)\right|<\frac14\mathbf{\Psi}^{\frac{p-1}p}(x_0)\eta_1.\]

Since $\underline{u}_\varepsilon$ is the minimizer of the functional $J_\varepsilon$ on $H_0^1(\Omega)$, we have
\begin{equation*}
  \frac{\varepsilon^2}2\int_{B_\delta(x_0)}\langle A(x)\nabla \underline{u}_\varepsilon, \nabla \underline{u}_\varepsilon\rangle dx-\int_{B_\delta(x_0)} H(x,\underline{u}_\varepsilon)dx \leq \frac{\varepsilon^2}2\int_{B_\delta(x_0)}\langle A(x)\nabla \bar{w}, \nabla \bar{w}\rangle dx-\int_{B_\delta(x_0)} H(x,\bar{w})dx,
\end{equation*}
where $\bar{w}-\underline{u}_\varepsilon\in H_0^1(B_\delta(x_0))$.

Let $v_\varepsilon=\underline{u}_\varepsilon-\mathbf{\Psi}^{\frac1p}$. Then $-M_3\leq -(\mathbf{\Psi}(x))^{\frac1p}\leq v_\varepsilon(x) \leq 0$ and the following estimate holds:
\begin{eqnarray*}
&& \frac{\varepsilon^2}2\int_{B_\delta(x_0)}\langle A(x)\nabla v_\varepsilon, \nabla v_\varepsilon\rangle dx-\int_{B_\delta(x_0)} \left[H(x,v_\varepsilon+\mathbf{\Psi}^{\frac1p})+\varepsilon^2\Div (A(x)\nabla \mathbf{\Psi}^{\frac1p})v_\varepsilon \right]dx \\
&\leq& \frac{\varepsilon^2}2\int_{B_\delta(x_0)}\langle A(x)\nabla \bar{w}, \nabla \bar{w}\rangle dx-\int_{B_\delta(x_0)}\left[H(x,\bar{w}+\mathbf{\Psi}^{\frac1p})+\varepsilon^2\Div (A(x)\nabla \mathbf{\Psi}^{\frac1p})\bar{w}\right] dx,
\end{eqnarray*}
where $\bar{w}-v_\varepsilon\in H_0^1(B_\delta(x_0))$. Consequently, $v_\varepsilon$ is the minimizer of the following problem
\begin{equation}\label{13}
  \inf\left\{\frac{\varepsilon^2}2\int_{B_\delta(x_0)}\langle A(x)\nabla \bar{w}, \nabla \bar{w}\rangle dx-\int_{B_\delta(x_0)}G(x,\bar{w}) dx;\;\; \bar{w}-v_\varepsilon\in H_0^1(B_\delta(x_0))\right\},
\end{equation}
where
\[G(x,t)=H(x,t+\mathbf{\Psi}^{\frac1p})+\varepsilon^2\Div (A(x)\nabla \mathbf{\Psi}^{\frac1p})t.\]

It is apparent that
\[h(x, t+\mathbf{\Psi}^{\frac1p})=\left\{
\begin{array}{ll}
\mathbf{\Psi}(x)-|t+\mathbf{\Psi}(x)^{\frac1p}|^p, & t>-\mathbf{\Psi}(x)^{\frac1p}, \\
\mathbf{\Psi}(x), &t\leq -\mathbf{\Psi}(x)^{\frac1p}.
\end{array}\right.\]
It is straightforward to verify that for  $t>-\mathbf{\Psi}(x)^{\frac1p}$, the following estimates hold:
\[\mathbf{\Psi}(x)-\left|t+\mathbf{\Psi}(x)^{\frac1p}\right|^p\leq -p\mathbf{\Psi}(x)^{\frac{p-1}p}t, \]
and
\[\mathbf{\Psi}(x)-\left|t+\mathbf{\Psi}(x)^{\frac1p}\right|^p\geq -p\mathbf{\Psi}(x)^{\frac{p-1}p}t-M_4t^2-M_5|t|^p,\]
where $M_4$ and $M_5$ are positive constant large enough.
 In particular, for $-\mathbf{\Psi}(x)^{\frac1p}<t<0$, we have
\[\mathbf{\Psi}(x)-\left|t+\mathbf{\Psi}(x)^{\frac1p}\right|^p\geq -p\mathbf{\Psi}(x)^{\frac{p-1}p}t-\tilde{M}|t|^{\min\{2,p\}},\]
where $\tilde{M}$ is a positive constant.

Let $\tilde{f}(t)=-pb^{\frac{p-1}p}t-\tilde{M}|t|^{\min\{2,p\}}$, and let $t_0$ denote the largest negative maximum point of $\tilde{f}(t)$. We define the functions $h_1(t)$ and $h_2(t)$ as follows:
\[h_1(t)=\left\{\begin{array}{ll}
pB, &\mbox{if $t<-B^{\frac1p}$}, \\
-pB^{\frac{p-1}p}t, &\mbox{if $-B^{\frac1p}\leq t<0$}, \\
-pb^{\frac{p-1}p}t, &\mbox{if $t\geq 0$},
\end{array}\right.
\]
and
\[
h_2(t)=\left\{\begin{array}{ll}
\tilde{f}(t_0), &\mbox{if $t<t_0$}, \\
\tilde{f}(t), &\mbox{if $t_0\leq t<0$}, \\
-pB^{\frac{p-1}p}t-M_4t^2-M_5|t|^p, &\mbox{if $t\ge 0$},
\end{array}\right.
\]
where
\[B=\sup_{x\in B_\delta(x_0)}\mathbf{\Psi}(x), \quad \mathrm{and} \quad b=\inf_{x\in B_\delta(x_0)}\mathbf{\Psi}(x).\]
It is apparent that
\begin{equation}\label{e99}
h_2(t)\leq h(x, t+\mathbf{\Psi}^{\frac1p}(x))\leq h_1(t).
\end{equation}

Then for any $x,y\in B_\delta(x_0)$, we have
\[\left|\frac{\mathbf{\Psi}^{\frac{p-1}p}(x)}{\mathbf{\Psi}^{\frac{p-1}p}(y)}-1\right|<\eta_1,\]
and
\begin{equation}\label{e100}
X_1<\Div(A(x)\nabla \mathbf{\Psi}^{\frac1p})<Y_1,
\end{equation}
where $X_1=-pB^{\frac{p-1}p}\eta+pB^{\frac{p-1}p}f(x_0)$,  $Y_1=pB^{\frac{p-1}p}\eta+pb^{\frac{p-1}p}f(x_0)$.

To estimate $v_\varepsilon$ which is the minimizer of \eqref{13}, we investigate the following minimization problem
\begin{equation}\label{14}
\min\left\{\frac{\varepsilon^2}2\int_{B_\delta(x_0)}\langle A(x)\nabla w, \nabla w \rangle dx-\int_{B_\delta(x_0)}H_i(w)+\varepsilon^2 a_0 w\;:\; w-c_0\in H_0^1(B_\delta(x_0))\right\},
\end{equation}
where $i=1,2$, $c_0\leq 0$, $a_0\leq 0$ and $H_i(t)=\int_0^th_i(\tau)d\tau$. Let $w^\varepsilon_{i,a_0,c_0}$ be the minimizer of \eqref{14} and it solves the following problem
\begin{equation*}
\left\{\begin{array}{ll}
-\varepsilon^2 \Div(A(x)\nabla w)=\tilde{g}_i(w), &\mbox{in $B_\delta(x_0)$}, \\
w=c_0,  &\mbox{on $\partial B_\delta(x_0)$},
\end{array}\right.
\end{equation*}
where $\tilde{g}_i(t):=h_i(t)+\varepsilon^2 a_0$. It is apparent that $\tilde{g}_i(t)$ has a falling zero which we denote by $t_{i,\varepsilon}$. In fact,
\[t_{1,\varepsilon}=\varepsilon^2\frac{a_0}{pB^{\frac{p-1}p}}, \quad \mathrm{and} \quad t_{2,\varepsilon}=\varepsilon^2\frac{a_0}{pb^{\frac{p-1}p}}+o(\varepsilon^2).\]

From the similar method as in Lemma \ref{lm1}, we get $w^\varepsilon_{i,a_0,c_0}$ converges to $0$ uniformly on compact sets in $\Omega$ as $\varepsilon \to 0^+$. We also have $c_0< w^\varepsilon_{i,a_0,c_0}< t_{i,\varepsilon}$ for $c_0<0$ and $t_{i,\varepsilon}<w^\varepsilon_{i,a_0,c_0}<0$ for $c_0=0$. We only consider the case $c_0<0$ for simplicity, since similar conclusions below also hold in the case $c_0=0$ via the same method.

Let $\psi_\varepsilon=-\varepsilon \log(t_{i,\varepsilon}-w^\varepsilon_{i,a_0,c_0})$. Then it solves the following problem
\begin{equation}\label{17}
\left\{\begin{array}{ll}
\varepsilon \Div(A(x)\nabla \psi_\varepsilon)-\langle A(x)\nabla \psi_\varepsilon, \nabla\psi_\varepsilon\rangle+ e^{\psi_\varepsilon/\varepsilon}g(t_{i,\varepsilon}-e^{-\psi_\varepsilon/\varepsilon})=0, &\mbox{in $B_\delta(x_0)$} \\
\psi_\varepsilon=-\varepsilon \log(t_{i,\varepsilon}-c_0), &\mbox{on $\partial B_\delta(x_0)$}.
\end{array}\right.
\end{equation}
In order to estimate the solutions of \eqref{17}, we consider the following auxiliary problem as in \cite[Lemma 4.2]{Ni_Wei1995}:
\begin{equation}\label{10}
\left\{\begin{array}{ll}
\varepsilon \Div(A(x)\nabla \psi)-\langle A(x)\nabla \psi, \nabla \psi\rangle +1=0, &\mbox{in $B_\delta(x_0)$} \\
\psi=0, &\mbox{on $\partial B_\delta(x_0)$}.
\end{array}\right.
\end{equation}

\begin{lemma}\label{lm3}
For $\varepsilon>0$ small enough, \eqref{10} admits a unique solution $\psi^\varepsilon$, which satisfies $\|\psi^\varepsilon\|_{L^\infty(\Omega)}\leq C_1$ with $C_1>0$ be a constant independent of $\varepsilon$. Moreover, $\psi^\varepsilon$ also satisfies the following estimate
\begin{equation}\label{16}
\mu d(x, \partial B_\delta(x_0))\leq \psi^\varepsilon(x) \leq \rho  d(x, \partial B_\delta(x_0)),
\end{equation}
where $\mu$ and $\rho$ are positive constants.
\end{lemma}

\begin{proof}
It is obvious that $0$ is a subsolution of \eqref{10}. With the help of \eqref{e98},  we fix a vector $X_0$ such that $\langle A(x)X_0,X_0\rangle>2$ for any $x\in \R^n$. Choose positive constant $b$ large enough, such that $g(x)=\langle x, X_0\rangle +b>0$ on $\partial B_\delta (x_0)$. For $\varepsilon>0$ small enough, we get
\begin{equation*}
\left\{\begin{array}{ll}
\varepsilon \Div(A(x)\nabla g)-\langle A(x)\nabla g, \nabla g\rangle +1<0, &\mbox{in $B_\delta(x_0)$},  \\
g>0, &\mbox{on $\partial B_\delta(x_0)$}.
\end{array}\right.
\end{equation*}
From \cite[Theorem 1]{Amann_Crandall1978}, we get a solution $\psi^\varepsilon$ of \eqref{10} satisfying $0<\psi^\varepsilon<g(x)$. Moreover, the solution of \eqref{10} is unique via maximum principal. Then we have $\|\psi^\varepsilon\|_{L^\infty(B_\delta(x_0))}\leq C_1$.

Now we prove the estimate \eqref{16}. Let $d(x)=d(x, \partial B_\delta(x_0))$. For $x\not=x_0$, $d(x)$ is a $C^2$ function. We define $\psi_\varepsilon^+(x)=\rho d(x)$ where $\rho$ is a positive constant large enough, so that $\psi_\varepsilon^+(x)$ is a supersolution of \eqref{10}.

Let $\psi_\varepsilon^-(x)=\mu d(x)$,  where $\mu$ is a positive small enough so that $\psi_\varepsilon^-(x)$ is a subsolution of \eqref{10}. Hence the estimate \eqref{16} holds.
\end{proof}

From the similar argument of \cite[Theorem 2.1]{Dancer_Wei1997}, we obtain the following lemma.

\begin{lemma}\label{lm4}
The solution $\psi_\varepsilon$ of \eqref{17} has the following estimate
\[\mu \nu_0 d(x, \partial B_\delta (x_0))\leq \psi_\varepsilon(x) \leq \rho \nu_0 d(x, \partial B_\delta (x_0)),\]
where $\nu_0=\sqrt{-\tilde{g}_i'(t_{i,\varepsilon})}$.
\end{lemma}
\begin{proof}
Let $\bar{\delta}\in (0,\delta)$ be a constant sufficient near $\delta$. Then $t_{i,\varepsilon}-w^\varepsilon_{i,a_0,c_0}$ converges to $0$ on $B_{\bar\delta}(x_0)$. For any $\eta>0$, we have $t_{i,\varepsilon}-w^\varepsilon_{i,a_0,c_0}<\eta$ on $B_{\bar\delta}(x_0)$ for small $\varepsilon>0$. Let $w_\varepsilon^+$ be the unique solution of the following problem
\begin{equation}\label{47}
\left\{\begin{array}{ll}
\varepsilon \Div(A(x)\nabla w_\varepsilon^+)-\langle A(x)\nabla w_\varepsilon^+, \nabla w_\varepsilon^+\rangle +\tau=0, &\mbox{in $B_{\bar\delta}(x_0)$}, \\
w_\varepsilon^+=0, &\mbox{on $\partial B_{\bar\delta}(x_0)$},
\end{array}\right.
\end{equation}
where
\[\tau=\min_{t_{i,\varepsilon}-\eta<s<t_{i,\varepsilon}}(-\tilde{g}_i'(s)), \quad \mathrm{and} \quad \tilde\tau=\max_{t_{i,\varepsilon}-\eta<s<t_{i,\varepsilon}}(-\tilde{g}_i'(s)).\]
It is apparent that $\psi_\varepsilon$ is a supersolution of \eqref{47}. Then we have
\begin{equation}\label{19}
\psi_\varepsilon(x) \geq w_\varepsilon^+(x)\geq \sqrt{\tau}\mu d(x, \partial B_{\bar{\delta}}(x_0)), \quad \mathrm{where} \quad x\in  B_{\bar{\delta}}(x_0)).
\end{equation}

Now we construct a supersolution of \eqref{17}. Define
\[T=\{x\in B_\delta(x_0): w^\varepsilon_{i,a_0,c_0}> t_{i,\varepsilon}-\eta \}.\]
Since $t_{i,\varepsilon}-w^\varepsilon_{i,a_0,c_0}$ converges to $0$ on compact sets, we realize that given any compact set $K\subset B_{\delta}(x_0)$, we have $K\subset T$ for $\varepsilon$ small enough. Let $w_\varepsilon^-$ be the unique solution of the following problem
\begin{equation*}
\left\{\begin{array}{ll}
\varepsilon \Div(A(x)\nabla w_\varepsilon^-)-\langle A(x)\nabla w_\varepsilon^-, \nabla w_\varepsilon^-\rangle +\tilde{\tau}=0, &\mbox{in $B_\delta(x_0)$}, \\
w_\varepsilon^-=\tilde{e}, &\mbox{on $\partial B_\delta(x_0)$},
\end{array}\right.
\end{equation*}
where $\tilde{e}$ is a fixed constant some enough. From Lemma \ref{lm3}, we get
\[w_\varepsilon^-(x)\geq \sqrt{\tilde{\tau}}\mu d(x, \partial B_\delta(x_0))+\tilde{e}.\]
It is easy to see that $w_\varepsilon^-(x)$ is a supersolution of \eqref{17} on $T$. However, on $\partial T$, we have $\psi_\varepsilon(x)=-\varepsilon \log \eta\leq \tilde{e}/2\leq w_\varepsilon^-(x)$. Hence the following estimate holds on $T$:
\begin{equation}\label{18}
\psi_\varepsilon(x)\leq w_\varepsilon^-(x)\leq \sqrt{\tilde\tau}\rho d(x, \partial B_\delta(x_0))+\tilde{e}.
\end{equation}
On $B_\delta(x_0)\char92 T$, we also get the estimate above from Lemma \ref{lm3} and  direct computation.

Let $\bar\delta\to \delta$, $\eta\to 0$ and $\tilde{e}\to 0$, we get this lemma from \eqref{19} and \eqref{18}.
\end{proof}

With the help of Lemma \ref{lm4}, we get that  we get $w^\varepsilon_{i,a_0,c_0}=t_i+o(\varepsilon^2)$ in $B_{\frac{\delta}2}(x_0)$. That is
\[w^\varepsilon_{1,a_0,c_0}=\varepsilon^2\frac{a_0}{pB^{\frac{p-1}p}}+o(\varepsilon^2), \quad \mathrm{and} \quad w^\varepsilon_{2,a_0,c_0}=\varepsilon^2\frac{a_0}{pb^{\frac{p-1}p}}+o(\varepsilon^2).\]
From \eqref{e99}, \eqref{e100} and  Lemma \ref{lm2}, we also get
\[w^\varepsilon_{2, X_1, -M}\leq v_\varepsilon\leq w^\varepsilon_{1, Y_1,0}\quad \mathrm{in} \quad B_{\delta}(x_0).\]
For $\varepsilon>0$ small enough,
\[\frac{X_1}{pb^{\frac{p-1}p}}-\eta\leq \frac{w^\varepsilon_{1, X_1, -M}}{\varepsilon^2}\leq \frac{v_\varepsilon}{\varepsilon^2}\leq \frac{w^\varepsilon_{2,Y_1,0}}{\varepsilon^2}\leq \frac{Y_1}{pB^{\frac{p-1}p}}+\eta.\]
In fact,
\[\frac{Y_1}{pB^{\frac{p-1}p}}+\eta\leq f(x_0)+2\eta-f(x_0)\eta_1\leq f(x)+4\eta,\]
and
\[\frac{X_1}{pb^{\frac{p-1}p}}-\eta\geq f(x_0)-2\eta+(f(x_0)-\eta)\eta_1\geq f(x)-4\eta.\]
Hence we have
\[\left|\frac{v_\varepsilon(x)}{\varepsilon^2}-f(x)\right|\leq 4\eta, \quad \mathrm{where} \quad x\in B_{\frac{\delta}2}(x_0). \]

Let $K$ be a compact subset in $\Omega$. For any $\eta>0$, we cover $K$ by a finite number of balls $B_{\frac{\delta}2}(x_0)$, $x_0\in K$. Using the relationship above, we get for $\varepsilon>0$ small enough,
\[\left|\frac{v_\varepsilon(x)}{\varepsilon^2}-f(x)\right|\leq 4\eta, \quad \mathrm{where} \quad x\in K. \]
Then we get
\[\underline{u}_\varepsilon(x)=\mathbf{\Psi}^{\frac1p}(x)+\varepsilon^2\left(\frac{\Div(A(x)\nabla \mathbf{\Psi}^{\frac1p})}{p\mathbf{\Psi}^{\frac{p-1}p}}+o(1)\right), \qquad x\in K.\]
Then $\bar{u}_\varepsilon:=-\underline{u}_\varepsilon$ is the unique negative solution of \eqref{1} satisfying \eqref{e101}. Hence  Proposition \ref{pro1} follows.

\section{Projections of the error}\label{sect4}
Recall $S(\mathcal{V})$ is expanded in \eqref{74}. In this section, we expand the terms
\[\int_\R \eta_\delta^\varepsilon S(\mathcal{V})w_{j,x}dx\qquad  \mathrm{and}  \qquad \int_\R \eta_\delta^\varepsilon S(\mathcal{V})w_{j,x}dx.\]

We first consider the term $\int_\R \eta_\delta^\varepsilon S(\mathcal{V})w_{j,x}dx$. It is easy to get
\begin{equation*}
\int_\R \sum_{k=1}^N \varepsilon (\lambda_0 e_k +\varepsilon^2 a_{11}\alpha^{1-p}e_k'')Z_k w_{j,x}=O(\varepsilon^{\frac52+\mu_1}\sum_{k\not=j}(|e_k|+\varepsilon^2|e_k''|)).
\end{equation*}
From \eqref{75}, we get
\begin{equation*}
\sum_{k\not=j}\varepsilon^2\int_\R \hat{\mathbf{A}}_k(z,x-\beta f_k)w_{j,x}dx=O(\varepsilon^3)\sum_{k\not=j}[|f_k|+|f_k'|+|f_k''|+\varepsilon |e_k'|(|f_k|+|f_k'|)].
\end{equation*}
However
\begin{eqnarray*}
&&\int_\R p\alpha_p \chi_{\mathcal{U}_j}[|w_j-1|^{p-2}(w_j-1)+1]\left(e^{-\sqrt{p}(x-\beta f_{j-1})}+e^{\sqrt{p}(x-\beta f_{j+1})}\right)w_{j,x}dx \\
&=& \int_{-\frac{\beta}2(f_j-f_{j-1})}^{\frac{\beta}2(f_{j+1}-f_j)}p\alpha_p [|w-1|^{p-2}(w-1)+1]\left[e^{-\sqrt{p}x}e^{-\sqrt{p}\beta (f_j-f_{j-1})}+e^{\sqrt{p}x}e^{-\sqrt{p}\beta (f_{j+1}-f_j)}\right]w_x(x)dx \\
&=&p\alpha_p C_0 \left[e^{-\sqrt{p}\beta (f_j-f_{j-1})}-e^{-\sqrt{p}\beta(f_{j+1}-f_j)}\right] +O(\varepsilon^{3-\mu}),
\end{eqnarray*}
where
\begin{equation}\label{e114}
C_0=\frac12\int_\R [|w-1|^{p-2}(w-1)+1](e^{-\sqrt{p}x}-e^{\sqrt{p}x})w_xdx>0
\end{equation}
from Lemma \ref{lm9}.
Along the same lines, we get
\begin{equation*}
\sum_{k\not=j} p\alpha_p\int_\R \chi_{\mathcal{U}_k}[|w_k-1|^{p-2}(w_k-1)+1]\left[e^{-\sqrt{p}(x-\beta f_{k-1})}+e^{\sqrt{p}(x-\beta f_{k+1})}\right]w_{j,x}dx =O(\varepsilon^{\frac52+\mu_1}),
\end{equation*}

We first derive some identities. Taking derivatives of both sides of \eqref{34}, we have
\begin{eqnarray}\label{76}
&&-w_{0,xxx}-p|w-1|^{p-2}(w-1)w_{0,x}-p(p-1)|w-1|^{p-2}w_0w_x \\
\nonumber&=&w_{xx}+\frac{2p}{p+3}[|w-1|^{p-2}(w-1)+1]+\frac{2p(p-1)}{p+3}|w-1|^{p-2}xw_x.
\end{eqnarray}
Multiplying the both side of \eqref{76} by $w_2$ and $w_3$, respectively and integrating,  we get the following identities from \eqref{35} and \eqref{36}:
\begin{eqnarray}\label{e58}
&&\int_\R [w_{0,xx}+p(p-1)|w-1|^{p-2}w_0w_2]w_x dx \\
\nonumber&=&\int_\R w_{2,x}w_xdx-\frac{2p}{p+3}\int_\R [|w-1|^{p-2}(w-1)+1]w_2dx-\frac{2p(p-1)}{p+3}\int_\R |w-1|^{p-2}xw_2 w_xdx
\end{eqnarray}
and
\begin{eqnarray}\label{e59}
&&(p-1)\int_\R |w-1|^{p-2}w_0w_xdx+p(p-1)\int_\R |w-1|^{p-2}w_0w_3 w_x dx \\
\nonumber&=& \int_\R w_{3,x}w_x dx-\frac{2p}{p+3}\int_\R [|w-1|^{p-2}(w-1)+1]w_3dx-\frac{2p(p-1)}{p+3}\int_\R |w-1|^{p-2}xw_3 w_xdx.
\end{eqnarray}
From \eqref{37}, \eqref{35} and \eqref{36}, we get the following identities via the same method as above:
\begin{eqnarray}\label{e60}
\nonumber&&\int_\R [w_{1,xx}+p(p-1)|w-1|^{p-2}w_1w_2]w_x dx \\
\nonumber&=&-\int_\R w_{2,x}w_xdx-\int_\R xw_{2,xx}w_x dx+\frac{p}{p+3}\int_\R [|w-1|^{p-2}(w-1)+1]w_2dx \\
&&+\frac{p(p-1)}{p+3}\int_\R |w-1|^{p-2}xw_2w_xdx,
\end{eqnarray}
and
\begin{eqnarray}\label{e61}
&&\int_\R [xw_{3,xx}+p(p-1)|w-1|^{p-2}w_1w_3]w_xdx +(p-1)\int_\R |w-1|^{p-2}w_1w_xdx \\
\nonumber&=& -\int_\R w_{3,x}w_xdx +\frac{p}{p+3}\int_\R [|w-1|^{p-2}(w-1)+1]w_3dx+\frac{p(p-1)}{p+3}\int_\R |w-1|^{p-2}xw_3w_xdx.
\end{eqnarray}

By virtue of the condition \eqref{e57} and direct computation, we get
\begin{eqnarray*}
&&\int_\R \hat{\mathbf{A}}_j(z,x-\beta f_j)w_{j,x}dx \\
&=&\int_\R w_x^2dx \left\{-a_{11}\alpha^{1-p}(\beta f_j)'' -2a_{11}\alpha^{-p}\alpha'(\beta f_j)' +a_{11}\alpha^{1-p}\beta^{-1}\beta'(\beta f_j)'+a_{22}\alpha^{1-p}(\beta f_j)'\right. \\
&&-b_{11}\alpha^{1-p}(\beta f_j)'+2a_{11}\alpha^{-p}\alpha'\beta'f_j-a_{11}\alpha^{1-p}\beta^{-1}(\beta')^2f_j+a_{11}\alpha^{1-p}\beta''f_j+2a_{22}\alpha^{-p}\alpha'\beta f_j \\
&&\left. +\frac{p+3}2\beta^{-1}\alpha^{-1}\mathbf{q}_{tt}f_j-a_{33}\alpha^{1-p}\beta f_j -2a_{22}\alpha^{1-p}\beta'f_j +b_{11}\alpha^{1-p}\beta' f_j+b_{22}\alpha^{1-p}\beta f_j+2a_{22}\alpha^{1-p}\beta' f_j\right\} \\
&&+\int_\R Z_xw_x dx\left\{-2\varepsilon a_{11}\alpha^{1-p}(\beta f_j)'e_j +2\varepsilon a_{11}\alpha^{1-p}\beta'f_j e_j' +2\varepsilon a_{22}\alpha^{1-p}\beta f_j e_j'\right\} \\
&&+\alpha^{2-2p}\beta^3(a_{32})^2f_j\left\{\int_\R \left[w_{1,xx}+p(p-1)|w-1|^{p-2}w_1w_2+xw_{2,xx}-\frac{p(p-1)}{p+3}|w-1|^{p-2}xw_2\right]w_xdx\right.  \\
&&-\frac{p}{p+3}\int_\R \left[xw_{3,xx}+p(p-1)|w-1|^{p-2}w_1w_3-\frac{p(p-1)}{p+3}|w-1|^{p-2}xw_3+(p-1)|w-1|^{p-2}w_1\right]w_xdx \\
&&\left. +\frac{p(p-1)}{(p+3)^2} \int_\R x\left[|w-1|^{p-2}-1\right]w_xdx\right\} \\
&&\alpha^{2-2p}\beta^3 a_{32}b_{21} f_j \left\{\left[w_{0,xx}+p(p-1)|w-1|^{p-2}w_0w_2+w_{2,x} +\frac{2p(p-1)}{p+3}|w-1|^{p-2}xw_2\right]w_x dx\right. \\
&&-\frac{p}{p+3}\int_\R \left[p(p-1)|w-1|^{p-2}w_0w_3+(p-1)|w-1|^{p-2}w_0+\frac{2p(p-1)}{p+3}|w-1|^{p-2}xw_3+w_{3,x}\right]w_xdx  \\
&& +\frac{2p}{p+3}\int_\R \left[xw_{3,xx}+p(p-1)|w-1|^{p-2}w_1w_3+(p-1)|w-1|^{p-2}w_1-\frac{p(p-1)}{p+3}|w-1|^{p-2}xw_3\right]w_x dx  \\
&&\left.-\frac{4p(p-1)}{(p+3)^2}\int_\R x\left[|w-1|^{p-2}-1\right]w_x dx\right\} \\
&& +\alpha^{2-2p}\beta^3(b_{21})^2f_j \left\{\frac{2p}{p+3}\int_\R \left[(p-1)|w-1|^{p-2}w_0+p(p-1)|w-1|^{p-2}w_0w_3+w_{3,x}\right]w_x dx\right. \\
&& \left.+\frac{4p^2(p-1)}{(p+3)^2}\int_\R|w-1|^{p-2}xw_3w_x dx +\frac{4p(p-1)}{(p+3)^2}\int_\R x\left[|w-1|^{p-2}-1\right]w_x dx \right\}.
\end{eqnarray*}
From \eqref{31}, \eqref{32} and direct computation, we have
\[\int_\R w_{2,x}w_x dx=-\frac14\int_\R w_x^2dx, \qquad \int_\R w_{3,x}w_x dx=-\frac{p+3}{4p}\int_\R w_x^2 dx,\]
\[\int_\R \left[|w-1|^{p-2}(w-1)+1\right]w_2 dx=\frac{p+3}{4p}\int_\R w_x^2 dx,\]
\[\int_\R x\left[|w-1|^{p-2}-1\right]w_xdx =-\frac1{p-1}\int_\R \left[|w-1|^{p-2}(w-1)+1\right]dx +\frac{p+3}{2p}\int_\R w_x^2dx, \]
and
\[\int_\R \left[|w-1|^{p-2}(w-1)+1\right]w_3 dx=\frac{(p+3)(p-3)}{4p^2}\int_\R w_x^2 dx-\frac1p\int_\R \left[|w-1|^{p-2}(w-1)+1\right]dx.\]
From these identities, \eqref{e57}, \eqref{31}, \eqref{32}, \eqref{e58}, \eqref{e59}, \eqref{e60} and \eqref{e61}, we get
\begin{eqnarray*}
\int_\R \hat{\mathbf{A}}_j(z,x-\beta f_j)w_{j,x}dx&=& \alpha^{1-p}\beta \int_\R w_x^2 dx\left\{-a_{11}f_j''+\left[a_{22}-b_{11}-a_{11}\left(\beta^{-1}\beta' +2\alpha^{-1}\alpha'\right)\right]f_j'\right. \\
&& \left.+\left[a_{22}\left(\beta^{-1}\beta'+2\alpha^{-1}\alpha'\right) +b_{22}-a_{33}+\frac{p+3}2\alpha^{p-2}\beta^{-2} \mathbf{q}_{tt}\right]f_j \right\} \\
&& +\alpha^{1-p}\beta\left(-2\varepsilon a_{11}e_j' f_j' +2\varepsilon a_{22}e_j' f_j\right)\int_\R Z_x w_x dx \\
&& -\frac{p+1}{p+3}\alpha^{2-2p}\beta^3 a_{32}b_{21} f_j\int_\R w_x^2 dx+\frac{p+2}{2(p+3)}\alpha^{2-2p}\beta^3 (a_{32})^2 f_j\int_\R w_x^2 dx \\
&& -\frac2{p+3}\alpha^{2-2p}\beta^3(b_{21})^2 f_j \int_\R w_x^2 dx.
\end{eqnarray*}
However, from the definitions of $\tilde{\theta}_j$ and $\tilde{\theta}_{j2}$ in \eqref{e62} and \eqref{e63}, we get
\[\sum_{k=1}^N \int_\R \tilde{\theta}_k w_{j,x}=O(\varepsilon^{3-\mu}), \quad \mathrm{and} \quad \sum_{k=1}^N \int_\R \tilde{\theta}_{k2} w_{j,x}=O(\varepsilon^{3-\mu})\]
In summary, we get
\begin{eqnarray}\label{e94}
\nonumber\int_\R \eta_\delta^\varepsilon S(\mathcal{V})w_{j,x}dx&=& \varepsilon^2 \alpha^{1-p}\beta \int_\R w_x^2 dx\left\{ -a_{11}f_j''+\left[a_{22}-b_{11}-a_{11}\left(\beta^{-1}\beta' +2\alpha^{-1}\alpha'\right)\right]f_j' \right. \\
\nonumber&&+\left[a_{22}\left(\beta^{-1}\beta'+2\alpha^{-1}\alpha'\right) +b_{22}-a_{33}+\frac{p+3}2\alpha^{p-2}\beta^{-2} \mathbf{q}_{tt}\right]f_j \\
\nonumber&&\left. + \left[\frac{p+2}{2(p+3)}\alpha^{1-p}\beta^2 (a_{32})^2  -\frac{p+1}{p+3}\alpha^{1-p}\beta^2 a_{32}b_{21}-\frac2{p+3}\alpha^{1-p}\beta^2(b_{21})^2\right]f_j\right\} \\
\nonumber&& +\varepsilon^2\alpha^{1-p}\beta\left(-2\varepsilon a_{11}e_j' f_j' +2\varepsilon a_{22}e_j' f_j\right)\int_\R Z_x w_x dx   \\
&& +p\alpha_pC_0 \left[e^{-\sqrt{p}\beta (f_j-f_{j-1})}-e^{-\sqrt{p}\beta(f_{j+1}-f_j)}\right] +\varepsilon^2 \Theta_j(\varepsilon z)
\end{eqnarray}
where $\|\Theta_j\|_{L^2(0,1)}\leq C\varepsilon^{1-\mu}$.

Now we consider the term $\int_\R \eta_\delta^\varepsilon S(\mathcal{V})Z_j(x)dx$.  It is easy to get
\begin{equation*}
\sum_{k=1}^N \varepsilon (\varepsilon^2 a_{11}\alpha^{1-p}e_k''+\lambda_0 e_k )\int_\R Z_k Z_j dx = \varepsilon (\varepsilon^2 a_{11}\alpha^{1-p}e_j''+\lambda_0 e_j) +O\left(\varepsilon^{3-\mu}\sum_{k\not=j}\left(|e_k|+|e_k''|\right)\right)
\end{equation*}
and
\[\sum_{k=1}^N\varepsilon^2 \int_\R \hat{\mathbf{A}}_k(z, x-\beta f_k)Z_j(x)dx=O\left(\varepsilon^{4-\mu}\sum_{k\not=j}\left[|f_k|+|f_k'|+|f_k''|+\varepsilon |e_k'|\left(|f_k|+|f_k'|\right)\right]\right).\]
However,
\begin{eqnarray*}
&&p\alpha_p \int_\R \chi_{\mathcal{U}_j}\left[|w_k-1|^{p-2}(w_k-1)\right]\left(e^{-\sqrt{p}(x-\beta f_{k-1})}+e^{\sqrt{p}(x-\beta f_{k+1})}\right)Z_j(x)dx \\
&=& p \alpha_p \int_{\frac12\beta(f_{j-1}+f_j)}^{\frac12\beta(f_{j+1}+f_j)} \left[|w_k-1|^{p-2}(w_k-1)\right]\left(e^{-\sqrt{p}(x-\beta f_{k-1})}+e^{\sqrt{p}(x-\beta f_{k+1})}\right)Z_j(x)dx \\
&=& p \alpha_p e^{-\sqrt{p}(f_j-f_{j-1})} \int_{\frac12\beta(f_{j-1}-f_j)}^{\frac12\beta(f_{j+1}-f_j)}\left[|w-1|^{p-2}(w-1)+1\right]e^{-\sqrt{p}x}Z(x)dx \\
&&+ p \alpha_p e^{-\sqrt{p}(f_{j+1}-f_j)} \int_{\frac12\beta(f_{j-1}-f_j)}^{\frac12\beta(f_{j+1}-f_j)}\left[|w-1|^{p-2}(w-1)+1\right]e^{\sqrt{p}x}Z(x)dx \\
&=& p\alpha_p C_1\left[e^{-\sqrt{p}(f_j-f_{j-1})}+ e^{-\sqrt{p}(f_{j+1}-f_j)}\right],
\end{eqnarray*}
where
\[C_1=\frac12\int_\R [|w-1|^{p-2}(w-1)+1]\left(e^{\sqrt{p}x}+e^{-\sqrt{p}x}\right)Z(x)dx.\]
Other terms is estimated by a similar method. We get
\begin{equation}\label{e95}
\int_\R \eta_\delta^\varepsilon S(\mathcal{V})Z_j(x)dx=\varepsilon(\varepsilon^2 a_{11}\alpha^{1-p}e_j'' +\lambda_0 e_j) +p\alpha_p C_1\left[e^{-\sqrt{p}(f_j-f_{j-1})}+ e^{-\sqrt{p}(f_{j+1}-f_j)}\right] +\varepsilon^2\Xi_j(\varepsilon z),
\end{equation}
where $\|\Xi_j\|_{L^2(0,1)}\le C\varepsilon^{1-\mu}$.

\section*{Declarations}
\noindent\textbf{Data availability statement:} No data was used for the research described in the article.

\noindent\textbf{Conflict of Interest:} No potential conflict of interest was reported by the author.

\end{document}